\documentclass[reqno,12pt]{amsart}

\usepackage[T1]{fontenc}
\usepackage{lmodern}
\usepackage[protrusion=true,expansion=false]{microtype}
\usepackage{amsmath,amssymb,amsthm,mathtools,mathrsfs,esint}
\usepackage{aliascnt}
\usepackage{enumitem}
\usepackage{hyperref,cleveref,tikz}

\newtheorem{theorem}{Theorem}[section]

\newaliascnt{lemma}{theorem}
\newtheorem{lemma}[lemma]{Lemma}
\aliascntresetthe{lemma}

\newaliascnt{proposition}{theorem}
\newtheorem{proposition}[proposition]{Proposition}
\aliascntresetthe{proposition}

\newaliascnt{corollary}{theorem}
\newtheorem{corollary}[corollary]{Corollary}
\aliascntresetthe{corollary}

\newaliascnt{claim}{theorem}

\aliascntresetthe{claim}

\theoremstyle{definition}
\newaliascnt{definition}{theorem}
\newtheorem{definition}[definition]{Definition}
\aliascntresetthe{definition}

\theoremstyle{remark}
\newaliascnt{remark}{theorem}
\newtheorem{remark}[remark]{Remark}
\aliascntresetthe{remark}

\newaliascnt{example}{theorem}

\aliascntresetthe{example}

\crefname{theorem}{theorem}{theorems}
\Crefname{theorem}{Theorem}{Theorems}
\crefname{lemma}{lemma}{lemmas}
\Crefname{lemma}{Lemma}{Lemmas}
\crefname{proposition}{proposition}{propositions}
\Crefname{proposition}{Proposition}{Propositions}
\crefname{corollary}{corollary}{corollaries}
\Crefname{corollary}{Corollary}{Corollaries}
\crefname{definition}{definition}{definitions}
\Crefname{definition}{Definition}{Definitions}
\crefname{claim}{claim}{claims}
\Crefname{claim}{Claim}{Claims}
\crefname{remark}{remark}{remarks}
\Crefname{remark}{Remark}{Remarks}
\crefname{example}{example}{examples}
\Crefname{example}{Example}{Examples}

\numberwithin{equation}{section}

\DeclareMathOperator{\dist}{dist}
\DeclareMathOperator{\diam}{diam}
\DeclareMathOperator{\tr}{tr}
\DeclareMathOperator{\sym}{sym}
\DeclareMathOperator{\dev}{dev}

\DeclareMathOperator{\interior}{int}
\DeclareMathOperator{\Mob}{M\ddot ob}

\DeclareMathOperator{\GL}{GL}

\newcommand{\R}{\mathbb R}
\newcommand{\Sn}{\mathbb S^n}
\newcommand{\Id}{I_n}
\newcommand{\Edev}{\mathcal E^{\mathrm D}}
\newcommand{\Kconf}{\mathcal K_n}
\newcommand{\Qzero}{\mathcal Q_0}

\newcommand{\whR}{\widehat{\mathbb R}^{\,n}}

\newcommand{\avnorm}[3]{\left\|#1\right\|_{L^{#2}_{\mathrm{av}}(#3)}}

\title[Liouville stability for quasiregular mappings]{A sharp stability inequality of Liouville's theorem for quasiregular mappings on bounded domains}
\author{Yi Ru-Ya Zhang}
\address{State Key Laboratory of Mathematical Sciences, Academy of Mathematics and Systems Science, Chinese Academy of Sciences, Beijing 100190, China}
\address{Institute of Mathematics, Academy of Mathematics and Systems Science, Chinese Academy of Sciences, Beijing 100190, China}
\email{yzhang@amss.ac.cn}

\thanks{The author was supported by the National Key R\&D Program of China (Grant Nos. 2025YFA1018400 and 2021YFA1003100), the NSFC (Grant Nos. 12288201 and 12571128), the Chinese Academy of Sciences, and the CAS Project for Young Scientists in Basic Research (Grant No. YSBR-031).}
 

\begin{document}

\begin{abstract}
Let $n \ge 3$ and fix $p>n$. We establish a sharp quantitative stability result in $W^{1,p}$ for Liouville's theorem for nonconstant, sense-preserving quasiregular mappings whose outer distortion $K$ is sufficiently close to one. We first reorganize Reshetnyak's classical local argument \cite{R1976}, drawing systematically on modern techniques such as those developed in \cite{FZ2022}.
We then give two globalization schemes on bounded, connected domains. For bounded John domains, instead of Reshetnyak's original construction \cite{R19762}, which glues M\"obius transformations on adjacent Whitney cubes, we use Whitney chains and Bojarski's enlarged-cube estimate to obtain a quantitatively sharp Euclidean stability result.  
Finally, we extend this idea to general bounded, connected open subsets of $\mathbb{R}^n$ via the Lorentzian representation of the M\"obius group, yielding a compactified, weighted stability theorem whose underlying measure is absolutely continuous with respect to Lebesgue measure and has density bounded above by one.
\end{abstract}

\maketitle
\tableofcontents

\section{Introduction}

Let $n\ge3$. Liouville's rigidity theorem states that a nonconstant conformal mapping between domains in $\R^n$ is the restriction of a M\"obius transformation.
Equivalently, if a nonconstant sense-preserving quasiregular mapping has distortion coefficient $K = 1$, then it must be a M\"obius transformation. The corresponding stability problem concerns   whether, and with respect to a suitable norm, a mapping satisfying 
$ K - 1 \ll 1$
is necessarily close to a M\"obius transformation.

The natural analytic class for this question is the class of mappings with bounded distortion in the sense of Reshetnyak. In modern terminology these are nonconstant sense-preserving quasiregular mappings. Quasiconformal mappings are the
homeomorphic quasiregular mappings. We recall the precise definition in \Cref{def:quasiregular}; see also Rickman~\cite[Chapter~I]{R1993} and Gehring--Martin--Palka~\cite[Chapters~2 and~6]{GMP2017}. Reshetnyak proved a
local quantitative stability theorem and a global theorem on the geometric class $J(d,D)$ of domains. In a local form, his result gives a M\"obius transformation
$\varphi$ such that
$$
 g:=\varphi^{-1}\circ f
$$
satisfies the scale-invariant estimate
$$
 \left(\fint_Q |Dg-\Id|^q\,dx\right)^{1/q}
 +\frac1{\ell(Q)}  \left(\fint_Q |g-\operatorname{id}|^q\,dx\right)^{1/q}
 \le C(n,q)\,[K-1]
$$
for every admissible exponent
$$
 n\le q<\frac{C(n)}{K-1}.
$$
This was proven in \cite{R1976}; see also \cite[Chapter~4, Theorem~3.2]{R1994}. The corresponding global theorem on $J(d,D)$-domains was proven in \cite{R19762}; see also \cite[Chapter~4, Theorem~4.1]{R1994}. For a comprehensive theoretical foundation of mappings with bounded distortion, we refer the reader to, for instance,~\cite{R1989,R1993,GMP2017}. 
In addition, we refer the reader to~\cite{FZ2005} for a seminal result concerning the rigidity properties of conformal matrices, to \cite{LZ2022} for stability estimates  in  $W^{1,2}$ for Sobolev maps from $\mathbb S^{n-1}$ to $\mathbb R^n$ which are sharp when $n=3$, and to~\cite{GLZ2025,GLZ2024} for more recent advances on the stability properties of M\"obius transformations acting on spheres. In particular, the stability estimate in \cite{GLZ2024} is given with respect to the $W^{1,n-1}$-norm, which provides an optimal distance across all dimensions.

\subsection{Main results}
The present paper has two aims. The first is expository and local: For a fixed exponent $p>n$, we rewrite both Reshetnyak's local proof and his proof for John domains in a more transparent form. The second is to extend the stability estimate from John domains to arbitrary open sets, with a natural weighted formulation.

Let us first state the main results. The precise definition of quasiregularity is given in \Cref{def:quasiregular} below. For a cube $Q$, let $x_Q$ denote its center and
$\ell(Q)$ its sidelength. We write
$$
 S_Q(x):=\frac{x-x_Q}{\ell(Q)}.
$$
If $M$ is a M\"obius transformation of $\widehat{\R}^{\,n}$, then 
$$
 \widehat M:=\sigma\circ M\circ\sigma^{-1}
$$
denotes its action on $\mathbb S^n$, where $\sigma$ is the inverse stereographic projection defined in \eqref{eq:stereographic}.

The first theorem is a local result for the fixed exponent $p>n$ used throughout the paper. It is weaker than Reshetnyak's original local result \cite{R1976} since we do not explicitly estimate the admissible range of $p$.

\begin{theorem} \label{thm:local-main}
Let $n\ge 3$ and $p>n$. There exist constants
$$
\delta_{\mathrm{loc}}=\delta_{\mathrm{loc}}(n,p)>0,
\qquad C_{\mathrm{loc}}=C_{\mathrm{loc}}(n,p)>0,
\qquad \zeta_0=\zeta_0(n)\in(0,1),
$$
with the following property. Let $Q=Q(x_0,r)$ be an open cube of sidelength $r$, assume that $2Q\subset\subset \Omega$, and let
$  f:2Q\to\R^n$
be a nonconstant sense-preserving $K$-quasiregular mapping satisfying
$$
 0< K-1\le\delta_{\mathrm{loc}}.
$$
Then there exists a sense-preserving M\"obius transformation $\varphi_Q$ which is finite on $\frac32Q$, such that
$$
 g_Q:=\varphi_Q^{-1}\circ f=\operatorname{id}+u_Q \quad \text{on }Q, \qquad  \Pi_{\zeta_0,Q}u_Q=0,
$$
where $\Pi_{\zeta_0,Q}$ is the projection defined in \eqref{eq:unit-projection},
and
$$
 \left(\fint_Q|Dg_Q-\Id|^p\,dx\right)^{1/p}
+\frac1r\left(\fint_Q|g_Q-\operatorname{id}|^p\,dx\right)^{1/p}   
+\frac1r\|g_Q-\operatorname{id}\|_{L^\infty(Q)}
\le C_{\mathrm{loc}}(K-1).
$$
\end{theorem}

Our argument is almost entirely self-contained. We rely on Reshetnyak's original normalization argument for Proposition~\ref{prop:preliminary-gauge}. The proof of the trace-free Korn-type inequality in Theorem~\ref{thm:projected-korn}, interpreted below as a spectral-gap estimate for the conformal Killing operator (also known as the Ahlfors operator \cite{A1974,A1976}), uses the negative-norm estimate of Lewintan--Neff~\cite[Lemma~3.1]{LN2021}.

The exponent $1$ on $K-1$ in Theorem~\ref{thm:local-main} is optimal, as can be seen by taking $f$ to be a fixed anisotropic linear perturbation of the identity. In particular, $K-1$ cannot be replaced by $(K-1)^\alpha$ for any $\alpha>1$. The same example also applies for all the stability theorems below. 
The constant $\delta_{\rm loc}$ is determined implicitly within the proof, and for an explicit computation, we refer the reader to \cite{R2005}. Moreover,   explicit estimates on the constant $C_{\rm loc}$ can be found in \cite{S1976,S1987,S1988}.

\begin{definition}[John domain]\label{def:john}
Let $J\ge1$. A bounded domain $\Omega\subset\R^n$ is called a
$J$-John domain if there exists a point $x_\ast\in\Omega$ such that, for every
$x\in\Omega$, there are $L>0$ and an arclength-parametrized rectifiable curve
$$
 \gamma:[0,L]\to\Omega,
 \qquad \gamma(0)=x,
 \qquad \gamma(L)=x_\ast,
$$
satisfying
$$
 t\le J\,\dist(\gamma(t),\partial\Omega)
 \qquad\text{for every }t\in[0,L].
$$
The point $x_\ast$ is called a John center.
\end{definition}

The next result was originally proven by Reshetnyak in \cite{R19762}. 

\begin{theorem}\label{thm:john-main}
Let $n\ge3$, $p>n$, and $J\ge1$. There exist constants
$$
 \delta_J=\delta_J(n,p,J)>0,
 \qquad  C_J=C_J(n,p,J)>0,
$$
with the following property. Let $\Omega\subset\R^n$ be a bounded
$J$-John domain, and let $ f:\Omega\to\R^n$
be a nonconstant sense-preserving $K$-quasiregular mapping satisfying
$$
 0< K-1\le\delta_J.
$$
Then there exists a sense-preserving M\"obius transformation $\Phi$ such that
$\Phi^{-1}\circ f$ is finite throughout $\Omega$ and
$$
\frac{\|\Phi^{-1}\circ f-\operatorname{id}\|_{L^\infty(\Omega)}}{\diam\Omega}
+|\Omega|^{-1/p}\|D(\Phi^{-1}\circ f)-\Id\|_{L^p(\Omega)}
\le C_J(K-1).
$$
\end{theorem}

Although the local result Theorem~\ref{thm:local-main} and the John-domain version Theorem~\ref{thm:john-main} were previously obtained by Reshetnyak, our contribution lies in providing a more transparent proof and a unified absorption method for the former, as well as a more flexible proof based on Boman chains  for the latter. To the author's knowledge, these approaches are new. Furthermore, the strategy used to establish Theorem~\ref{thm:john-main} can be extended to arbitrary domains, yielding a weighted compactified version of the theorem.

Let us be more specific and formulate the general theorem in the framework of weighted Sobolev spaces. 
Let $\Omega\subset\mathbb R^n$ be a bounded domain and, for a fixed translated dyadic grid, 
let $\mathscr W$ be the family of maximal dyadic cubes $Q$ satisfying
$$
64\,\overline Q\subset\Omega.
$$
Choose $Q_\ast\in\mathscr W$ with maximal sidelength and set
$\ell_\ast:=\ell(Q_\ast).$
We say that two cubes $P,Q\in\mathscr W$ are adjacent if
$$
\operatorname{int}(2P)\cap\operatorname{int}(2Q)\ne\varnothing.
$$
The resulting adjacency graph is connected and has vertex degree bounded by a constant depending only on $n$. 
Equip this graph with a rooted spanning tree $T$ having root $Q_\ast$, 
and let $d_T(Q)$ denote the tree distance from $Q_\ast$ to $Q\in\mathscr W$.

\begin{theorem}\label{thm:weighted-main}
Let $n\ge3$ and $p>n$. There exist constants
$$
\delta_\ast=\delta_\ast(n,p)>0,  \qquad
\theta=\theta(n,p)\in(0,1),  \qquad  C_\ast=C_\ast(n,p)>0,
$$
with the following property. Let $\Omega\subset\R^n$ be bounded, connected and open,
construct $Q_\ast$, $\mathscr W$, and $T$ as above, and define
\begin{equation}\label{eq:defn-w}
w(x):=\theta^{d_T(Q)} \left(\frac{\ell(Q)}{\ell_\ast}\right)^p
 \quad\text{for almost every }x\in Q,  \qquad  d\mu(x):=w(x)\,dx.
\end{equation}
Let $f:\Omega\to\R^n$
be a nonconstant sense-preserving $K$-quasiregular mapping satisfying
$$
  0< K-1\le\delta_\ast.
$$
Then there exists a sense-preserving M\"obius transformation $\Phi$ such that, with
$$
  S_\ast:=S_{Q_\ast},  \qquad
 F:=\widehat{S_\ast\circ\Phi^{-1}}\circ\sigma\circ f, \qquad
  F_0:=\sigma\circ S_\ast,
$$
one has
$  F,F_0\in W^{1,p}_{\mathrm{loc}}(\Omega;\mathbb S^{n}),$
and
$$
\int_\Omega \left(|F-F_0|^p+\ell_\ast^p|DF-DF_0|^p\right)d\mu
\le C_\ast(K-1)^p\mu(\Omega).
$$
Moreover,
$$
  0<w(x)\le 1  \quad\text{for almost every }x\in\Omega,\qquad
  \ell_\ast^n\le\mu(\Omega)\le C(n,p)\ell_\ast^n,
$$
and $\mu$ has full support in $\Omega$. In particular, for each $Q\in\mathscr W$,
$$
\mu(Q)=\theta^{d_T(Q)} \frac{\ell(Q)^{n+p}}{\ell_\ast^p}.
$$
\end{theorem}

The measure appearing in \Cref{thm:weighted-main} depends on the domain $\Omega$, on the choice of dyadic grid, on the root cube, on the associated rooted tree, and on the exponent $p$, but it is independent of the underlying function $f$. It constitutes a natural choice for the chaining argument employed in the proof. Moreover, in order to cut down certain dependence, one may replace $d_T(Q)$ in  \eqref{eq:defn-w} by the quasihyperbolic distance from $x$ to a fixed point $x_0 \in \Omega$, in which case Theorem~\ref{thm:weighted-main} continues to hold up to a multiplicative constant depending on the dimension $n$ (and on the choice of the base point $x_0$). This variant can be treated elsewhere, and we concentrate on the proof based on the Whitney decomposition  in the present manuscript.

Furthermore, although \Cref{thm:weighted-main} is also formulated for John domains, it does not, in that setting, immediately yield the conclusion of \Cref{thm:john-main}. 
The reason is that John geometry gives two additional crucial properties, namely \eqref{eq:john-terminal-shadow} and \eqref{eq:john-suffix-sum} below. The first enables the Bojarski-type enlarged-cube summation in \Cref{lem:Bojarski-dilation} with respect to Lebesgue measure, whereas the second ensures that the associated M\"obius transformation remains in pole-free regions along the entire chain; see \Cref{prop:john-backward-orbit}. In the general non-John setting, it is unclear how to establish analogues of these two properties even with the aid of the weight.

\subsection{Strategy of the proofs}
We first explain the local reorganization. Reshetnyak's proof of the local theorem has several layers. The proof begins with a compactness argument which produces a preliminary M\"obius normalizer. The normalized map is then adjusted by a finite-dimensional M\"obius correction so that its error is orthogonal to the infinitesimal M\"obius fields. Next, Reshetnyak proves comparison estimates between 
$$D(\varphi^{-1}\circ f)-\Id\quad  \text{ and } \quad Df-D\varphi,$$
and he applies a bounded-specific-oscillation argument, a version of the John--Nirenberg inequality for BMO functions, to obtain an infinitesimal gain of integrability:
For some $\varepsilon>0$ and some function $\nu(\delta)\to0$,
$$
 \int_Q |D(\varphi^{-1}\circ f)-\Id|^{p+\varepsilon}\,dx
 \le \nu(\delta) \int_Q |D(\varphi^{-1}\circ f)-\Id|^p\,dx,
  \qquad \delta=K-1.
$$
This is the estimate used in ~\cite[Chapter~4, Section~3.3]{R1994}. The last step in Reshetnyak's proof then estimates the trace-free symmetric part of $D(\varphi^{-1}\circ f)-\Id$ and splits the domain into small-gradient and large-gradient regions.

Our first contribution is to isolate the finite-dimensional normalization and the last nonlinear absorption. Let
$$
\mathcal E^{\mathrm D}u := \operatorname{sym}Du-\frac{\operatorname{div}u}{n}\Id
$$
be the trace-free symmetric gradient. The kernel of $\mathcal E^{\mathrm D}$ is the finite-dimensional space
$$
\mathcal K_n =\left\{\mathbf a+Ax+\lambda x+\bigl(2(\mathbf c\cdot x)x-|x|^2\mathbf c\bigr): \mathbf a,\mathbf c\in\R^n,\ A^T=-A,\ \lambda\in\R \right\}.
$$
We impose the normalization
$$
  \Pi_{\zeta,Q}(g-\operatorname{id})=0
$$
by minimizing a finite-dimensional moment functional over a small local chart of the M\"obius group. This is analogous in spirit to \cite[Lemma~4.1]{FZ2022}: One chooses a nearby point on a
finite-dimensional extremal manifold and obtains exact cancellation of tangent modes. The functional used here is different. The functional used here is different: It is one half of the squared norm of the moment map
$$
a\ \mapsto\ \Pi_{\zeta,Q}\bigl(\Psi(a)\circ g-\operatorname{id}\bigr).
$$
where $\Psi(a)$ is a local M\"obius chart. The derivative of this moment map at $(a,g)=(0,\operatorname{id})$ is the identity matrix, and hence the minimizer is unique in a small parameter ball.

The nonlinear estimate used in the absorption is the following matrix inequality, which  corresponds to the idea of \cite[Lemma 2.1]{FZ2022}. If $I_n+A\in\GL^+(n)$ and
$$
  K_O(I_n+A):= \frac{\|I_n+A\|_{\mathrm{op}}^n}{\det(I_n+A)} \le 1+\delta,
$$
then
$$
\left|\operatorname{dev}(\operatorname{sym}A)
\right|\le C_n\left[\delta(1+|A|)+\frac{|A|^2}{1+|A|}\right];
$$
see \eqref{eq:defn-sym-dev} for the notation. 
The second term is chosen deliberately:
$$
\frac{|A|^2}{1+|A|}\sim |A|^2\quad \text{ when } \ |A|\le 1,
\qquad  \frac{|A|^2}{1+|A|}  \sim |A|\quad\text{ when } \ |A|\ge 1.
$$
Now for $p\ge 1$, in conjunction with
$$
\left(\frac{t^2}{1+t}\right)^p\le t^{p+\varepsilon}
\qquad \text{ for any} \ t\ge 0  \text{ and }  0<\epsilon\le p,
$$
this yields a single unified estimate that replaces Reshetnyak’s original decomposition into small-gradient and large-gradient regions. Thus the final local absorption becomes
$$
\begin{aligned}
\|Dg-\Id\|_{L^p(Q)}
&\le C\|\mathcal E^{\mathrm D} (g-\operatorname{id}) \|_{L^p(Q)}
\\
&\le C\delta |Q|^{1/p} + C\delta\|Dg-\Id\|_{L^p(Q)}
+ C\nu(\delta)^{1/p}\|Dg-\Id\|_{L^p(Q)}.
\end{aligned}
$$
For $K-1$ small, the last two terms are absorbed. The coercive estimate for
$\mathcal E^{\mathrm D}$ modulo $\mathcal K_n$ is ultimately a finite-kernel
estimate of Reshetnyak~\cite{R1970}; related conformal Korn inequalities can be found, for example, in Dain~\cite{D2006} and Lewintan--Neff~\cite{LN2021}.

The global part of the paper has two different mechanisms. 
On bounded John domains, we do not use the original argument of Reshetnyak by gluing M\"obius transformations on neighboring Whitney cubes. 
Instead, we utilize the idea of Bojarski \cite{B1988}. Given a John center $x_\ast$, we construct Whitney chains (also referred to as Boman chains) by tracing John curves from the terminal cubes to a distinguished root cube that contains $x_\ast$. The resulting chain
$$
  Q_0,\ldots,Q_m=Q
$$
has two key properties: For some constants $\Lambda_J, C_J>0$,
$$
Q\subset \Lambda_J Q_j, \qquad \sum_{k=j}^m \ell(Q_k)\le C_J\ell(Q_j).
$$
The second estimate ensures that all backward M\"obius transformations along the chain remain confined to regions that are free of poles; see \Cref{prop:john-backward-orbit}.
The first estimate gives
$$
  N(x)+1  \le  \sum_{R\in\mathscr W}\chi_{\Lambda_JR}(x),
$$
where $N(x)$ is the length of the selected chain ending at the Whitney cube containing $x$. Bojarski's enlarged-cube lemma then yields
$$
 \|N+1\|_{L^q(\Omega)}  \le C(n,J)q\,|\Omega|^{1/q},  \qquad q\ge2,
$$
and consequently
$$
  \int_\Omega e^{\tau_JN(x)}\,dx  \le C(n,J)|\Omega|;
$$
see also \cite{KR1997}. 
Bojarski's chain condition and enlarged-cube estimate can be found in
\cite{B1988}; his formulation allows chain lengths depending on the terminal cube.

For general connected open sets, the construction of Boman chains is not available. In particular, an extended corridor composed of Whitney cubes of mutually comparable side length can generate an arbitrarily large number of M\"obius transitions. Consequently, we employ a rooted Whitney tree $T$ and incorporate a factor depending on the tree depth into the measure. If $Q_\ast$ is the root cube and $d_T(Q)$ is the tree distance from
$Q_\ast$ to $Q$, then
$$
  w(x)  = \theta^{d_T(Q)} \left(\frac{\ell(Q)}{\ell(Q_\ast)}\right)^p \quad\text{for a.e. }x\in Q,
  \qquad  d\mu=w\,dx.
$$
The power $\ell(Q)^p$ exactly cancels the first-order scaling in the derivative
term, while $\theta^{d_T(Q)}$ compensates for exponential growth along the
rooted tree. The resulting measure satisfies
$$
  0<w\le1,  \qquad  \mu(E)\le |E|,
$$
and, on each Whitney cube,
$$
 \mu(Q) =  \theta^{d_T(Q)}  \frac{\ell(Q)^{n+p}}{\ell(Q_\ast)^p}.
$$
Thus the measure is absolutely continuous with respect to Lebesgue measure and is independent of the mapping $f$.

The arbitrary-domain estimate is stated after compactification. This is necessary: A single root M\"obius normalizer $\Phi$ may have the property that $\Phi^{-1}\circ f$ takes the value $\infty$ on a remote part of the domain. We therefore use the inverse stereographic projection
$$
 \sigma:\widehat{\R}^{\,n}:=\R^n\cup\{\infty\}\to\mathbb S^n
$$
and prove an estimate for sphere-valued maps. The M\"obius group is propagated through its future-cone-preserving Lorentz representation, following the standard matrix realization of higher-dimensional M\"obius transformations; see Gehring--Martin--Palka~\cite[Section~3.7]{GMP2017} and Reshetnyak~\cite[Chapter~2]{R1994}. A Euclidean weighted estimate can be recovered only under a quantitative separation from the stereographic pole.

\subsection*{Organization of the paper}

\Cref{sec:preliminaries} introduces the notation and fundamental analytical framework employed throughout the manuscript, including the treatment of cubes, matrices, quasiregular mappings, M\"obius space, stereographic projection, and basic Sobolev-type estimates. \Cref{sec:conformal-korn} is dedicated to a detailed study of the trace-free symmetric gradient and to the derivation of the projected conformal Korn inequality. 

The proof of \Cref{thm:local-main} is carried out in \Cref{sec:local-stability}. \Cref{sec:whitney-edge} develops the Whitney decomposition and establishes quantitative estimates for neighboring Whitney cubes, whereas \Cref{sec:lorentz} formulates the Lorentz representation and derives quantitative bounds for M\"obius transformations. The proof of \Cref{thm:john-main} is presented in \Cref{sec:john}, and, finally, \Cref{sec:weighted} provides the proof of \Cref{thm:weighted-main}.

\medskip
\noindent{\bf AI disclaimer}: During the preparation of this manuscript, the author used ChatGPT Pro (version 5.6) and Overleaf AI to refine the text and assist in drafting technical and tedious proofs, especially those in the Appendix. The core concepts, theoretical insights, and interpretations were developed and articulated solely by the author. Following the use of these tools, the author independently reviewed, edited, and verified the content as necessary, and assumes full responsibility for all statements and conclusions presented in the final manuscript.




\section{Preliminary facts}\label{sec:preliminaries}

\subsection{Notation}\label{sec:notation}

A \emph{domain} is a nonempty connected open subset of $\R^n$. For a set $E\subset\R^n$, $|E|$ denotes its Lebesgue measure, $\overline E$ its closure, and $\interior E$ its interior. For nonempty sets $E,F\subset\R^n$, set
$$
\dist(E,F):=\inf\{|x-y|:x\in E,\ y\in F\},\qquad 
\diam (E):=\sup\{|x-y|:x,y\in E\},
$$
and we abbreviate $\dist(x,F):=\dist(\{x\},F)$.

The open Euclidean ball of center $a$ and radius $r>0$ is
$$
B(a,r):=\{x\in\R^n:|x-a|<r\}.
$$
An open cube $Q=Q(a,r)$ has center $a$ and sidelength $r$:
$$
Q(a,r):=a+\left(-\frac r2,\frac r2\right)^n.
$$
We write $x_Q:=a$ and $\ell(Q):=r.$
For $\lambda>0$, $\lambda Q$ denotes the cube with center $x_Q$ and sidelength $\lambda\ell(Q)$. We fix the normalized cube
$$
\Qzero:=\left(-\frac12,\frac12\right)^n.
$$
Then the affine normalization of $Q$ and its inverse are
\begin{equation}\label{eq:defn-SQ}
    S_Q(x):=\frac{x-x_Q}{\ell(Q)}, \qquad S_Q^{-1}(y)=x_Q+\ell(Q)y.
\end{equation}
Thus $S_Q(Q)=\Qzero$.

If $h\in L^1(E)$ and $0<|E|<\infty$, then
$$
\fint_Eh\,dx:=\frac1{|E|}\int_Eh\,dx.
$$

Throughout the manuscript, $C$ denotes a  positive constant whose value may vary from one occurrence to the next. When pertinent, its dependence on parameters is indicated explicitly by writing $C(n,p,J)$. In contrast, a constant equipped with a specific label, such as $C_{\mathrm{loc}}$, is fixed once chosen and subsequently retains this value for the remainder of the discussion.

We write $O(n)$ for the group of all $n\times n$ orthogonal matrices under matrix multiplication, $GL^+(n)$ for the group of all $n\times n$ real matrices with strictly positive determinant, and
$$SO(n)=O(n)\cap GL^+(n).$$

Fix a vector $\xi\in\R^n$.  The translated dyadic grid $\mathscr D_\xi$ consists of the cubes
$$
Q=\xi+2^{-k}\bigl(m+[0,1]^n\bigr),
\qquad k\in\mathbb Z, \quad
m\in\mathbb Z^n.
$$
Each $Q\in\mathscr D_\xi$ has a unique \emph{dyadic parent} $Q^+\in\mathscr D_\xi$ of sidelength
$$
\ell(Q^+)=2\ell(Q).
$$
Two dyadic cubes have disjoint interiors unless one contains the other.  At a fixed generation their half-open representatives form a partition of $\R^n$.  We work with the open interiors and ignore their boundaries, which have Lebesgue measure zero.  Whenever coverage at boundary points is relevant, we use the closures of the cubes or their doubled dilates.

We next introduce some  notation on matrices. For two matrices $A,B\in\R^{n\times n}$, set
$$
A:B:=\tr(A^TB),
\qquad
|A|:=(A:A)^{1/2}.
$$
The operator norm is
$$
\|A\|_{\mathrm{op}}:=\sup_{|\xi|=1}|A\xi|.
$$
These two norms are not identified with one another; instead, they satisfy the inequality
$$
\|A\|_{\mathrm{op}} \leq |A| \leq \sqrt{n}\,\|A\|_{\mathrm{op}}.
$$
Define 
\begin{equation}\label{eq:defn-sym-dev}
\sym A:=\frac{A+A^T}{2}, \qquad \dev A:=A-\frac{\tr A}{n}\Id.
\end{equation}

For a matrix $A\in\GL^+(n):=\{A\in\R^{n\times n}:\det A>0\}$, its (outer) distortion is
$$
K_O(A):=\frac{\|A\|_{\mathrm{op}}^n}{\det A}.
$$
We use $K_O(A)$ only when $A\in\GL^+(n)$, and the zero matrix is treated separately whenever it occurs as the derivative of a quasiregular mapping.

\subsection{Quasiregular mappings and Sobolev inequalities}\label{sec:QC}
We now introduce the foundational terminology associated with quasiregular and quasiconformal mappings.

\begin{definition}[Quasiregular and quasiconformal mappings]\label{def:quasiregular}
Let $\Omega\subset\R^n$ be a domain and let $K\ge1$. A mapping
$$
f:\Omega\to\R^n
$$
is called \emph{nonconstant sense-preserving $K$-quasiregular} if
$$
f\in C(\Omega;\R^n)\cap W^{1,n}_{\mathrm{loc}}(\Omega;\R^n),
\qquad
f\not\equiv\text{constant},
$$
its Jacobian $J_f(x):=\det Df(x)$
is nonnegative for almost every $x\in\Omega$, and
$$
\|Df(x)\|_{\mathrm{op}}^n\le KJ_f(x)
\quad\text{for almost every }x\in\Omega.
$$
A quasiregular mapping that is a homeomorphism onto its image is called \emph{quasiconformal}.
\end{definition}

This convention agrees with the modern terminology in e.g. Rickman~\cite[Chapter~I]{R1993}, while Reshetnyak uses the phrase \emph{mapping with bounded distortion}; see~\cite[Chapter~I, Section~4]{R1989}. We record the following elementary observation. 

\begin{lemma}\label{lem:zero-jacobian}
Let $f$ be as in \Cref{def:quasiregular}. Then
$$
J_f(x)=0\quad\to\quad Df(x)=0
$$
for almost every $x\in\Omega$ at which the distortion inequality holds.
\end{lemma}

\begin{proof}
If $J_f(x)=0$, then
$$
0\le\|Df(x)\|_{\mathrm{op}}^n\le KJ_f(x)=0.
$$
Hence $\|Df(x)\|_{\mathrm{op}}=0$, which is equivalent to $Df(x)=0$.
\end{proof}

At almost every point where $Df(x)\ne0$, \Cref{lem:zero-jacobian} gives $J_f(x)>0$, and therefore
$$
K_O(Df(x))\le K.
$$
We only use $K_O(Df(x))$ at the set where $Df(x)\neq 0.$

We also recall the fundamental openness and discreteness theorem: Every nonconstant quasiregular mapping is open and discrete.  In particular, its restriction to every nonempty subdomain is nonconstant.  This result is due to Reshetnyak; see Rickman~\cite[Chapter~I, Section~4]{R1993} and Reshetnyak~\cite[Chapter~II, Section~6]{R1989}.

The M\"obius space is the one-point compactification
$$
\whR:=\R^n\cup\{\infty\}.
$$
We identify it with the unit sphere
$$
\Sn:=\{\xi\in\R^{n+1}:|\xi|=1\}
$$
through the map 
\begin{equation}\label{eq:stereographic}
\sigma(x) = \left(\frac{2x}{1+|x|^2}, \frac{|x|^2-1}{1+|x|^2}\right) \quad\quad x\in\R^n, \quad
\sigma(\infty)=\mathbf{e}_{n+1}.
\end{equation}
Then for $\xi=(\xi',\xi_{n+1})\in\Sn\setminus\{\mathbf{e}_{n+1}\}$,
$$
\sigma^{-1}(\xi)=\frac{\xi'}{1-\xi_{n+1}}.
$$
These formulas and the induced chordal metric are recorded in~\cite[Section~2.2]{GMP2017}.

A M\"obius transformation of $\whR$ is a finite composition of reflections in spheres and hyperplanes. Its restriction to the complement of its pole is smooth and conformal. We denote the full M\"obius group by $\Mob(n)$ and its sense-preserving subgroup by $\Mob^+(n)$. If $M\in\Mob(n)$, its smooth action on the sphere is
\begin{equation}\label{eq:hat-operation}
 \widehat M:=\sigma\circ M\circ\sigma^{-1}:\Sn\to\Sn. 
\end{equation}

A M\"obius transformation fixes $\infty$ if and only if its restriction to $\R^n$ is a similarity. In particular, a general M\"obius transformation is not affine; see~\cite[Section~3.3]{GMP2017}  and ~\cite[Chapter~2, Theorem~1.3]{R1994}.

Now we recall the Poincar\'e and Morrey estimates on cubes. 
\begin{lemma}[{\cite[Section 4.5.2, Theorem 2]{EG1992}}]\label{lem:cube-poincare}
Let $1\le p<\infty$, let $Q\subset\R^n$ be an open cube, and let $u\in W^{1,p}(Q;\R^m)$. With
$$
 u_Q:=\fint_Q u\,dx,
$$
one has
$$
 \|u-u_Q\|_{L^p(Q)}
 \le C(n,p)\ell(Q)\|Du\|_{L^p(Q)}.
$$
\end{lemma}

\begin{lemma}[{\cite[Section~4.5.3, Theorem~3]{EG1992}}]\label{lem:morrey-cube}
Let $p>n$, let $Q\subset\R^n$ be a cube, and let $u\in W^{1,p}(Q;\R^m)$.  If
$$
u_Q:=\fint_Qu\,dx,
$$
then $u$ has a continuous representative on $Q$ and
$$
\|u-u_Q\|_{L^\infty(Q)}
\le C(n,p)\ell(Q)^{1-n/p}\|Du\|_{L^p(Q)}.
$$
\end{lemma}

The following inner version is an immediate consequence of \Cref{lem:cube-poincare}. 
\begin{lemma}\label{lem:inner-poincare}
Let $1\le p<\infty$, let $Q\subset\R^n$ be a cube, and let $\zeta\in(0,1)$. If $u\in W^{1,p}(Q;\R^m)$ satisfies
$$
\fint_{\zeta Q}u\,dx=0,
$$
then
$$
\|u\|_{L^p(Q)}
\le C(n,p,\zeta)\ell(Q)\|Du\|_{L^p(Q)}.
$$
\end{lemma}
\begin{proof}
Let
$$
u_Q:=\fint_Qu\,dx.
$$
Then \Cref{lem:cube-poincare} yields
\begin{equation}\label{eq:poincare-u}
\|u-u_Q\|_{L^p(Q)} \le C(n,p)\ell(Q)\|Du\|_{L^p(Q)}.
\end{equation}

Since the average of $u$ over $\zeta Q$ vanishes,
$$
u_Q = \fint_{\zeta Q}(u_Q-u(x))\,dx.
$$
By H\"older's inequality,
\begin{align}
|u_Q|&\le |\zeta Q|^{-1}\int_{\zeta Q}|u-u_Q|\,dx\notag \\
&\le |\zeta Q|^{-1/p}\|u-u_Q\|_{L^p(\zeta Q)}\le \zeta^{-n/p}|Q|^{-1/p}\|u-u_Q\|_{L^p(Q)}. \label{eq:average-u}
\end{align}
Consequently, combining \eqref{eq:poincare-u} and \eqref{eq:average-u} yields
\begin{align*}
\|u\|_{L^p(Q)}
&\le \|u-u_Q\|_{L^p(Q)}+|Q|^{1/p}|u_Q|\\
&\le \bigl(1+\zeta^{-n/p}\bigr)\|u-u_Q\|_{L^p(Q)}\le C(n,p,\zeta)\ell(Q)\|Du\|_{L^p(Q)}.
\end{align*}
This proves the lemma. 
\end{proof}

\subsection{The conformal Killing operator and its projected coercivity}\label{sec:conformal-korn}
We first recall the following fact on distributions.  
\begin{lemma} \label{lem:vanishing-hessian}
Let $U\subset\R^n$ be a domain.
\begin{enumerate}[label=\textup{(\roman*)},leftmargin=2.3em]
\item If $T\in\mathscr D'(U)$ and $\partial_iT=0$ in $\mathscr D'(U)$ for every $i$, then $T$ is a constant distribution on $U$.
\item If $h\in\mathscr D'(U)$ and $\partial_i\partial_jh=0$ in $\mathscr D'(U)$ for all $i,j$, then there exist $\alpha\in\R$ and $\beta\in\R^n$ such that $h=\alpha+\beta\cdot x$ in $\mathscr D'(U)$.
\item If $v\in\mathscr D'(U;\R^m)$ and $D^2v=0$, then $v(x)=a+Bx$ in distributions for some $a\in\R^m$ and $B\in\R^{m\times n}$.
\end{enumerate}
\end{lemma}

\begin{proof}
We first prove \textup{(i)}. Let $B\subset\subset U$ be a ball and let $\rho\in C_c^\infty(B(0,1))$ be a nonnegative mollifier with integral one. For sufficiently small $\epsilon>0$, the convolution
$$
 T_\epsilon:=T*\rho_\epsilon
$$
is smooth on a neighborhood of $\overline B$. Since distributional differentiation commutes with convolution, then
$$
 \nabla T_\epsilon=(\nabla T)*\rho_\epsilon=0
 \quad\text{on } B.
$$
Hence $T_\epsilon$ is a constant $c_\epsilon$ on $B$. Choose $\phi\in C_c^\infty(B)$ with $\int_B\phi\,dx=1$. Since $T_\epsilon\to T$ in $\mathscr D'(B)$,
$$
 c_\epsilon=\langle T_\epsilon,\phi\rangle
 \to\langle T,\phi\rangle=:c_B.
$$
Thus $T=c_B$ in $\mathscr D'(B)$. 

If two such balls overlap, their constants agree after testing on a nonzero function supported in the overlap. Recall that every connected open subset of $\R^n$ is path connected.Thus, any two points can  joined by a compact path, whose image can be covered by a finite chain of overlapping balls compactly contained in $U$. Therefore all local constants agree, and $T$ is constant on $U$.

For \textup{(ii)}, apply \textup{(i)} to each first derivative $\partial_i h$. There are constants $\beta_i$ such that $\partial_i h=\beta_i$. The distribution $h-\beta\cdot x$ has zero gradient, and hence another application of \textup{(i)} gives $h=\alpha+\beta\cdot x$. For \textup{(iii)}, apply \textup{(ii)} separately to each scalar component of $v$.  The resulting constants assemble into a vector $a\in\R^m$ and the resulting first-order coefficients assemble into a matrix $B\in\R^{m\times n}$, thus $v(x)=a+Bx$ in distributions.
\end{proof}

For $u\in W^{1,1}_{\mathrm{loc}}(\Omega;\R^n)$, define
\begin{equation}\label{eq:conformal-killing-operator}
\Edev u := \dev(\sym Du)
= \frac{Du+Du^T}{2} -\frac{\operatorname{div}u}{n}\Id.
\end{equation}
This is usually called the \emph{conformal Killing operator} applied to $u$. For $n\ge3$, its kernel is the finite-dimensional space of conformal Killing fields. We give the classification of conformal Killing fields below, and postpone its proof into the appendix; see also \cite[Theorem 3.3]{L2013}. 

\begin{proposition}\label{prop:conformal-killing-fields}
Let $U\subset\R^n$ be a domain, let $n\ge3$, and let $u\in W^{1,1}_{\mathrm{loc}}(U;\R^n)$. Then
$$
\Edev u=0 \quad\text{in }\mathscr D'(U)
$$
if and only if there exist vectors $\mathbf{a},\mathbf{c}\in\R^n$, a skew-symmetric matrix $A\in\R^{n\times n}$, and a scalar $\lambda\in\R$ such that
\begin{equation}\label{eq:ck-field}
u(x)
= \mathbf{a}+Ax+\lambda x+2(\mathbf{c}\cdot x)x-|x|^2\mathbf{c}
\quad\text{for almost every }x\in U.
\end{equation}
\end{proposition}

We denote the vector space of the fields in \eqref{eq:ck-field} by $\Kconf,$
whose dimension is
$$
\dim\Kconf =n+\frac{n(n-1)}2+1+n =\frac{(n+1)(n+2)}2.
$$
The same kernel is described in Reshetnyak~\cite{R1970} and Dain~\cite{D2006}.

Now we introduce the moment projection. Let $\zeta\in(0,1)$. On the normalized cube $\Qzero$, we equip $\Kconf$ with the inner product
$$
\langle X,Y\rangle_{\zeta}
:=
\fint_{\zeta\Qzero}X(y)\cdot Y(y)\,dy.
$$
Choose an orthonormal basis
$$
X_1,\ldots,X_N,
\qquad
N:=\frac{(n+1)(n+2)}2.
$$
For $v\in L^1(\Qzero;\R^n)$, define
\begin{equation}\label{eq:unit-projection}
\Pi_{\zeta,\Qzero}v
:=
\sum_{\alpha=1}^N
\left(
\fint_{\zeta\Qzero}v(y)\cdot X_\alpha(y)\,dy
\right)X_\alpha.
\end{equation}
This definition is independent of the chosen orthonormal basis. It is the orthogonal projection onto $\Kconf$ with respect to the normalized $L^2(\zeta\Qzero)$ inner product whenever $v\in L^2(\zeta\Qzero)$.

For a general cube $Q$, define the vector-field scaling
$$
(\mathcal T_Qu)(y):=\frac{u(x_Q+\ell(Q)y)}{\ell(Q)}, \qquad y\in\Qzero.
$$
The transformation $\mathcal T_Q$ maps $\Kconf$ bijectively onto itself. Indeed, let
$$
X(x)=a+Bx+\lambda x+2(c\cdot x)x-|x|^2c, \qquad B^T=-B.
$$
Writing $x=x_Q+\ell(Q)y$, a direct expansion gives
$$
(\mathcal T_QX)(y) =a_Q+B_Qy+\lambda_Qy +2(c_Q\cdot y)y-|y|^2c_Q,
$$
where
$$
c_Q=\ell(Q)c, \qquad B_Q=B+2(x_Qc^T-cx_Q^T), \qquad \lambda_Q=\lambda+2c\cdot x_Q,
$$
and
$$
a_Q=\ell(Q)^{-1} \bigl(a+Bx_Q+\lambda x_Q +2(c\cdot x_Q)x_Q-|x_Q|^2c\bigr).
$$
Note that the matrix $B_Q$ is skew-symmetric. Hence
$\mathcal T_QX\in\mathcal K_n$. Applying the same argument to
$\mathcal T_Q^{-1}$ proves that $\mathcal T_Q$ maps
$\mathcal K_n$ bijectively onto itself.

We define the corresponding projection on $Q$ as
\begin{equation}\label{eq:scaled-projection}
\Pi_{\zeta,Q}:= \mathcal T_Q^{-1} \circ\Pi_{\zeta,\Qzero}\circ\mathcal T_Q.
\end{equation}
In particular,
$$
\Pi_{\zeta,Q}X=X \qquad\text{for every }X\in\Kconf.
$$
We show such a projection is bounded. 
\begin{lemma}\label{lem:projection-bounded}
Let $1<p<\infty$ and $\zeta\in(0,1)$. There exists $C=C(n,p,\zeta)$ such that, for every cube $Q$ and every $u\in L^p(Q;\R^n)$,
\begin{equation}\label{eq:projection-bound}
\|\Pi_{\zeta,Q}u\|_{L^p(Q)} +\ell(Q)\|D(\Pi_{\zeta,Q}u)\|_{L^p(Q)} \le C\|u\|_{L^p(Q)}.
\end{equation}
Consequently, $\Pi_{\zeta,Q}$ is a bounded projection from $W^{1,p}(Q;\R^n)$ onto $\Kconf$.
\end{lemma}

\begin{proof}
It is enough to prove the estimate on $\Qzero$. By \eqref{eq:unit-projection} and H\"older's inequality, for each $X_\alpha$
\begin{align*}
\left| \fint_{\zeta\Qzero}v\cdot X_\alpha\,dy \right|
&\le
|\zeta\Qzero|^{-1} \|v\|_{L^p(\zeta\Qzero)} \|X_\alpha\|_{L^{p'}(\zeta\Qzero)}\\
&\le C(n,p,\zeta)\|v\|_{L^p(\Qzero)},
\end{align*}
where $p'=p/(p-1)$. Since the family $\{X_\alpha\}_{\alpha=1}^N$ is fixed and finite,
\begin{equation}\label{eq:unit-cube-estimate} 
  \|\Pi_{\zeta,\Qzero}v\|_{W^{1,p}(\Qzero)} \le C(n,p,\zeta)\|v\|_{L^p(\Qzero)}.  
\end{equation}

Apply this estimate to $v=\mathcal T_Qu$. Since the change of variables $x=x_Q+\ell(Q)y$ gives
$$
\|\mathcal T_Qu\|_{L^p(\Qzero)} =\ell(Q)^{-1-n/p}\|u\|_{L^p(Q)},
$$
$$
\|\mathcal T_Q(\Pi_{\zeta,Q}u)\|_{L^p(\Qzero)} =\ell(Q)^{-1-n/p}\|\Pi_{\zeta,Q}u\|_{L^p(Q)},
$$
and
$$
\|D(\mathcal T_Q(\Pi_{\zeta,Q}u))\|_{L^p(\Qzero)}
=\ell(Q)^{-n/p}\|D(\Pi_{\zeta,Q}u)\|_{L^p(Q)},
$$
multiplying \eqref{eq:unit-cube-estimate} by $\ell(Q)^{1+n/p}$ yields \eqref{eq:projection-bound}.
\end{proof}

We briefly specify the notation in the negative-norm estimate. For a vector-valued distribution
$$
a=(a_1,\ldots,a_n)    \in \mathscr D'(\Qzero;\R^n),
$$
its generalized curl is the distribution
$$
\operatorname{curl}_n a\in \mathscr D'\!\left( \Qzero;\R^{\binom{n}{2}}\right)
$$
defined by
$$
(\operatorname{curl}_n a)_{ij} :=\partial_i a_j-\partial_j a_i,
\qquad 1\le i<j\le n.
$$
For a matrix-valued distribution
$$
P=(P_{ij})_{i,j=1}^n \in \mathscr D'(\Qzero;\R^{n\times n}),
$$
we define its symmetric and skew-symmetric parts by
$$
\operatorname{sym}P:=\frac12(P+P^{T}),
\qquad
\operatorname{skew}P:=\frac12(P-P^{T}),
$$
where $(P^{T})_{ij}:=P_{ji}$. Thus,
$$
(\operatorname{skew}P)_{ij} =\frac12(P_{ij}-P_{ji}).
$$

The generalized matrix curl $\operatorname{Curl}_nP$ is obtained by applying $\operatorname{curl}_n$ to each row of $P$. More precisely,
if
$$
P_{k\bullet}:=(P_{k1},\ldots,P_{kn}) \in \mathscr D'(\Qzero;\R^n)
$$
denotes the $k$-th row, then
$$
\operatorname{Curl}_nP
:=
\begin{pmatrix}
    \operatorname{curl}_nP_{1\bullet}\\
    \vdots\\
    \operatorname{curl}_nP_{n\bullet}
    \end{pmatrix}
    \in
    \mathscr D'\!\left(
        \Qzero;
        \R^{n\times\binom{n}{2}}
    \right).
$$
More precisely, its components are
\begin{equation}\label{eq:generalized-matrix-curl}
(\operatorname{Curl}_nP)_{k,ij}
    :=\partial_iP_{kj}-\partial_jP_{ki},
\qquad 1\le k\le n,\quad 1\le i<j\le n.
\end{equation}
All derivatives in \eqref{eq:generalized-matrix-curl} are understood in the distributional sense. Thus, for every
$\varphi\in C_c^\infty(\Qzero)$,
$$\left\langle
(\operatorname{Curl}_nP)_{k,ij},\varphi
\right\rangle
=-\left\langle P_{kj},\partial_i\varphi\right\rangle
+\left\langle P_{ki},\partial_j\varphi\right\rangle.
$$
In particular, if $u\in W^{1,p}(\Qzero;\R^n)$ and $P=Du$, then
the commutation of distributional derivatives gives
$$
(\operatorname{Curl}_nDu)_{k,ij}
=\partial_i\partial_j u_k -\partial_j\partial_i u_k =0,
$$
and hence
$$
\operatorname{Curl}_n(Du)=0
\qquad\text{in }\mathscr D'\!\left(
\Qzero;\R^{n\times\binom{n}{2}}\right).
$$

 We are now prepared to present the following projected trace-free Korn inequality, which gives the precise coercive estimate employed in the local stability analysis. The original result is due to Reshetnyak; one can  obtain it by combining \Cref{lem:projection-bounded} with  \cite[Theorem~2]{R1970}; see also \cite[Chapter~4, Lemma~2.2]{R1994}. The proof below introduces its spectral-gap mechanism by combining a negative-norm  estimate with compactness.

\begin{proposition}
\label{thm:projected-korn}
Let $n\ge3$, $1<p<\infty$, and $\zeta\in(0,1)$. There exists $C=C(n,p,\zeta)$ such that, for every cube $Q\subset\R^n$ and every $u\in W^{1,p}(Q;\R^n)$ satisfying
\begin{equation}\label{eq:orthogonal-condition}
    \Pi_{\zeta,Q}u=0,
\end{equation}
one has
\begin{equation}\label{eq:projected-korn}
\frac1{\ell(Q)}\|u\|_{L^p(Q)}
+\|Du\|_{L^p(Q)}
\le
C\|\Edev u\|_{L^p(Q)}.
\end{equation}
\end{proposition}

\begin{proof}
We first work on the normalized cube $\Qzero$, and use the negative-norm estimate given by Lewintan--Neff~\cite[Lemma~3.1]{LN2021}. On $\Qzero$, this estimate can be interpreted as, for $P\in\mathscr D'(\Qzero;\R^{n\times n})$ with
$$\|\dev(\sym P)\|_{L^p(\Qzero)} + \|\operatorname{Curl}_nP\|_{W^{-1,p}(\Qzero)}<\infty, $$
one has
\begin{align}
\|P\|_{L^p(\Qzero)} \le C(n,p)\bigg(&
\left\|\operatorname{skew}P+\frac{\tr P}{n}\Id\right\|_{W^{-1,p}(\Qzero)}
+\|\dev(\sym P)\|_{L^p(\Qzero)} \notag\\
&+\|\operatorname{Curl}_nP\|_{W^{-1,p}(\Qzero)}
\bigg).
\label{eq:ahlfors-negative-norm}
\end{align}
Set $P=Du$. Since distributional derivatives commute, $\operatorname{Curl}_n(Du)=0$. Moreover, componentwise one has
$$
\left(\operatorname{skew}Du+\frac{\operatorname{div}u}{n}\Id\right)_{ij}
=\frac12(\partial_j u_i-\partial_i u_j)
+\frac{\delta_{ij}}{n}\sum_{k=1}^n\partial_k u_k.
$$
Since for $\phi\in W^{1,p'}_0(\Qzero)$, where $p'=p/(p-1)$,
$$
|\langle\partial_j u_i,\phi\rangle|
=\left|\int_{\Qzero}u_i\,\partial_j\phi\,dx\right|
\le \|u_i\|_{L^p(\Qzero)}\|\phi\|_{W^{1,p'}_0(\Qzero)},
$$
consequently we conclude
$$
\left\|\operatorname{skew}Du+\frac{\operatorname{div}u}{n}\Id\right\|_{W^{-1,p}(\Qzero)}
\le C(n,p)\|u\|_{L^p(\Qzero)}.
$$
Since $\dev(\sym Du)=\Edev u$, \eqref{eq:ahlfors-negative-norm} yields the following estimate
\begin{equation}\label{eq:ahlfors-garding}
\|Du\|_{L^p(\Qzero)}
\le C(n,p)\bigl(
\|u\|_{L^p(\Qzero)}+\|\Edev u\|_{L^p(\Qzero)}
\bigr).
\end{equation}

We claim that, under the orthogonal condition \eqref{eq:orthogonal-condition},
the lower-order term in \eqref{eq:ahlfors-garding} can be removed. Suppose otherwise. By homogeneity there are $u_j\in W^{1,p}(\Qzero;\R^n)$ such that
\begin{equation}\label{eq:ahlfors-null-sequence}
\Pi_{\zeta,\Qzero}u_j=0,
\qquad
\|u_j\|_{L^p(\Qzero)}+\|Du_j\|_{L^p(\Qzero)}=1,
\qquad
\|\Edev u_j\|_{L^p(\Qzero)}\to 0.
\end{equation}
Since each $u_j$ is uniformly bounded in $W^{1,p}(\Qzero;\R^n)$, after passing to a subsequence, there exists $u\in L^p(\Qzero;\R^n)$ such that
$$
u_j\to  u\quad\text{in }L^p(\Qzero;\R^n).
$$
This compactness can be seen directly on the cube. Partition $\Qzero$ into dyadic subcubes $R$ of sidelength $h$ and let $A_hv$ equal the average of $v$ on each $R$. On one such cube, write $M_i$ for averaging in the $i$-th coordinate. Then
$$
I-M_1\cdots M_n =\sum_{i=1}^n M_1\cdots M_{i-1}(I-M_i).
$$
Each $M_i$ is an $L^p$-contraction, while the one-dimensional fundamental theorem of calculus and H\"older's inequality give
$$
\|(I-M_i)w\|_{L^p(R)} \le C(p)h\|\partial_iw\|_{L^p(R)}.
$$
Summing over the subcubes therefore yields
$$
\|v-A_hv\|_{L^p(\Qzero)} \le C(n,p)h\|Dv\|_{L^p(\Qzero)}.
$$
For fixed $h$, the range of $A_h$ is of finite dimension; hence bounded subsets of $W^{1,p}(\Qzero)$ are precompact in $L^p(\Qzero)$.

Apply \eqref{eq:ahlfors-garding} to $u_j-u_k$. Together with
\eqref{eq:ahlfors-null-sequence}, this gives
$$
\|D(u_j-u_k)\|_{L^p(\Qzero)}
\le C(n,p)\bigl(
\|u_j-u_k\|_{L^p(\Qzero)}
+\|\Edev u_j-\Edev u_k\|_{L^p(\Qzero)}
\bigr)\to 0.
$$
Thus $(u_j)$ is a Cauchy sequence in $W^{1,p}(\Qzero;\R^n)$ and converges strongly to $u$. By \Cref{lem:projection-bounded},
$\Pi_{\zeta,\Qzero}u=0$, while the last limit in \eqref{eq:ahlfors-null-sequence} gives $\Edev u=0$. Hence the conformal-Killing classification proved above gives $u\in\Kconf$. Since $\Pi_{\zeta,\Qzero}$ is the identity on $\Kconf$, we obtain $u=0$. This contradicts the normalization in \eqref{eq:ahlfors-null-sequence}. We have therefore proved
\begin{equation}\label{eq:ahlfors-unit-gap}
\|v\|_{L^p(\Qzero)}+\|Dv\|_{L^p(\Qzero)}
\le C(n,p,\zeta)\|\Edev v\|_{L^p(\Qzero)}
\end{equation}
whenever $\Pi_{\zeta,\Qzero}v=0$.

Finally, let $Q$ be an arbitrary cube and set $v=\mathcal T_Qu$. By
\eqref{eq:scaled-projection},
$\Pi_{\zeta,\Qzero}v=0$. The change of variables
$x=x_Q+\ell(Q)y$ gives
$$
\|v\|_{L^p(\Qzero)} =\ell(Q)^{-1-n/p}\|u\|_{L^p(Q)},
$$
and
$$
\|Dv\|_{L^p(\Qzero)}=\ell(Q)^{-n/p}\|Du\|_{L^p(Q)},
\qquad \|\Edev v\|_{L^p(\Qzero)}=\ell(Q)^{-n/p}\|\Edev u\|_{L^p(Q)}.
$$
Plugging these identities into \eqref{eq:ahlfors-unit-gap} and multiplying by $\ell(Q)^{n/p}$ prove \eqref{eq:projected-korn}.
\end{proof}

\begin{remark}
When $p=2$, let $P_{\Kconf}$ be the $L^2(\Qzero)$-orthogonal projection onto $\Kconf$ and set
$$
\lambda_A:=\inf\left\{
\|\Edev v\|_{L^2(\Qzero)}^2:
 v\in H^1(\Qzero;\R^n),\ P_{\Kconf}v=0,\ \|v\|_{L^2(\Qzero)}=1
\right\}.
$$
The same compactness argument used in the proof, together with
\eqref{eq:ahlfors-garding} at $p=2$ and the identity
$\ker\Edev=\Kconf$, shows that $\lambda_A>0$. This is the intrinsic spectral gap of the Ahlfors quadratic form. Its formal Euler--Lagrange operator is
$$
\mathscr L_Av=-\operatorname{div}(\Edev v)
=-\frac12\Delta v-\frac{n-2}{2n}\nabla\operatorname{div}v.
$$
The intrinsic gap transfers to the moment complement used here. Indeed, if
$v=v_\perp+k$, where $k=P_{\Kconf}v$, then
$\Pi_{\zeta,\Qzero}v=0$ gives
$k=-\Pi_{\zeta,\Qzero}v_\perp$. The boundedness in
\Cref{lem:projection-bounded}, the definition of $\lambda_A$, and
\eqref{eq:ahlfors-garding} control both $v_\perp$ and $k$ by
$\Edev v$.

We point out that, no boundary condition is imposed in \Cref{thm:projected-korn}. In the variational realization of $\mathscr L_A$, the natural boundary condition is
$(\Edev v)\nu=0$ wherever a classical conormal trace is available. Thus the $p=2$ case can be read as a spectral-gap estimate for the Ahlfors Laplacian defined above. For general $1<p<\infty$, the proof gives the corresponding closed-range, or coercivity-gap, estimate on the complement $\ker\Pi_{\zeta,Q}$.
\end{remark}

\begin{remark}
Dain~\cite[Theorem~1.1]{D2006} established the corresponding $L^2$-conformal Korn inequality on Lipschitz domains and observed that its kernel is precisely the space of conformal Killing vector fields. Subsequently, Lewintan and Neff~\cite{LN2021} derived a stronger trace-free Korn-type inequality for incompatible tensor fields, incorporating a generalized curl operator together with tangential boundary conditions, which is used in the above proof. Negative-norm coercivity techniques, which play a central role in many contemporary proofs of Korn-type inequalities, can be found, for example, in \cite{DL1976} and \cite{N2012}. 

However, among all references known to the author, Reshetnyak's theorem~\cite{R1970} is a most direct tool for obtaining the present projected estimate on cubes in $\R^n$. 
\end{remark}

\section{The local stability argument}\label{sec:local-stability}

Throughout this section, $n\ge3$ and $p>n$ are fixed.  Moreover, whenever a $K$-quasiregular mapping is under consideration, we write
$$
\delta:=K-1.
$$
The constants in this section are allowed to depend on $n$ and $p$, but not on the cube, the mapping, or $\delta$, unless another dependence is displayed explicitly.

\subsection{The orthogonal selection}
We first present the qualitative part of Reshetnyak's construction.  The notation $\Pi_{\zeta,Q}$ is defined by  \eqref{eq:scaled-projection}.

\begin{proposition}[{\cite[Chapter~4, Lemmas~2.12--2.13]{R1994}}] \label{prop:preliminary-gauge}
There exist a number
$$
\delta_{\mathrm{pre}}=\delta_{\mathrm{pre}}(n)>0,
$$
a number $\zeta_0=\zeta_0(n)\in(0,1)$, and two moduli of continuity
$$
\omega_0,\ \omega_1:[0,\delta_{\mathrm{pre}}]\to[0,\infty), 
$$
i.e. $\omega_0$ and $\omega_1$ are nondecreasing and
$$
\lim_{t\to 0}\omega_0(t)= \lim_{t\to 0}\omega_1(t)=0,
$$
so that the following statement holds: Let $Q=Q(x_0,r)$ be a cube with $2Q\subset\subset\Omega$, and $f:2Q\to\R^n$
be a nonconstant sense-preserving $K$-quasiregular mapping with $0<\delta\le\delta_{\mathrm{pre}}$.  Then there is a sense-preserving M\"obius transformation $\varphi$ such that
$$
g:=\varphi^{-1}\circ f
$$
is finite on $2Q$,
\begin{equation}\label{eq:preliminary-projection}
\Pi_{\zeta_0,Q}(g-\operatorname{id})=0,
\end{equation}
\begin{equation}\label{eq:preliminary-uniform}
\|g-\operatorname{id}\|_{L^\infty(2Q)}
\le r\,\omega_0(\delta),
\end{equation}
and
\begin{equation}\label{eq:preliminary-Ln}
\left(\fint_Q|Dg-\Id|^n\,dx\right)^{1/n}
\le\omega_1(\delta).
\end{equation}
In particular, up to decreasing $\delta_{\mathrm{pre}}(n)$ if necessary, the transformation $\varphi$ has no pole in $\frac32Q$.
\end{proposition}

\begin{proof}
This is the scaled content of Reshetnyak's preliminary normalization {\cite[Chapter~4, Lemmas~2.12--2.13]{R1994}}.  We explain the correspondence of the domains and the estimates.

Reshetnyak writes the inner cube as $Q(x_0,r)$ and the outer cube as $Q(x_0,2r)$.  His Lemma~2.13 produces a M\"obius transformation $\varphi$ for which $g=\varphi^{-1}\circ f$ satisfies the orthogonal condition \eqref{eq:preliminary-projection} on the inner cube, a uniform estimate on the outer cube, and an $L^n$-estimate for $Dg-\Id$ on the inner cube.  After the affine change of variables
$$
y=S_Q(x)=\frac{x-x_0}{r},
$$
all three estimates become independent of $x_0$ and $r$.  By absorbing the harmless numerical factors in the two moduli, we obtain both \eqref{eq:preliminary-uniform} and \eqref{eq:preliminary-Ln}.

Finally, the uniform estimate \eqref{eq:preliminary-uniform} tends to zero with $\delta$. In fact, Reshetnyak's pole-separation lemma~\cite[Chapter~4, Lemma~2.6]{R1994}, applied with the three concentric cubes $Q\subset\subset\frac32Q\subset\subset2Q$, shows that $\varphi$ is finite on $\frac32Q$ once $\delta_{\mathrm{pre}}(n)$ is sufficiently small.

For completeness, we show that the normalizer is necessarily sense preserving once
$\delta_{\mathrm{pre}}$ is small. Suppose instead that $\varphi$ reverses orientation.
Since $g=\varphi^{-1}\circ f$ is finite on $2Q$, the map $\varphi^{-1}$ is smooth on
a neighborhood of the compact set $f(\overline Q)$. Hence the Sobolev chain rule gives
$$
Dg(x)=D\varphi^{-1}(f(x))Df(x) \quad\text{for a.e. }x\in Q.
$$
As $f$ is sense preserving and $\varphi^{-1}$ reverses orientation, we have
$$
J_g(x)\le 0 \quad\text{for a.e. }x\in Q.$$
Moreover, for every matrix $A\in\R^{n\times n}$ with $\det A\le0$, one has
\begin{equation}\label{eq:negative-determinant-distance}
 |A-\Id|\ge1.
\end{equation}
Indeed, if $\|A-\Id\|_{\mathrm{op}}<1$, then
$$
 \Id+t(A-\Id)
$$
is invertible for every $t\in[0,1]$ by the Neumann-series criterion, and its determinant therefore has the constant positive sign of $\det\Id$, contradicting $\det A\le0$.  Since the Frobenius norm dominates the operator norm, \eqref{eq:negative-determinant-distance} follows. Now applying \eqref{eq:negative-determinant-distance} to $A=Dg(x)$ gives
$$
 \fint_Q|Dg-\Id|^n\,dx\ge1,
$$
which contradicts \eqref{eq:preliminary-Ln} up to decreasing $\delta_{\mathrm{pre}}$ so that $\omega_1(\delta_{\mathrm{pre}})\ll 1$.
\end{proof}

The proposition above cites Reshetnyak's completed preliminary normalization, including the orthogonal condition \eqref{eq:preliminary-projection} and the qualitative derivative estimate \eqref{eq:preliminary-Ln}.  The orthogonal condition itself is selected by a finite-dimensional argument.  As this selection is not very transparent in the original proof, we now reconstruct it in a self-contained manner. 

We remark that, the following lemma recovers only the finite-dimensional modulation step; it does not give the compactness argument that leads to \eqref{eq:preliminary-Ln}. Its structure is analogous to the finite-dimensional extremal selection in Figalli–Zhang~\cite[Lemma~4.1]{FZ2022}, although the functional employed here is specifically adapted to the M\"obius moment map.

\begin{lemma}\label{lem:moment-modulation}
Let $\Qzero=(-\frac12,\frac12)^n$ and  $\zeta\in(0,1)$.  There exist constants
$$
\rho=\rho(n,\zeta)>0, \qquad \eta=\eta(n,\zeta)>0, \qquad C=C(n,\zeta)>0
$$
with the following property.  If $h:2\Qzero\to\R^n$ is continuous and
\begin{equation}\label{eq:h-close-id}
\|h-\operatorname{id}\|_{L^\infty(2\Qzero)}\le\eta,
\end{equation}
then there is a neighborhood $\mathscr U_{\Mob}=\mathscr U_{\Mob}(n,\zeta)$ of the identity map in $\Mob^+(n)$ and a unique transformation
$\Theta_h\in\mathscr U_{\Mob} $
such that
\begin{equation}\label{eq:Theta-h-orthogonal}
\Pi_{\zeta,\Qzero}(\Theta_h\circ h-\operatorname{id})=0.
\end{equation}
Moreover, $\Theta_h$ has no pole in $3\Qzero$ and
\begin{equation}\label{eq:modulation-C2}
\|\Theta_h-\operatorname{id}\|_{C^2(3\Qzero)}
\le C\|h-\operatorname{id}\|_{L^\infty(2\Qzero)}.
\end{equation}
\end{lemma}

\begin{proof}
Let
$$
N:=\dim\Kconf=\frac{(n+1)(n+2)}2.
$$
Choose an orthonormal basis $X_1,\ldots,X_N$ of $\Kconf$ for the scalar product
$$
\langle X,Y\rangle_\zeta =\fint_{\zeta\Qzero}X(x)\cdot Y(x)\,dx.
$$
Since the sense-preserving M\"obius transformations form a finite-dimensional Lie group, the Lie algebra of $\Mob^+(n)$, identified through infinitesimal vector fields on
$\R^n$, is exactly $\Kconf$. Thus, after choosing the basis
$X_1,\ldots,X_N$, the exponential map gives local coordinates at the identity
whose coordinate derivatives at the origin are the fields $X_\alpha$.
Namely, there is a smooth coordinate chart
$$
\Psi:B_{2\rho}(0)\subset\R^N\to\Mob^+(n)
$$
for some $\rho>0$, such that
$$
\Psi(0)=\operatorname{id}, \qquad \left.\frac{\partial}{\partial a_\alpha}\Psi(a)(x)\right|_{a=0} =X_\alpha(x) \quad \text{ for } \ 1\le\alpha\le N.
$$
Then, up to decreasing $\rho$ if necessary, every $\Psi(a)$ with $|a|\le2\rho$ is finite on $4\Qzero$.  

Set
$\mathscr U_{\Mob}:=\Psi(B_\rho(0)).$
The map
$$
(a,x)\mapsto\Psi(a)(x)
$$
is smooth on $\overline{B_{2\rho}(0)}\times4\overline{\Qzero}$.  We also require $0<\eta\ll 1/2$ so that \eqref{eq:h-close-id} implies $
h(2\Qzero)\subset3\Qzero.$
Then all compositions below are well defined.

For a continuous map $h$ satisfying \eqref{eq:h-close-id}, define the moment map
$$
\mathcal N(a;h)
:=
\Pi_{\zeta,\Qzero}
\bigl(\Psi(a)\circ h-\operatorname{id}\bigr)
\in\Kconf.
$$
In the basis $X_1,\ldots,X_N$, its coordinates are
\begin{equation}\label{eq:moment-coordinates}
\mathcal N_\beta(a;h)
=
\fint_{\zeta\Qzero}
\bigl(\Psi(a)(h(x))-x\bigr)\cdot X_\beta(x)\,dx.
\end{equation}
At $(a,h)=(0,\operatorname{id})$,
$$
\mathcal N(0;\operatorname{id})=0.
$$
Differentiating \eqref{eq:moment-coordinates} at this point gives
$$
\frac{\partial\mathcal N_\beta}{\partial a_\alpha}
(0;\operatorname{id})
=
\fint_{\zeta\Qzero}X_\alpha(x)\cdot X_\beta(x)\,dx
=\delta_{\alpha\beta}.
$$
Thus
\begin{equation}
D_a\mathcal N(0;\operatorname{id})=I_N.
\end{equation}

The derivatives of $\Psi$ are uniformly continuous on the compact set introduced above.  Hence, after decreasing $\rho$ and then choosing $\eta>0$ sufficiently small, every $h$ satisfying \eqref{eq:h-close-id} also satisfies
\begin{equation}\label{eq:moment-derivative-close}
\sup_{|a|\le\rho}
\|D_a\mathcal N(a;h)-I_N\|_{\mathrm{op}}
\le\frac14.
\end{equation}
The same compactness gives
\begin{equation}\label{eq:moment-at-zero}
|\mathcal N(0;h)|
\le C_0\|h-\operatorname{id}\|_{L^\infty(2\Qzero)}
\end{equation}
for a constant $C_0=C_0(n,\zeta)$.

For $a,b\in\overline{B_\rho(0)}$, the line segment joining $a$ to $b$ lies in the same ball.  Then by the fundamental theorem of calculus and \eqref{eq:moment-derivative-close},
\begin{align}
|\mathcal N(a;h)-\mathcal N(b;h)|
&\ge |a-b| - \int_0^1
\|D_a\mathcal N(b+t(a-b);h)-I_N\|_{\mathrm{op}}
|a-b|\,dt\notag\\
&\ge\frac34|a-b|.
\label{eq:moment-lower-lipschitz}
\end{align}
In particular, $a\mapsto\mathcal N(a;h)$ is injective on $\overline{B_\rho(0)}$.

Now consider 
$$\mathcal J_h(a):=\frac12|\mathcal N(a;h)|^2.$$
This functional attains its minimum on the compact ball $\overline{B_\rho(0)}$.  For $|a|=\rho$, \eqref{eq:moment-lower-lipschitz} and \eqref{eq:moment-at-zero} yield
$$
|\mathcal N(a;h)|
\ge\frac34\rho-|\mathcal N(0;h)|.
$$
Choose $\eta>0$ even smaller so that $C_0\eta\le\rho/8$.  Then
$$
|\mathcal N(a;h)|\ge\frac58\rho
\quad\text{when }|a|=\rho,
\qquad
|\mathcal N(0;h)|\le\frac18\rho.
$$
Therefore, every minimizer $a_h$ lies in the open ball $B_\rho(0)$.  At such a minimizer,
$$
0=D\mathcal J_h(a_h)
=D_a\mathcal N(a_h;h)^T\mathcal N(a_h;h).
$$
According to \eqref{eq:moment-derivative-close}, the matrix $D_a\mathcal N(a_h;h)$ is invertible.  Hence $\mathcal N(a_h;h)=0,$
and the uniqueness follows from \eqref{eq:moment-lower-lipschitz}. This yields \eqref{eq:Theta-h-orthogonal}. 

Moreover, by taking $b=0$ there and using \eqref{eq:moment-at-zero}, we obtain
$$
|a_h|
\le\frac43|\mathcal N(0;h)|
\le C\|h-\operatorname{id}\|_{L^\infty(2\Qzero)}.
$$
Set $\Theta_h:=\Psi(a_h)$.  For every multi-index $\gamma$ with $|\gamma|\le2$, the fundamental theorem of calculus in the parameter variable gives
$$
 D_x^\gamma\Theta_h-D_x^\gamma\operatorname{id}
 =\int_0^1
 D_aD_x^\gamma\Psi(ta_h)[a_h]\,dt.
$$
Observe that the derivatives $D_aD_x^\gamma\Psi$ are bounded on the compact set $\overline{B_\rho(0)}\times3\overline{\Qzero}$.  Together with the bound for $|a_h|$, this proves \eqref{eq:modulation-C2}.  By construction of the chart, every $\Psi(a)$ with $|a|\le2\rho$ is finite on $4\Qzero$, and therefore $\Theta_h$ is pole-free on $3\Qzero$.
\end{proof}

\begin{remark}
For an arbitrary cube $Q$, one obtains the corresponding construction by conjugating the previous procedure with the affine transformation $S_Q$ defined in \eqref{eq:defn-SQ}. The scale covariance of $\Pi_{\zeta,Q}$, established in \eqref{eq:scaled-projection}, implies that the modulated mapping satisfies the prescribed orthogonal condition on $Q$. The lemma thus provides a quantitative refinement of the finite-dimensional correction step appearing in the proof of Reshetnyak’s  normalization lemma~\cite[Chapter~4, Lemma~2.12]{R1994}.  

In the proof of \Cref{thm:local-main} we continue to use \Cref{prop:preliminary-gauge} as the combined source of the orthogonal condition, the pole margin, and the qualitative $W^{1,n}$-smallness; \Cref{lem:moment-modulation} explains and independently verifies the selection of the conformal-Killing gauge.
\end{remark}

\subsection{The relative bounded-specific-oscillation gain}
In this subsection, we present a self-contained proof of the higher integrability results stated in \cite[Section 5, Chapter 3]{R1994}.

For a measurable field $H:R\to\R^d$ and $1<q<\infty$, we define its normalized $L^q$-mean by
\begin{equation}
\avnorm{H}{q}{R}
:=
\left(\fint_R|H|^q\,dx\right)^{1/q}.
\end{equation}
The following comparison estimates for M\"obius transformations are used  in the fixed concentric configuration
$$
R\subset\subset\frac32R\subset\subset2R.
$$

\begin{lemma} \label{lem:mobius-nested-derivatives}
Let $1<\kappa<2$ and $1\le q<\infty$.  There exists $C=C(n,q,\kappa)$ such that the following holds.  Let $R$ be a cube and let $\varphi$ be a M\"obius transformation with no pole in $\kappa R$.  Then
\begin{equation}\label{eq:mobius-mean-comparability}
 C^{-1}\avnorm{D\varphi}{q}{R}  \le \|D\varphi(x)\|_{\mathrm{op}} \le C \avnorm{D\varphi}{q}{R} \qquad \text{ for } \ x\in R,
\end{equation}
and
\begin{equation}\label{eq:mobius-D-lipschitz}
 \|D\varphi(x)-D\varphi(y)\|_{\mathrm{op}}  \le C \avnorm{D\varphi}{q}{R}\frac{|x-y|}{\ell(R)}
 \qquad   \text{ for } \ x,y\in R.
\end{equation}
\end{lemma}

\begin{proof}
If $\varphi(\infty)=\infty$, then
$$
 \varphi(x)=\mathbf{b}+\lambda Ox,  \qquad \mathbf{b}\in\R^n,  \quad \lambda>0,  \quad O\in O(n).
$$
In this case
$$
 D\varphi(x)=\lambda O,  \qquad  D^2\varphi(x)=0  \qquad  \text{ for any }\  x\in\R^n,
$$
and therefore
$$
 \|D\varphi(x)\|_{\mathrm{op}}=\lambda,  \qquad  \avnorm{D\varphi}{q}{R}=\sqrt n\,\lambda.
$$
Thus \eqref{eq:mobius-mean-comparability} holds with constants depending only on $n$, and the left-hand side of \eqref{eq:mobius-D-lipschitz} is zero.  

Suppose now that the pole $\mathbf{a}=\varphi^{-1}(\infty)$ is finite.  Then
\begin{equation}\label{eq:mobius-explicit-form}
 \varphi(x)=\mathbf{b}+\lambda O\frac{x-\mathbf{a}}{|x-\mathbf{a}|^2}
\end{equation}
for some $\mathbf{b}\in\R^n$, $\lambda>0$, and $O\in O(n)$. If $\mathbf{e}_x=\frac{x-\mathbf a}{|x-\mathbf a|},$
then differentiation of \eqref{eq:mobius-explicit-form} gives
\begin{equation*}
 D\varphi(x)  =\frac{\lambda}{|x-\mathbf{a}|^2}O(I_n-2\,\mathbf{e}_x\otimes \mathbf{e}_x).
\end{equation*}
 Hence
\begin{equation}\label{eq:mobius-D-norm}
 \|D\varphi(x)\|_{\mathrm{op}} = \frac{\lambda}{|x-\mathbf{a}|^2}.
\end{equation}
Additionally, another differentiation gives
\begin{equation}\label{eq:mobius-D2-bound}
 |D^2\varphi(x)| \le C(n)\frac{\lambda}{|x-\mathbf{a}|^3}.
\end{equation}

As $\varphi$ has no pole in $\kappa R$ for $\kappa>1$, we have $\mathbf{a}\notin\kappa R$ and hence
\begin{equation}\label{eq:lower-bound-dist}
     \dist(\mathbf{a},R) \ge \frac{\kappa-1}{2}\ell(R).
\end{equation}
 
Consequently, for any $x,y \in R$, the triangle inequality yields
$$
 |x-\mathbf{a}| \le |y-\mathbf{a}| + \operatorname{diam}(R) \le \left(1 + \frac{2\sqrt{n}}{\kappa - 1}\right)|y-\mathbf{a}|.
$$
Hence, the quantity $|x-\mathbf{a}|$, for $x \in R$, is uniformly comparable to any one of its values, with a multiplicative constant depending only on $n$ and $\kappa$. Thus, \eqref{eq:mobius-D-norm} implies \eqref{eq:mobius-mean-comparability}.

Moreover, the fundamental theorem of calculus, \eqref{eq:mobius-D2-bound}  and \eqref{eq:lower-bound-dist} give
$$
 \|D\varphi(x)-D\varphi(y)\|_{\mathrm{op}}
 \le |x-y|\sup_{z\in R}|D^2\varphi(z)|\le C(n,\kappa)\avnorm{D\varphi}{q}{R}\frac{|x-y|}{\ell(R)},
$$
which is \eqref{eq:mobius-D-lipschitz}.
\end{proof}

Now we have the following M\"obius comparison estimates.
\begin{lemma} \label{lem:mobius-comparison}
Let $1<q<\infty$.  There are constants
$$
 a_0=a_0(n)>0,  \qquad  C=C(n,q,\,\zeta_0)>0
$$
so that the following holds:  Let $R$ be a cube, $f:2R\to\R^n$ be continuous and belong to $W^{1,n}(2R;\R^n)$,   $\varphi$ be a M\"obius transformation finite on $\frac32R$, and set
$$
 g:=\varphi^{-1}\circ f=\operatorname{id}+u.
$$
Assume that $g$ is defined on $2R$ and
\begin{equation}\label{eq:comparison-hypotheses}
 \|u\|_{L^\infty(2R)}
 \le a_0\ell(R),
 \qquad
 \Pi_{\zeta_0,R}u=0.
\end{equation}
If $Df-D\varphi\in L^q(R)$, then $Dg-I_n\in L^q(R)$ and
\begin{equation}\label{eq:comparison-forward}
 \|Dg-I_n\|_{L^q(R)}
 \le
 C \avnorm{D\varphi}{q}{R}^{-1}
 \|Df-D\varphi\|_{L^q(R)}.
\end{equation}
Conversely, if $Dg-I_n\in L^q(R)$, then $Df-D\varphi\in L^q(R)$ and
\begin{equation}\label{eq:comparison-reverse}
 \|Df-D\varphi\|_{L^q(R)}
 \le
 C \avnorm{D\varphi}{q}{R}
 \|Dg-I_n\|_{L^q(R)}.
\end{equation}
\end{lemma}

\begin{proof}
We first prove the reverse estimate \eqref{eq:comparison-reverse}.  Assume that $Dg-I_n\in L^q(R)$.  Write $\ell=\ell(R)$.  By choosing $a_0$ sufficiently small, the uniform bound in \eqref{eq:comparison-hypotheses} implies
$$
 g(R)\subset\frac 5 4 R\subset\subset\frac 3 2 R.
$$
Then the pole of $\varphi$ is outside $\frac 3 2 R$, and we can apply \Cref{lem:mobius-nested-derivatives} first to the cube $\frac 5 4 R$, whose $\frac 6 5$-dilate is $\frac 3 2 R$.  

We claim that
$$
    \avnorm{D\varphi}{q}{\frac54R} \le C(n,q)\avnorm{D\varphi}{q}{R}.
$$
Indeed, according to \eqref{eq:mobius-mean-comparability}, both normalized means are comparable to the value of $\|D\varphi\|_{\mathrm{op}}$ at any fixed point of $R$.  Since $x,\,g(x)\in\frac 5 4 R$, we therefore obtain
\begin{equation}\label{eq:Dphi-g-comparable}
 \sup_{x\in R}\|D\varphi(g(x))\|_{\mathrm{op}}
 \le C(n,\,q) \avnorm{D\varphi}{q}{R}.
\end{equation}
Moreover, since $g=\operatorname{id}+u$, \eqref{eq:mobius-D-lipschitz} yields that, for each $x\in R$,
\begin{equation}
 \|D\varphi(g(x))-D\varphi(x)\|_{\mathrm{op}}
 \le C(n) \avnorm{D\varphi}{q}{R}\frac{|x-g(x)|}{\ell}\le  C(n,q) \avnorm{D\varphi}{q}{R}\frac{|u(x)|}{\ell}.
\end{equation}
In addition, the chain rule applied to $f=\varphi\circ g$ gives, almost everywhere in $R$,
\begin{equation}\label{eq:comparison-chain-rule}
 Df-D\varphi
 =D\varphi(g)(Dg-I_n)+D\varphi(g)-D\varphi.
\end{equation}
Thus, combining \eqref{eq:Dphi-g-comparable}--\eqref{eq:comparison-chain-rule} implies
\begin{align*}
 \|Df-D\varphi\|_{L^q(R)}
 &\le C(n,q) \avnorm{D\varphi}{q}{R}
 \left( \|Dg-I_n\|_{L^q(R)} +\ell^{-1}\|u\|_{L^q(R)}
 \right).
\end{align*}
Since translations belong to $\Kconf$, the condition $\Pi_{\zeta_0,R}u=0$ implies
$$
 \fint_{\zeta_0R}u\,dx=0.
$$
Thus Lemma~\ref{lem:inner-poincare} gives
$$
 \ell^{-1}\|u\|_{L^q(R)}  \le C(n,q,\zeta_0)\|Du\|_{L^q(R)}.
$$
This proves \eqref{eq:comparison-reverse} as $Du=Dg-\Id$.

On a separate note, \eqref{eq:comparison-forward} is a
  variant of Reshetnyak's comparison lemma \cite[Chapter~4, Lemma~2.7]{R1994}, with the fixed dilation parameters
$$
 k_1=\frac32,  \qquad  k_2=2,  \qquad \zeta=\zeta_0.
$$
Reshetnyak's estimate gives
$$
\|Dg-\Id\|_{L^q(R)}\le   C(n,q,\zeta_0) |R|^{1/q}\|D\varphi\|_{L^q(R)}^{-1} \|Df-D\varphi\|_{L^q(R)}.
$$
Since
$$
\avnorm{D\varphi}{q}{R} =|R|^{-1/q}\|D\varphi\|_{L^q(R)},
$$
this is precisely \eqref{eq:comparison-forward}.
\end{proof}

\begin{definition}[Admissible comparison family]\label{def:comparison-family}
Fix $1<q<\infty$.  For every cube $R\subset\R^n$, let
$$
 \mathscr S(R)\subset C(R;\R^d)
$$
be a family of nonzero continuous vector fields.  We call $\mathscr S$ an \emph{admissible $L^q$-comparison family} if there are constants
$$
 0<\alpha_1\le1\le\alpha_2<\infty,
 \qquad
 1\le\alpha_3<\infty,
$$
independent of $R$, such that the following properties hold.
\begin{enumerate}[label=\textup{(S\arabic*)},leftmargin=2.7em]
\item If $R'\subset R$ is a cube and $\phi\in\mathscr S(R)$, then $\phi|_{R'}\in\mathscr S(R')$.
\item For every $\phi\in\mathscr S(R)$ and every $x\in R$,
\begin{equation}\label{eq:comparison-pointwise-size}
 \alpha_1\avnorm{\phi}{q}{R}
 \le |\phi(x)|
 \le\alpha_2\avnorm{\phi}{q}{R}.
\end{equation}
\item For every $\phi,\psi\in\mathscr S(R)$ and every $x\in R$,
\begin{equation*}
 |\phi(x)-\psi(x)|
 \le\alpha_3\avnorm{\phi-\psi}{q}{R}.
\end{equation*}
\end{enumerate}
\end{definition}

\begin{definition}[Relative specific oscillation]
Let $U\subset\R^n$ be open, $F:U\to\R^d$ be a measurable vector field, and $\mathscr S$ be an admissible $L^q$-comparison family.  We say that $F$ has \emph{relative $L^q$-specific oscillation at most $w$} if, for every cube $R\subset\subset U$, there is a vector field $\phi_R\in\mathscr S(R)$ such that
\begin{equation}\label{eq:relative-oscillation-def}
 \avnorm{F-\phi_R}{q}{R}
 \le w\,\avnorm{\phi_R}{q}{R}.
\end{equation}
The vector field $\phi_R$ is called a comparison field for $F$ on $R$.
\end{definition}

The following dyadic stopping lemma is a covering result needed in the self-improvement argument.  We record its proof so as to avoid invoking Calder\'on--Zygmund decomposition. 
Recall that a dyadic cube $R$ is of the form
$$
R=2^{-k}([0,1]^n+m), \qquad k\in\mathbb Z,\quad m\in\mathbb Z^n.
$$
A \emph{dyadic child} of $R$ is any of the $2^n$ smaller subcubes
obtained by bisecting each side of $R$.

\begin{lemma}\label{lem:weighted-dyadic-stopping}
Let $R_0$ be a cube, and let $u,v\in L^1(R_0)$ be nonnegative.  Assume that
\begin{equation}\label{eq:stopping-global-strict}
 \int_{R_0}u\,dx<\int_{R_0}v\,dx
\end{equation}
and that there exists $D\ge1$ such that, whenever $R'$ is a dyadic child of a dyadic subcube $R\subset R_0$,
\begin{equation}\label{eq:dyadic-weight-doubling}
 \int_Rv\,dx\le D\int_{R'}v\,dx.
\end{equation}
Then there is an at most countable family of pairwise interior-disjoint dyadic subcubes $\{R_j\}$ of $R_0$ such that
\begin{equation}\label{eq:stopping-two-sided}
 \int_{R_j}v\,dx  \le\int_{R_j}u\,dx  \le D\int_{R_j}v\,dx
\end{equation}
for every $j$, and
\begin{equation}\label{eq:outside-stopping}
 u(x)\le v(x)
\end{equation}
for almost every $x\in R_0\setminus\bigcup_jR_j$.
\end{lemma}

\begin{proof}
Select the maximal dyadic subcubes $R_j\subset R_0$ for which
$$
 \int_{R_j}u\,dx\ge\int_{R_j}v\,dx.
$$
The strict inequality \eqref{eq:stopping-global-strict} ensures that $R_0$ itself is not selected.  Observe that the maximal dyadic cubes have pairwise disjoint interiors.

Let $\widetilde R_j$ be the dyadic parent of $R_j$.  By maximality of $R_j$, we have
$$
 \int_{\widetilde R_j}u\,dx
 <\int_{\widetilde R_j}v\,dx.
$$
Consequently, using $R_j\subset\widetilde R_j$ and \eqref{eq:dyadic-weight-doubling},
$$
 \int_{R_j}u\,dx
 \le\int_{\widetilde R_j}u\,dx
 <\int_{\widetilde R_j}v\,dx
 \le D\int_{R_j}v\,dx.
$$
This proves \eqref{eq:stopping-two-sided}.

If $x$ does not belong to any selected cube, then every dyadic cube in the nested chain shrinking to $x$ fails the stopping inequality.  The dyadic differentiation theorem therefore gives
$$
 u(x)\le v(x)
$$
for almost every such $x$.  This proves \eqref{eq:outside-stopping}.
\end{proof}

The next theorem is the relative power self-improvement used in the local stability argument.  Its exponent range and comparison-field form correspond to Reshetnyak's general theorem \cite[Chapter~3, Theorem~5.2]{R1994} and to the estimates \cite[Chapter~4, equations~(3.3)--(3.5)]{R1994}, which form a John--Nirenberg-type inequality for BMO functions; see also \cite[Chapter 7.1.2]{G2004}.  We include a complete dyadic proof in the present notation.

\begin{theorem}\label{thm:relative-gurov}
Let $1<q<\infty$, let $\mathscr S$ be an admissible $L^q$-comparison family with constants $\alpha_1,\alpha_2,\alpha_3$, and let $F:U\to\R^d$ have relative $L^q$-specific oscillation at most $w$.  
Fix a comparison field $\phi_R\in\mathscr S(R)$ satisfying \eqref{eq:relative-oscillation-def} for every cube $R\subset\subset U$.  For each   such cube $R$, define
\begin{equation}\label{eq:defn-ERt}
 E_R(t) :=\{x\in R:|F(x)-\phi_R(x)|\ge t|\phi_R(x)|\},
 \qquad t>0.
\end{equation}
There exist constants
$$
 w_0>0,  \qquad c_0>0, \qquad C_0<\infty,
$$
depending only on $n,q,\alpha_1,\alpha_2,\alpha_3$, such that, for $0<w\le w_0$, 
\begin{equation}\label{eq:relative-tail-decay}
 |E_R(t)| \le \frac{C_0}{w^q(1+t)^{c_0/w}\,\avnorm{\phi_R}{q}{R}^q} \int_R|F-\phi_R|^q\,dx
\end{equation}
whenever
\begin{equation}\label{eq:tail-threshold}
 t\ge C_0w.
\end{equation}

Moreover, fix a cube $R_0\subset\subset U$.  If comparison fields $\phi_S\in\mathscr S(S)$ satisfying \eqref{eq:relative-oscillation-def} are prescribed only for $R_0$ and its dyadic subcubes $S\subset R_0$, then \eqref{eq:relative-tail-decay} holds with $R=R_0$.  In particular, no comparison field outside this dyadic family is needed.
\end{theorem}

\begin{proof}
Fix the comparison field $\phi_R$ chosen in the theorem.
For $t>0$, define  
\begin{equation}\label{eq:defn-Pt}
P(t):=\sup_{R\subset\subset U} \frac{\avnorm{\phi_R}{q}{R}^q|E_R(t)|}
 {\displaystyle\int_R|F-\phi_R|^q\,dx}.
\end{equation}
If the denominator vanishes, then $F=\phi_R$ almost everywhere on $R$ and $E_R(t)$ is null; thus the corresponding quotient is defined to be zero and the conclusion of the theorem holds trivially. Thus, we may assume that the denominator is strictly positive and  $P(t)$ is well defined. 

\medskip
\noindent{\bf Step 1: An Estimate on $P$.} 
By \eqref{eq:comparison-pointwise-size} and \eqref{eq:defn-ERt}, for each $R\subset\subset U$ on $E_R(t)$ we have
$$
 |F-\phi_R|^q \ge t^q|\phi_R|^q  \ge(\alpha_1t)^q\avnorm{\phi_R}{q}{R}^q.
$$
Therefore,
$$\int_R|F-\phi_R|^q\,dx\ge (\alpha_1t)^q\avnorm{\phi_R}{q}{R}^q |E_R(t)|$$
and hence \eqref{eq:defn-Pt} gives
\begin{equation}\label{eq:P-chebyshev}
 P(t)\le    (\alpha_1t)^{-q}  \qquad \text{ for any }\ t>0.
\end{equation}

\medskip
\noindent{\bf Step 2: A decomposition $\{R_j\}$ of an arbitrary cube $R_0$.} 
Fix a cube $R_0\subset\subset U$, write
$$
 \phi_0:=\phi_{R_0}, \qquad
 f_0:=F-\phi_0, \qquad  m_0:=\avnorm{\phi_0}{q}{R_0},
$$
and choose a number $k>1$, to be fixed below.  We verify the conditions of \Cref{lem:weighted-dyadic-stopping} on $R_0$ are satisfied by
$$
 u:=|f_0|^q,
 \qquad
 v:=(kw)^q|\phi_0|^q.
$$
In fact, the global strict inequality \eqref{eq:stopping-global-strict} follows from \eqref{eq:relative-oscillation-def} since $k>1$. Thus it suffices to show \eqref{eq:dyadic-weight-doubling}.  If $R'$ is a dyadic child of $R$, then property \textup{(S1)} of \Cref{def:comparison-family} allows us to regard $\phi_0|_R$ as an element of $\mathscr S(R)$, and \eqref{eq:comparison-pointwise-size} gives
\begin{equation*}
 \int_{R'}|\phi_0|^q\,dx
 \ge  \alpha_1^q\avnorm{\phi_0}{q}{R}^q|R'| =2^{-n}\alpha_1^q
 \int_R|\phi_0|^q\,dx.
\end{equation*}
Thus \eqref{eq:dyadic-weight-doubling} holds with
\begin{equation*}
 D:=2^n\alpha_1^{-q}.
\end{equation*}
Thus we can apply \Cref{lem:weighted-dyadic-stopping} to $R_0$. 

Let $\{R_j\}$ be the resulting stopping cubes obtained by \Cref{lem:weighted-dyadic-stopping}.  Then \eqref{eq:stopping-two-sided} gives
\begin{equation}\label{eq:stopping-e0}
 (kw)^q\int_{R_j}|\phi_0|^q\,dx
 \le\int_{R_j}|f_0|^q\,dx
 \le D(kw)^q\int_{R_j}|\phi_0|^q\,dx.
\end{equation}
For each $j$, let $\phi_j:=\phi_{R_j}$ be a comparison field for $F$ on $R_j$, and write
$$
 f_j:=F-\phi_j, \qquad m_j:=\avnorm{\phi_j}{q}{R_j},
 \qquad  m_{0,j}:=\avnorm{\phi_0}{q}{R_j}.
$$
From \eqref{eq:stopping-e0},
\begin{equation}\label{eq:e0-local-mean}
 \avnorm{f_0}{q}{R_j}
 \le D^{1/q}kw\,m_{0,j}.
\end{equation}
Moreover, the triangle inequality, \eqref{eq:relative-oscillation-def}, and \eqref{eq:e0-local-mean} yield
\begin{align*}
 m_{0,j} &\le m_j+\avnorm{\phi_0-\phi_j}{q}{R_j}\\
 &\le m_j+\avnorm{f_0}{q}{R_j}+\avnorm{f_j}{q}{R_j} \le(1+w)m_j+D^{1/q}kw\,m_{0,j}.
\end{align*}
Choose $0<w_0\ll 1$ small, depending on $k$,  so that
\begin{equation*}
 w_0\le1, \qquad
 D^{1/q}kw_0\le\frac12.
\end{equation*}
Then
\begin{equation}\label{eq:m0j-mj}
 m_{0,j}\le4m_j.
\end{equation}
Now Property \textup{(S3)} of \Cref{def:comparison-family}, together with the triangle inequality and \eqref{eq:e0-local-mean}--\eqref{eq:m0j-mj}, gives, for every $x\in R_j$,
\begin{align*}
 |\phi_j(x)-\phi_0(x)|
 &\le\alpha_3\avnorm{\phi_j-\phi_0}{q}{R_j}\\
 &\le\alpha_3\bigl(\avnorm{f_j}{q}{R_j}+\avnorm{f_0}{q}{R_j}\bigr)\\
 &\le\alpha_3\bigl(wm_j+D^{1/q}k w m_{0,j}\bigr)\\
 &\le A_0w\,m_j,
\end{align*}
where
\begin{equation}\label{eq:defn-A0}
 A_0:=\alpha_3\bigl(1+4D^{1/q}k\bigr).
\end{equation}
By the lower bound in \eqref{eq:comparison-pointwise-size}, since $m_j:=\avnorm{\phi_j}{q}{R_j}$, we have
\begin{equation}\label{eq:relative-comparison-fields}
 |\phi_j(x)-\phi_0(x)| \le\eta_w|\phi_j(x)|, \qquad  \eta_w:=\frac{A_0}{\alpha_1}w.
\end{equation}
One may further decrease $w_0$ (depending on $k$) so that $\eta_w\le\frac 1 6$ whenever $w\le w_0$.

\medskip
\noindent{\bf Step 3: One iteration step.} 
Let $t>0$ and set
\begin{equation*}
 t_1:=(1-\eta_w)t-\eta_w.
\end{equation*}
If $x\in E_{R_0}(t)\cap R_j$, then \eqref{eq:defn-ERt} and \eqref{eq:relative-comparison-fields} gives
\begin{align*}
 |F(x)-\phi_j(x)|
 &\ge |F(x)-\phi_0(x)|-|\phi_j(x)-\phi_0(x)|\\
 &\ge t|\phi_0(x)|-\eta_w|\phi_j(x)|\\
 &\ge \bigl((1-\eta_w)t-\eta_w\bigr)|\phi_j(x)|\\
 &= t_1|\phi_j(x)|.
\end{align*}
Therefore, whenever $t_1>0$,
\begin{equation}\label{eq:tail-inclusion}
 E_{R_0}(t)\cap R_j\subset E_{R_j}(t_1).
\end{equation}
On the other hand, outside the stopping cubes, \eqref{eq:outside-stopping} gives
$$
 |F-\phi_0|^q=|f_0|^q\le(kw)^q|\phi_0|^q.
$$
Consequently, if $t>kw$, then $E_{R_0}(t)$ is contained, up to a null set, in $\bigcup_jR_j$.  Thus, according to  \eqref{eq:tail-inclusion}, the definition of $P(t_1)$ in \eqref{eq:defn-Pt} , and \eqref{eq:relative-oscillation-def}, we obtain
\begin{equation}\label{eq:tail-before-sum}
|E_{R_0}(t)|\le \sum_j|E_{R_j}(t_1)| \le P(t_1)\sum_j
m_j^{-q}\int_{R_j}|f_j|^q\,dx \le P(t_1)w^q\sum_j|R_j|.
\end{equation}
The lower stopping inequality \eqref{eq:stopping-two-sided} on $u=|f_0|^q$ and $v=(kw)^q |\phi_0|^q$, together with  the pointwise lower bound in Property  \textup{(S2)} of \Cref{def:comparison-family}, for $\phi_0$ on the original cube $R_0$ give
\begin{equation*}
 \int_{R_0}|f_0|^q\,dx \ge(kw)^q\sum_j\int_{R_j}|\phi_0|^q\,dx \ge(kw)^q\alpha_1^qm_0^q\sum_j|R_j|.
\end{equation*}
Substitution in \eqref{eq:tail-before-sum} yields
$$
 |E_{R_0}(t)| \le  \frac{P(t_1)}{\alpha_1^qk^qm_0^q}  \int_{R_0}|f_0|^q\,dx.
$$
Since $R_0\subset\subset U$ was arbitrary,
\begin{equation*}
 P(t)\le\frac1{\alpha_1^qk^q}P(t_1)
\end{equation*}
whenever $t>kw$ and $t_1>0$.  
Choose
\begin{equation*}
 k:=\frac{4^{1/q}}{\alpha_1}.
\end{equation*}
Then
\begin{equation}\label{eq:P-quarter-recurrence}
 P(t)\le\frac14P\bigl((1-\eta_w)t-\eta_w\bigr).
\end{equation}
By the definition  of $A_0$ \eqref{eq:defn-A0} and the inequalities $\alpha_1\le1\le\alpha_3$, one has $k\le2A_0/\alpha_1$ and therefore the definition of $\eta_w$ in \eqref{eq:relative-comparison-fields} yields $kw\le2\eta_w$.  
Set
$$
 u_0:=3\eta_w.
$$
Then $u_0>kw$ and
$$
 (1-\eta_w)u_0-\eta_w
 =2\eta_w-3\eta_w^2>0.
$$
Consequently, \eqref{eq:P-quarter-recurrence} is valid whenever $t\ge u_0$.

\medskip
\noindent{\bf Step 4: Iteration on $t$.} 
Set
$$
 T(t):=(1-\eta_w)t-\eta_w 
$$
so that the identity
\begin{equation}\label{eq:iteration-T}
 1+T(t)=(1-\eta_w)(1+t)
\end{equation}
allows an explicit iteration.  Define
$$
 1+u_m:=(1+u_0)(1-\eta_w)^{-m} \qquad \text{ for any }\  m\ge 0.
$$
If $u_m\le t<u_{m+1}$, then by \eqref{eq:iteration-T}
$$T^{(m)}(t)=\underbrace{T \circ\cdots\circ T}_{\text{$m$ times}}(t)\ge u_0=3\eta_w$$
and $T^{(m+1)}(t)<u_0$. Iterating \eqref{eq:P-quarter-recurrence} $m$ times and then applying \eqref{eq:P-chebyshev} gives
\begin{equation}\label{eq:P-after-iteration}
 P(t) \le4^{-m}P(T^{(m)}(t)) \le C(\alpha_1,\,q)\eta_w^{-q}4^{-m}.
\end{equation}

Let
$$
 \beta_w:=\frac{\log4}{-\log(1-\eta_w)}.
$$
Since $t<u_{m+1}$,
$$
 1+t<1+u_{m+1}=(1+u_0)(1-\eta_w)^{-(m+1)}.
$$
Raising this inequality to the power $-\beta_w$ and using
$ (1-\eta_w)^{\beta_w}=\frac1 4,$
we obtain
$$
 4^{-m}\le4(1+u_0)^{\beta_w}(1+t)^{-\beta_w}.
$$
Since $u_0=3\eta_w$ and $\eta_w\le\frac16$, the quantity
$(1+u_0)^{\beta_w}$ is bounded by an absolute constant: indeed, note that
$$
 \beta_w\log(1+3\eta_w)= \frac{\log4}{-\log(1-\eta_w)}\log(1+3\eta_w) \le \frac{\log4}{\eta_w}\,3\eta_w =3\log4,
$$
where we used 
$$\log(1+3\eta_w)\le3\eta_w \quad \text{ and }\quad  -\log(1-\eta_w)\ge\eta_w.$$ 
Furthermore, since $0<\eta_w<\frac 1 6$,
$$
 -\log(1-\eta_w)\le2\eta_w,
$$
so
$$
 \beta_w\ge\frac{\log4}{2\eta_w}
 =\frac{c_0}{w},
 \qquad
 c_0:=\frac{\alpha_1\log4}{2A_0}.
$$
Together with \eqref{eq:P-after-iteration} and $\eta_w=(A_0/\alpha_1)w$, this proves
$$
 P(t)\le C_0w^{-q}(1+t)^{-c_0/w}
$$
for $t\ge u_0=3\eta_w$.  By the definition of $P(t)$ in \eqref{eq:defn-Pt}, this is exactly \eqref{eq:relative-tail-decay}, up to increasing $C_0$ so that \eqref{eq:tail-threshold} implies $t\ge3\eta_w$.

\medskip
\noindent{\bf Step 5: The final comparison.} 
We finally verify the last claim in the theorem.  Fix a  cube $R\subset\subset U$ and define
$$
 P_R(t) :=\sup_{S} \frac{\avnorm{\phi_S}{q}{S}^q|E_S(t)|} {\displaystyle\int_S|F-\phi_S|^q\,dx},
$$
where the supremum is taken only over the dyadic subcubes $S\subset R$, and where a zero denominator again contributes the value zero.  

The proof of \eqref{eq:P-chebyshev} uses only Property \textup{(S2)} of Definition~\ref{def:comparison-family}, i.e. \eqref{eq:comparison-pointwise-size}, on the cube $S$, so it gives
$$
 P_R(t)\le(\alpha_1 t)^{-q}.
$$
Starting from any dyadic subcube $S\subset R$, every stopping cube selected by \Cref{lem:weighted-dyadic-stopping} is a dyadic descendant of $S$ and hence of $R$.  The comparison field attached to each selected cube is therefore among the prescribed fields.  Repeating lines \eqref{eq:stopping-e0}--\eqref{eq:tail-before-sum} with $S$ in place of $R_0$ gives
$$
 P_R(t)\le\frac14P_R((1-\eta_w)t-\eta_w)
$$
under the same restrictions on $t$ and $w$.  The explicit iteration \eqref{eq:P-after-iteration} now uses only $P_R$.  Taking $S=R$ proves the claimed tail estimate on the terminal cube without any hypothesis outside its dyadic descendants.
\end{proof}

Now we state the following self-improvement property. 
\begin{corollary}\label{cor:relative-bso-power}
Under the assumptions of \Cref{thm:relative-gurov}, let $s>q$.  Then there are constants $c_1>0$ and $C_{\rm self}<\infty$, depending only on $n,q,s,\alpha_1,\alpha_2,\alpha_3$, such that, whenever
\begin{equation}\label{eq:s-range-gurov}
 0<w\le w_0,
 \qquad
 s\le\frac{c_1}{w},
\end{equation}
one has
\begin{equation}\label{eq:power-self-improvement}
 \int_R|F-\phi_R|^s\,dx
 \le C_{\rm self} w^{s-q}\avnorm{\phi_R}{q}{R}^{s-q}
 \int_R|F-\phi_R|^q\,dx.
\end{equation}
\end{corollary}

\begin{proof}
Recall that $E_R(t)=\{x\in R\colon |F(x)-\phi_R(x)|\ge t|\phi_R(x)|\}$. 
Write
$$
  \psi:=F-\phi_R,
 \qquad m:=\avnorm{\phi_R}{q}{R},
 \qquad \mathcal F(x):=\frac{|\psi(x)|}{|\phi_R(x)|}.
$$
The lower bound in \eqref{eq:comparison-pointwise-size} ensures that $|\phi_R(x)|\ge\alpha_1m>0$ on $R$.  By the layer-cake formula,
\begin{equation}\label{eq:integral-R}
 \int_R\mathcal F^s\,dx =s\int_0^\infty t^{s-1}|E_R(t)|\,dt.
\end{equation}
Let $t_w:=C_0w$, where $C_0$ is the constant in \Cref{thm:relative-gurov}.

For $0<t<t_w$, Chebyshev's inequality and $|\phi_R|\ge\alpha_1m$ give
$$
 |E_R(t)| \le \frac1{(\alpha_1tm)^q}  \int_R|\psi|^q\,dx.
$$
Therefore, as $t_w=C_0w$, 
\begin{align}
 s\int_0^{t_w}t^{s-1}|E_R(t)|\,dt
 &\le \frac{s}{\alpha_1^qm^q} \int_R|\psi|^q\,dx \int_0^{t_w}t^{s-q-1}\,dt\notag\\
 &\le C(C_0, \alpha_1, s, q) w^{s-q}m^{-q} \int_R|\psi|^q\,dx.
 \label{eq:low-tail-integral}
\end{align}
For $t\ge t_w$, \eqref{eq:relative-tail-decay} gives
\begin{align}
 s\int_{t_w}^{\infty}t^{s-1}|E_R(t)|\,dt
 &\le
 \frac{C(C_0, s)}{w^qm^q}
 \int_R|\psi|^q\,dx \int_0^\infty t^{s-1}(1+t)^{-c_0/w}\,dt.
 \label{eq:high-tail-start}
\end{align}
Choose $c_1\le c_0/4$ with $c_0$ given by Theorem~\ref{thm:relative-gurov}, and put 
$$ b:=\frac{c_0}{w}.$$
Then \eqref{eq:s-range-gurov} implies $b\ge 4 s$.  Since $s>q>1$, it follows that
\begin{equation}\label{eq:lower-b}
     b-s-1\ge\frac b2>0.
\end{equation}

Now in the last integral in \eqref{eq:high-tail-start}, applying the change of variables
$$
  t=\frac{\tau}{1-\tau},  \qquad
  1+t=\frac{1}{1-\tau},\qquad
 dt=(1-\tau)^{-2}\,d\tau,
$$
it follows that
\begin{align*}
 \int_0^\infty t^{s-1}(1+t)^{-b}\,dt=\int_0^1\tau^{s-1}(1-\tau)^{b-s-1}\,d\tau:=\mathcal B(s,b-s).
\end{align*}
The elementary inequality $1-t\le e^{-t}$ for $0\le t<1$ together with \eqref{eq:lower-b} gives
\begin{align*}
 \mathcal B(s,b-s)
 &\le\int_0^1 t^{s-1}e^{-(b-s-1)t}\,dt \le\int_0^\infty t^{s-1}e^{-(b-s-1)t}\,dt\\
 &=\Gamma(s)(b-s-1)^{-s} \le 2^s\Gamma(s)b^{-s} 
 \le C(c_0,s) w^s,
\end{align*}
where the Gamma function
$$\Gamma(s)=\int_0^\infty t^{s-1}e^{-t}\,dt=(b-s-1)^s\int_0^\infty t^{s-1}e^{-(b-s-1)t}\,dt $$
is finite since $s>0$.
Substitution in \eqref{eq:high-tail-start}, together with \eqref{eq:integral-R} and \eqref{eq:low-tail-integral}, yields
\begin{equation}\label{eq:ratio-power-bound}
 \int_R\mathcal F^s\,dx =  s\int_0^\infty t^{s-1}|E_R(t)|\,dt   \le C(c_0, C_0, s, q,\alpha_1)w^{s-q}m^{-q}  \int_R|\psi|^q\,dx.
\end{equation}
Finally, the upper bound in \eqref{eq:comparison-pointwise-size} gives
$$
 |\psi|^s=\mathcal F^s|\phi_R|^s  \le\alpha_2^sm^s\mathcal F^s.
$$
Multiplying \eqref{eq:ratio-power-bound} by $\alpha_2^sm^s$ proves \eqref{eq:power-self-improvement}.
\end{proof}

We next show that the gradients of M\"obius transformations are comparison fields. 
\begin{lemma}\label{lem:mobius-comparison-family}
Fix $1<q<\infty$.  For a cube $R$, let $\mathscr S_q(R)$ be the family of fields $D\psi|_R$, where $\psi$ is a M\"obius transformation having no pole in $\frac32R$.  Then $\mathscr S_q$ is an admissible $L^q$-comparison family.  Its constants depend only on $n$ and $q$.
\end{lemma}

\begin{proof}
We verify the three conditions in \Cref{def:comparison-family}.

Let $R'\subset R$ be a cube and let $D\psi|_R\in\mathscr S_q(R)$.  Note that the inclusion $R'\subset R$ gives
$$
 |x_{R'}-x_R|_\infty
 \le\frac{\ell(R)-\ell(R')}{2},
$$
where $|\cdot|_\infty$ denotes the $\ell^\infty$-norm. 
Hence, for $x\in\frac32R'$, the triangle inequality yields
\begin{align*}
 |x-x_R|_\infty
 &\le |x-x_{R'}|_\infty+|x_{R'}-x_R|_\infty\\
 &<\frac 3 4\ell(R')+\frac{\ell(R)-\ell(R')}{2}\\
 &=\frac 1 2\ell(R)+\frac14\ell(R')
 \le\frac 3 4\ell(R).
\end{align*}
Therefore $\frac3 2R'\subset\frac 3 2 R$.  As the pole of $\psi$ is outside $\frac3 2 R$, it is also outside $\frac 3 2 R'$. Thus, $D\psi|_{R'}\in\mathscr S_q(R')$.  This proves Property \textup{(S1)} of \Cref{def:comparison-family}.

Applying \Cref{lem:mobius-nested-derivatives} with $\kappa=\frac32$,  we conclude that, for every $x\in R$,
$$
 C(n)^{-1}\avnorm{D\psi}{q}{R}  \le \|D\psi(x)\|_{\mathrm{op}} \le C(n) \avnorm{D\psi}{q}{R}.
$$
For any conformal linear map $L:\R^n\to\R^n$,
$$
 |L|=\sqrt n\,\|L\|_{\mathrm{op}}.
$$
Since the derivative of a M\"obius transformation is conformal at every finite point, multiplication by $\sqrt n$ converts the preceding estimate into
$$
 \alpha_1\avnorm{D\psi}{q}{R}
 \le |D\psi(x)|
 \le \alpha_2\avnorm{D\psi}{q}{R},
$$
where $\alpha_1,\alpha_2$ depend only on $n$ and $q$.  This is Property \textup{(S2)} of \Cref{def:comparison-family}.

Let $D\psi_1|_R,D\psi_2|_R\in\mathscr S_q(R)$. Applied to the nested cubes $R\subset\subset \frac32R$, Reshetnyak's difference estimate \cite[Chapter~4, Lemma~2.5]{R1994} gives
$$
 |D\psi_1(x)-D\psi_2(x)|
 \le \frac{C(n)}{|R|}\int_R|D\psi_1-D\psi_2|\,dy
 \qquad \text{for any} \ x\in R.
$$
H\"older's inequality yields
$$
 \frac1{|R|}\int_R|D\psi_1-D\psi_2|\,dy
 \le
 \left(\fint_R|D\psi_1-D\psi_2|^q\,dy\right)^{1/q}.
$$
Consequently
$$
 |D\psi_1(x)-D\psi_2(x)|
 \le C(n,q)\avnorm{D\psi_1-D\psi_2}{q}{R},
$$
which is Property  \textup{(S3)} of \Cref{def:comparison-family}.
\end{proof}

\begin{proposition}\label{prop:specialized-bso-gain}
Fix $p>n$.  There exist
$$
\epsilon=\epsilon(n,p)\in(0,1),
\qquad
\delta_{\mathrm{gain}}=\delta_{\mathrm{gain}}(n,p)>0,
$$
and a nondecreasing function
$$
\nu:[0,\delta_{\mathrm{gain}}]\to[0,\infty),
\qquad
\lim_{t\to 0}\nu(t)=0,
$$
such that the following holds.  Under the assumptions of \Cref{prop:preliminary-gauge}, if $0<\delta\le\delta_{\mathrm{gain}}$ and $u=g-\operatorname{id}$, then
\begin{equation}\label{eq:specialized-gain}
\int_Q|Du|^{p+\epsilon}\,dx
\le
\nu(\delta)
\int_Q|Du|^p\,dx.
\end{equation}
In particular, both integrals in \eqref{eq:specialized-gain} are finite.
\end{proposition}

\begin{proof}
 Let $R\subset Q$ be any cube with sides parallel to those of $Q$.  Then
\begin{equation}\label{eq:double-subcube-inclusion}
2R\subset2Q.
\end{equation}
Indeed, if $x_R$ and $x_Q$ are the centers of $R$ and $Q$, the inclusion $R\subset Q$ gives
$$
|x_R-x_Q|_\infty
\le\frac{\ell(Q)-\ell(R)}2.
$$
For $x\in2R$,
$$
|x-x_Q|_\infty
\le |x-x_R|_\infty+|x_R-x_Q|_\infty
<\ell(R)+\frac{\ell(Q)-\ell(R)}2
\le\ell(Q),
$$
which is precisely \eqref{eq:double-subcube-inclusion}.

We construct the comparison map on each subcube, without extending $f$ beyond   $2Q$.  If $R=Q$, use the M\"obius map $\varphi_Q:=\varphi$ already supplied by \Cref{prop:preliminary-gauge}.  If $R\subsetneq Q$, then $\ell(R)<\ell(Q)$ and the preceding coordinate estimate is strict; hence
$$
 \overline{2R}\subset 2Q.
$$
Apply \Cref{prop:preliminary-gauge} to the restriction $f|_{2R}$, taking the open cube $2Q$ as the ambient domain.  In both cases we obtain a M\"obius transformation $\varphi_R$, finite on $\frac32R$, and a normalized map
$$
g_R:=\varphi_R^{-1}\circ f
$$
such that
$$
\left(\fint_R|Dg_R-\Id|^n\,dx\right)^{1/n}
\le\omega_1(\delta).
$$
Up to decreasing $\delta_{\mathrm{gain}}$, the uniform estimate in \Cref{prop:preliminary-gauge} is smaller than the constant $a_0$ in \Cref{lem:mobius-comparison}.  The reverse estimate \eqref{eq:comparison-reverse} with exponent $n$ gives
\begin{align}
\int_R|Df-D\varphi_R|^n\,dx
&\le C(n, \zeta_0)^n \avnorm{D\varphi_R}{n}{R}^n
\int_R|Dg_R-\Id|^n\,dx\notag\\
&\le \bigl(C(n, \zeta_0)\omega_1(\delta)\bigr)^n \int_R|D\varphi_R|^n\,dx.
\label{eq:relative-BSO}
\end{align}

We emphasize that, in the sequel, whenever we invoke \Cref{cor:relative-bso-power}, and consequently \Cref{thm:relative-gurov}, we rely only on the dyadic formulation of the latter with $R_0 = Q$. Hence, it suffices to construct comparison fields on $Q$ and on its dyadic subcubes; note that all such cubes have sides parallel to those of $Q$.

Set
$w(\delta):=C(n, \zeta_0)\omega_1(\delta).$
Then $w(\delta)\to0$ as $\delta\to 0$, and \eqref{eq:relative-BSO} holds uniformly on every subcube $R\subset Q$ whose sides are parallel to those of $Q$.  If $w(\delta)=0$, then $\omega_1(\delta)=0$ and \eqref{eq:preliminary-Ln} gives $Du=0$ almost everywhere on $Q$; both sides of \eqref{eq:specialized-gain} then vanish. Therefore, we may assume for the remainder of the proof that $w(\delta)>0$.

For each exponent $q\in(1,\infty)$, use the admissible family $\mathscr S_q$ supplied by \Cref{lem:mobius-comparison-family}.
Apply \Cref{cor:relative-bso-power} first with $q=n$ and $s=p.$ Since $w(\delta)\to0$, up to decreasing $\delta_{\mathrm{gain}}$ we have
$$
 p\le\frac{c_1(n,p)}{w(\delta)}.
$$
For the comparison map $\varphi=\varphi_Q$ on the cube $Q$, \eqref{eq:power-self-improvement} gives
\begin{equation}
 \int_Q|Df-D\varphi|^p\,dx \le  C_{\rm self}w(\delta)^{p-n}
 \avnorm{D\varphi}{n}{Q}^{p-n} \int_Q|Df-D\varphi|^n\,dx,
\end{equation}
where $C_{\rm self}$ depends only on $p$ and $n$.
This, together with the relative $L^n$ estimate \eqref{eq:relative-BSO} on $Q$ and H\"older's inequality,  implies
\begin{align*}
 \int_Q|Df-D\varphi|^p\,dx
 &\le C(n,p,\zeta_0)w(\delta)^p  \avnorm{D\varphi}{n}{Q}^p|Q|\\
 &\le C(n,p,\zeta_0)w(\delta)^p \avnorm{D\varphi}{p}{Q}^p|Q|\\
 &=C(n,p,\zeta_0)w(\delta)^p \int_Q|D\varphi|^p\,dx.
\end{align*}
Repeating this argument on every cube $R\subset Q$, with its own comparison map $\varphi_R$, shows that $Df$ has relative $L^p$ specific oscillation at most
\begin{equation}
 w_p(\delta):=C(n,p,\zeta_0)^{1/p}w(\delta)
\end{equation}
with respect to the admissible family $\mathscr S_p$.  In particular, for the comparison map $\varphi=\varphi_Q$,
$$
 Df-D\varphi\in L^p(Q).
$$
The forward comparison estimate \eqref{eq:comparison-forward}, at exponent $p$, then gives $Du\in L^p(Q).$

Choose  $\epsilon\in (0,1)$.  Decrease $\delta_{\mathrm{gain}}$ again so that
$$
 p+\epsilon \le\frac{c_1(n,p,p+\epsilon)}{w_p(\delta)}\quad \text{ and } \quad 
w_p(\delta)\le w_0=w_0(n,p,\alpha_1,\alpha_2,\alpha_3).
$$
Applying \Cref{cor:relative-bso-power} a second time, now with
$$
 q=p,  \qquad s=p+\epsilon,
$$
yields
\begin{equation}\label{eq:Lp-to-Lpe-error}
 \int_Q|Df-D\varphi|^{p+\epsilon}\,dx
 \le C'_{\rm self} w_p(\delta)^\epsilon
 \avnorm{D\varphi}{p}{Q}^\epsilon
 \int_Q|Df-D\varphi|^p\,dx,
\end{equation}
where $C'_{\rm self}$ depends on $n,p,\epsilon$.
This proves, in particular, that $Df-D\varphi\in L^{p+\epsilon}(Q)$.

Apply \eqref{eq:comparison-forward} with exponent $p+\epsilon$ and \eqref{eq:comparison-reverse} with exponent $p$.  We obtain
$$
 \|Du\|_{L^{p+\epsilon}(Q)}
 \le C(n,p, \epsilon,\zeta_0) \avnorm{D\varphi}{p+\epsilon}{Q}^{-1} \|Df-D\varphi\|_{L^{p+\epsilon}(Q)},
$$
and
$$
 \|Df-D\varphi\|_{L^p(Q)}
 \le C(n,p, \epsilon,\zeta_0) \avnorm{D\varphi}{p}{Q}
 \|Du\|_{L^p(Q)}.
$$
Substitution in \eqref{eq:Lp-to-Lpe-error} gives
\begin{align*}
 \|Du\|_{L^{p+\epsilon}(Q)}^{p+\epsilon}
 &\le C(n,p,\epsilon,\zeta_0) w_p(\delta)^\epsilon
 \left( \frac{\avnorm{D\varphi}{p}{Q}} {\avnorm{D\varphi}{p+\epsilon}{Q}} \right)^{p+\epsilon} \|Du\|_{L^p(Q)}^p.
\end{align*}
As normalized integral means are nondecreasing in their exponent, 
$$
 \avnorm{D\varphi}{p}{Q} \le \avnorm{D\varphi}{p+\epsilon}{Q}.
$$
Consequently,
$$
 \|Du\|_{L^{p+\epsilon}(Q)}^{p+\epsilon}
 \le C(n,p, \epsilon,\zeta_0) w_p(\delta)^\epsilon
 \|Du\|_{L^p(Q)}^p.
$$
Thus, the proposition follows with
$$
 \nu(\delta):=C(n,p, \epsilon,\zeta_0)w_p(\delta)^\epsilon.
$$
\end{proof}

\subsection{Some algebraic estimates on matrices}

For $F\in GL^+(n)$, define its outer distortion by
\begin{equation}\label{eq:outer-distortion-matrix}
K_O(F):=\frac{\|F\|_{\mathrm{op}}^n}{\det F}.
\end{equation}

\begin{lemma}\label{lem:singular-value-gap}
Let $F\in GL^+(n)$, and let
$$
0<s_1\le\cdots\le s_n=\lambda
$$
be the singular values of $F$.  Then
\begin{equation}\label{eq:matrix-log}
\sum_{i=1}^n(\lambda-s_i) \le\lambda\log K_O(F) \le\lambda\bigl(K_O(F)-1\bigr).
\end{equation}
\end{lemma}

\begin{proof}
For $0<t\le1$, the inequality $1-t\le-\log t$ gives
$$
\lambda-s_i=\lambda\left(1-\frac{s_i}{\lambda}\right)
\le\lambda\log\frac{\lambda}{s_i}.
$$
Summing in $i$ and using $\prod_i s_i=\det F$,
$$
\sum_{i=1}^n(\lambda-s_i)
\le\lambda\log\frac{\lambda^n}{\prod_i s_i}
=\lambda\log K_O(F).
$$
The second inequality follows from $\log t\le t-1$ for $t>0$.
\end{proof}

Define
\begin{equation}\label{eq:defn-b-matrix}
b(t):=\frac{t^2}{1+t} \qquad \text{ for }\  t\ge 0.
\end{equation}
Then
$$
b'(t)=1-\frac1{(1+t)^2},\qquad 0\le b'(t)<1.
$$
In particular, $b$ is $1$-Lipschitz.
The following proposition is a counterpart of \cite[Lemma 2.1]{FZ2022}. 

\begin{proposition}\label{prop:saturated-strain}
Let $A\in\R^{n\times n}$,  $F=\Id+A$, and  $\delta\ge0$.
Assume either
$$
F\in GL^+(n),\qquad K_O(F)\le1+\delta,
$$
or $F=0$. Then 
\begin{equation}\label{eq:saturated-strain}
|\dev(\sym A)| \le C(n)\left[ \delta(1+|A|)+\frac{|A|^2}{1+|A|}\right].
\end{equation}
\end{proposition}

\begin{proof}
Assume first that $F\in GL^+(n)$.  Let
$$
F=RU, \qquad R\in SO(n), \qquad U=(F^TF)^{1/2},
$$
be the (right) polar decomposition of $F$; see \cite[Chapter 8]{H2008}.  Let
$$
0<s_1\le\cdots\le s_n=\lambda
$$
be the eigenvalues of $U$, and define
$$
E:=U-\lambda\Id,
\qquad
B:=\lambda R-\Id.
$$
Then
\begin{equation}\label{eq:A-split}
A=F-\Id=R(U-\lambda\Id)+(\lambda R-\Id)=RE+B.
\end{equation}
Since the eigenvalues of $E$ are $s_i-\lambda\le 0$,
\begin{equation}\label{eq:E-small}
|E|=\left(\sum_{i=1}^n(\lambda-s_i)^2\right)^{1/2}
\le\sum_{i=1}^n(\lambda-s_i)\le\lambda\delta
\end{equation}
by \eqref{eq:matrix-log}. 

Additionally, for $R\in SO(n)$,
\begin{equation}\label{eq:matrix-R-I}
(R-\Id)^T(R-\Id) =2\Id-R-R^T =2(\Id-\sym R).
\end{equation}
Therefore, by noting that
$$\dev(\sym R)=\dev(\sym R-\Id),$$
we obtain from the contraction property $|\dev M|\le|M|$ and \eqref{eq:matrix-R-I} that
\begin{equation}\label{eq:dev-rotation}
|\dev(\sym R)|=|\dev(\sym R-\Id)|
\le|\sym R-\Id| \le\frac12|R-\Id|^2.
\end{equation}
Moreover,
\begin{equation}\label{eq:B-exact}
|B|^2 =|\lambda R-\Id|^2 =\lambda^2|R|^2+|\Id|^2-2\lambda\tr R =n(\lambda-1)^2+\lambda|R-\Id|^2.
\end{equation}
We claim that
\begin{equation}\label{eq:rotation-saturated}
\lambda|R-\Id|^2 \le C_n b(|B|),
\end{equation}
where $b$ is defined in \eqref{eq:defn-b-matrix}. 
Indeed, if $|B|\le1$, then $b(|B|)\ge|B|^2/2$, and \eqref{eq:B-exact} gives
$$
\lambda|R-\Id|^2\le n(\lambda-1)^2+\lambda|R-\Id|^2
=|B|^2 \le2b(|B|).
$$
Assume $|B|>1$.  Then $b(|B|)\ge|B|/2$, and we have two subcases: If $\lambda<2$, then $|R-\Id|\le2\sqrt n$, and hence
$$
\lambda|R-\Id|^2\le 8n \le 8n|B| \le16n b(|B|).
$$
If $\lambda\ge2$, then \eqref{eq:B-exact} gives
$$
|B|\ge\sqrt n(\lambda-1) \ge\frac{\sqrt n}{2}\lambda,
$$
and consequently,
$$
\lambda|R-\Id|^2 \le4n\lambda \le 8\sqrt n|B| \le 16\sqrt n b(|B|).
$$
This proves \eqref{eq:rotation-saturated}.

Since $\dev(\sym B)=\lambda\dev(\sym R)$, \eqref{eq:dev-rotation} and \eqref{eq:rotation-saturated} imply
\begin{equation}\label{eq:dev-B}
|\dev(\sym B)|
\le C_n b(|B|).
\end{equation}
Then by the $1$-Lipschitz property of $b$ and \eqref{eq:A-split} with $R\in SO(n)$,
\begin{equation}\label{eq:b-transport}
b(|B|)
\le b(|A|)+\bigl||B|-|A|\bigr|
\le b(|A|)+|A-B|
=b(|A|)+|E|.
\end{equation}
Using the contraction property $|\dev M|\le|M|$ together with $R\in SO(n)$  again, the triangle inequality, \eqref{eq:A-split}, \eqref{eq:dev-B}, and \eqref{eq:b-transport} give
\begin{align}\label{eq:dev-sym-A} 
|\dev(\sym A)| &\le|\dev(\sym B)|+|\dev(\sym(RE))|\notag\\
&\le C(n) b(|B|)+|RE|\notag\\ &\le C(n) b(|A|)+C(n)|E|=C(n) \frac{|A|^2}{1+|A|} + C(n) \delta\lambda,
\end{align}
where we recall \eqref{eq:E-small}. Now since
$$
\lambda=\|F\|_{\mathrm{op}} \le\|\Id\|_{\mathrm{op}}+\|A\|_{\mathrm{op}} \le C(n)(1+|A|),
$$
then \eqref{eq:dev-sym-A} directly yields \eqref{eq:saturated-strain}.

If $F=0$, then $A=-\Id$ and $\dev(\sym A)=\dev(-\Id)=0$.  The last claim follows. 
\end{proof}

\begin{lemma} \label{lem:saturated-power}
Let $p\ge1$ and $0<\epsilon\le p$.  Then
\begin{equation}
\left(\frac{t^2}{1+t}\right)^p
\le t^{p+\epsilon}
\qquad\text{ for any} \ t\ge 0.
\end{equation}
\end{lemma}

\begin{proof}
If $0\le t\le1$, then since $2p\ge p+\epsilon$, we have
$$
\left(\frac{t^2}{1+t}\right)^p
\le t^{2p} \le t^{p+\epsilon}.
$$ 
If $t\ge 1$, then
$$
\left(\frac{t^2}{1+t}\right)^p
\le t^p \le t^{p+\epsilon}.
$$
Thus the lemma follows. 
\end{proof}

\subsection{Proof of the local stability theorem}
We first need an auxiliary lemma expressing the conformal invariance of outer distortion. 
\begin{lemma}\label{lem:outer-distortion-invariant}
Let $A\in\GL^+(n)$ and let $L=\lambda O$, where $\lambda>0$ and $O\in SO(n)$.  Then
$$
 K_O(LA)=K_O(A).
$$
\end{lemma}

\begin{proof}
Since $A\in\GL^+(n)$ and $O\in SO(n)$,
$$
 \|LA\|_{\mathrm{op}}=\lambda\|A\|_{\mathrm{op}},
 \qquad
 \det(LA)=\lambda^n\det A.
$$
The quotient in \eqref{eq:outer-distortion-matrix} is unchanged.
\end{proof}

Now we are ready to show \Cref{thm:local-main}.
\begin{proof}[Proof of \Cref{thm:local-main}]
Let $Q=Q(x_0,r)$, let $f:2Q\to\R^n$ satisfy the assumptions of \Cref{thm:local-main}, and set $\delta=K-1$.  Choose $\delta_{\mathrm{loc}}(n,p)$ smaller than both $\delta_{\mathrm{pre}}(n)$ and $\delta_{\mathrm{gain}}(n,p)$.  Let $\varphi$ and
$$
g=\varphi^{-1}\circ f=\operatorname{id}+u
$$
be supplied by \Cref{prop:preliminary-gauge}, and let
$\zeta_0=\zeta_0(n)\in(0,1)$ be the number given there. Then
$\Pi_{\zeta_0,Q}u=0$, and \Cref{prop:specialized-bso-gain} gives an $\epsilon\in(0,1)$ and a modulus $\nu(\delta)\to0$ such that
\begin{equation}\label{eq:gain-for-local-proof}
\int_Q|Du|^{p+\epsilon}\,dx
\le\nu(\delta)\int_Q|Du|^p\,dx.
\end{equation}

We next derive a pointwise estimate for $\Edev u$.  The derivative of a sense-preserving M\"obius transformation at every finite point is a positive scalar multiple of a rotation.  Hence \Cref{lem:outer-distortion-invariant} applies to composition with $\varphi^{-1}$, and at almost every $x\in Q$, either $Df(x)=0$, in which case $Dg(x)=0$, or, by the Sobolev chain rule for $g=\varphi^{-1}\circ f$ with the invariance of $K_O$ under multiplication by a positive scalar times a rotation,
$$
Dg(x)\in GL^+(n), \qquad K_O(Dg(x))\le K=1+\delta.
$$
If $Dg(x)=0$, then $Du(x)=-\Id$ and $\Edev u(x)=0$.  In the second case, we apply \Cref{prop:saturated-strain} with $A=Du(x)$.  Thus, for almost every $x\in Q$,
\begin{equation}\label{eq:pointwise-local-strain}
|\Edev u(x)|\le C_n\left[ \delta(1+|Du(x)|)+\frac{|Du(x)|^2}{1+|Du(x)|} \right].
\end{equation}

Now the projected trace-free Korn inequality, i.e. \Cref{thm:projected-korn}, implies
$$
\|Du\|_{L^p(Q)}\le C\|\Edev u\|_{L^p(Q)}.
$$
Employing \eqref{eq:pointwise-local-strain}, we further obtain that
\begin{equation}\label{eq:pre-absorption}
\|Du\|_{L^p(Q)}
\le C\delta |Q|^{1/p} +C\delta\|Du\|_{L^p(Q)}
+C\left\|\frac{|Du|^2}{1+|Du|}\right\|_{L^p(Q)}.
\end{equation}

Moreover, by \Cref{lem:saturated-power} and \eqref{eq:gain-for-local-proof},
$$
\left\|\frac{|Du|^2}{1+|Du|}\right\|_{L^p(Q)}^p
\le\int_Q|Du|^{p+\epsilon}\,dx \le\nu(\delta)\int_Q|Du|^p\,dx.
$$
Hence, 
\begin{equation}\label{eq:remainder-absorbed}
\left\|\frac{|Du|^2}{1+|Du|}\right\|_{L^p(Q)} \le\nu(\delta)^{1/p}\|Du\|_{L^p(Q)}.
\end{equation}
Substituting \eqref{eq:remainder-absorbed} into \eqref{eq:pre-absorption},
\begin{equation}\label{eq:Du-LpQ}
\|Du\|_{L^p(Q)}\le C\delta |Q|^{1/p} +C\bigl(\delta+\nu(\delta)^{1/p}\bigr)
\|Du\|_{L^p(Q)}.
\end{equation}
Since $\nu(\delta)\to0$, we may decrease $\delta_{\mathrm{loc}}(n,p)$ so that, for any $0<\delta\le\delta_{\mathrm{loc}}$, 
$$
C\bigl(\delta+\nu(\delta)^{1/p}\bigr)
\le\frac 1 2.
$$
Thus it follows from both $g-\operatorname{id}=u$ and \eqref{eq:Du-LpQ} that
\begin{equation}\label{eq:local-derivative-final}
\left(\fint_Q|Dg-\Id|^p\,dx\right)^{1/p}
=\left(\fint_Q|Du|^p\,dx\right)^{1/p}
\le C(n,p)\delta.
\end{equation}

Since the constant vector fields belong to $\Kconf$, the equality $\Pi_{\zeta_0,Q}u=0$ implies
$$
\fint_{  \zeta_0Q}u\,dx=0.
$$
Then   \Cref{lem:inner-poincare} and \eqref{eq:local-derivative-final} give
\begin{equation}\label{eq:local-Lp-final}
\frac1 r\left(\fint_Q|g-\operatorname{id}|^p\,dx\right)^{1/p}
\le C(n,p)\delta.
\end{equation}

It remains to prove the uniform estimate.  Let
$$
u_Q:=\fint_Q u\,dx.
$$
The Morrey estimate in \Cref{lem:morrey-cube} gives
\begin{equation}\label{eq:local-morrey-step}
\|u-u_Q\|_{L^\infty(Q)}
\le C(n,p)r^{1-n/p}\|Du\|_{L^p(Q)}.
\end{equation}
Moreover, since the average of $u$ over $\zeta_0 Q   $ vanishes,
\begin{align*}
|u_Q|
&=\left|\fint_{ \zeta_0 Q}(u_Q-u(x))\,dx\right|\\
&\le |\zeta_0 Q |^{-1/p}
\|u-u_Q\|_{L^p(\zeta_0 Q)}\\
&\le C(n,p)|Q|^{-1/p} \|u-u_Q\|_{L^p(Q)}\\
&\le C(n,p)r^{1-n/p}\|Du\|_{L^p(Q)},
\end{align*}
where the last line uses the Poincar\'e inequality in \Cref{lem:cube-poincare} on $Q$.  Together with \eqref{eq:local-morrey-step} and \eqref{eq:local-derivative-final}, as $u=g-\operatorname{id}$, this yields
\begin{equation}\label{eq:local-Linfty-final}
\frac1 r\|g-\operatorname{id}\|_{L^\infty(Q)}
\le C(n,p)\delta.
\end{equation}
Now combining \eqref{eq:local-derivative-final}, \eqref{eq:local-Lp-final}, and \eqref{eq:local-Linfty-final}, we obtain the desired estimate.  The pole-free property of $\varphi$ on $\frac 3 2Q$ follows from \Cref{prop:preliminary-gauge}.
\end{proof}

\section{Whitney decompositions and the Lorentz representation}\label{sec:whitney-edge}

We now construct the local geometric framework used in both globalization arguments.  Both the John-domain argument  and the  weighted bounded-domain theorem use  the same Whitney decomposition, rooted at a largest cube.

Recall that, for a cube $Q=Q(x_Q,\ell(Q))$ and $\lambda>0$, the notation $\lambda Q$ denotes the concentric open cube of sidelength $\lambda\ell(Q)$.

\subsection{The Whitney decomposition}

\begin{definition}[Whitney cubes]\label{def:deep-whitney}
Let $\Omega\subsetneq\R^n$ be a domain.  A dyadic cube $Q\in\mathscr D_\xi$ is called \emph{admissible} if
\begin{equation}\label{eq:deep-admissible}
64\overline Q\subset\Omega.
\end{equation}
The family of maximal  admissible cubes is denoted by
$$
\mathscr W=\mathscr W(\Omega,\mathscr D_\xi),
$$
where $\mathscr D_\xi$ is defined in \Cref{sec:notation}. 
Maximality is taken with respect to inclusion among cubes of the fixed dyadic grid.
\end{definition}

Now we record the following geometric properties of a Whitney decomposition, where the proof is recorded in the Appendix; see also \cite[Appendix J]{G2004}.

\begin{lemma}\label{lem:deep-whitney}
Let $\Omega\subsetneq\R^n$ be a (bounded) domain and let $\mathscr W$ be as in \Cref{def:deep-whitney}.  Then the following assertions hold.

\begin{enumerate}[label=\textup{(W\arabic*)},leftmargin=3em]
\item The interiors of distinct cubes in $\mathscr W$ are disjoint, the closures of the cubes cover $\Omega$, the family $\mathscr W$ is locally finite in $\Omega$, and
\begin{equation}\label{eq:doubled-whitney-cover}
\Omega=\bigcup_{Q\in\mathscr W}2Q.
\end{equation}

\item There are dimensional constants $c_W,C_W>0$ such that
\begin{equation}\label{eq:ordinary-whitney-distance}
c_W\ell(Q)
\le \dist(Q,\partial\Omega)
\le C_W\ell(Q)
\qquad \text{ for any } \ Q\in\mathscr W.
\end{equation}

\item Define adjacency for two distinct cubes $P,Q\in \mathscr W$ by
\begin{equation}\label{eq:whitney-adjacency}
P\sim Q
\quad\Longleftrightarrow\quad
\interior(2P)\cap\interior(2Q)\ne\varnothing.
\end{equation}
If $P\sim Q$, then
\begin{equation}\label{eq:adjacent-size-comparability}
\frac 1 4 \le \frac{\ell(P)}{\ell(Q)} \le4.
\end{equation}

\item The graph with vertex set $\mathscr W$, in which two distinct vertices $P,Q\in\mathscr W$ are joined by an edge whenever \eqref{eq:whitney-adjacency} holds, has degree bounded by a number
\begin{equation}\label{eq:defn-Dw}
N_W=N_W(n)<\infty.
\end{equation}

\item If $P\sim Q$ and
$$
s_{P,Q}:=\min\{\ell(P),\ell(Q)\},
$$
then there is a point $a_{P,Q}\in\interior(2P)\cap\interior(2Q)$ such that, with
\begin{equation}\label{eq:bridge-radius}
r_{P,Q}:=\frac12s_{P,Q},
\end{equation}
one has
\begin{equation}\label{eq:bridge-ball}
\overline{B(a_{P,Q},4r_{P,Q})}
\subset
8P\cap8Q.
\end{equation}
\end{enumerate}
\end{lemma}

\begin{corollary} \label{cor:max-cube}
If $\Omega$ is bounded, then $\mathscr W$ contains a cube of maximal sidelength.
\end{corollary}

\begin{proof}
According to Lemma~\ref{lem:deep-whitney}, the sidelengths of the cubes in $\mathscr W$ are dyadic and are bounded above by a number depending on $\diam(\Omega)$.  Then a nonempty subset of $\{2^k:k\in\mathbb Z\}$ that is bounded above has a largest element. This proves the corollary. 
\end{proof}

\subsection{Local estimates on Whitney cubes}
Fix $p>n$ and assume henceforth that
\begin{equation}\label{eq:small-dist}
0<\delta:=K-1\le\delta_{\mathrm{loc}}(n,p).
\end{equation} 
For every $Q\in\mathscr W$, the admissible property gives
$$
16Q\subset64Q\subset\Omega.
$$
The restriction of a nonconstant quasiregular map to $16Q$ is nonconstant by the openness theorem recalled in \Cref{sec:QC}.  We may therefore apply \Cref{thm:local-main} with the cube $8Q$.

\begin{proposition} \label{prop:whitney-local-gauges}
Under the assumption \eqref{eq:small-dist}, for every $Q\in\mathscr W$, there is a sense-preserving M\"obius transformation $\varphi_Q$, finite on $12Q$, such that
$$
g_Q:=\varphi_Q^{-1}\circ f
$$
is finite on $8Q$ and
\begin{equation}\label{eq:whitney-local-uniform}
\|g_Q-\operatorname{id}\|_{L^\infty(8Q)}
\le C_{\mathrm{loc}}'(n,p)\delta\ell(Q),
\end{equation}
\begin{equation}\label{eq:whitney-local-derivative}
\left(
\fint_{8Q}|Dg_Q-\Id|^p\,dx
\right)^{1/p}
\le C_{\mathrm{loc}}'(n,p)\delta.
\end{equation}
The same choice also satisfies the scaled $L^p$ estimate
\begin{equation}
\frac1{\ell(Q)}
\left(
\fint_{8Q}|g_Q-\operatorname{id}|^p\,dx
\right)^{1/p}
\le C_{\mathrm{loc}}'(n,p)\delta.
\end{equation}
\end{proposition}

\begin{proof}
Apply \Cref{thm:local-main} with $8Q$ in place of $Q$, whose double is $16Q\subset\subset\Omega$, and its $3/2$-dilate is $12Q$.  Since $\ell(8Q)=8\ell(Q)$, the factors of $8$ then can be absorbed into the constant $C_{\mathrm{loc}}'(n,p)$. This gives the proposition. 
\end{proof}

Now we study the estimates between neighboring cubes. Let us first record the elementary surjectivity statement used to pass from information about $g_Q$ to information about a transition M\"obius map on a target ball.

\begin{lemma} \label{lem:brouwer-image}
Let $a\in\R^n$, let $R>0$, and let
$$
h:\overline{B(a,R)}\to\R^n
$$
be continuous.  Assume that
$$
\sup_{x\in\overline{B(a,R)}}|h(x)-x|
\le\epsilon
\qquad \text{ for some }\ 0<\epsilon<R.
$$
Then
\begin{equation}\label{eq:brouwer-image-inclusion}
B(a,R-\epsilon)
\subset
h\bigl(\overline{B(a,R)}\bigr).
\end{equation}
\end{lemma}

\begin{proof}
Fix $y\in B(a,R-\epsilon)$ and define
$$
T_y(x):=y+x-h(x)
\qquad \text{ for } \ x\in\overline{B(a,R)}.
$$
For such $x$,
$$
|T_y(x)-a|
\le |y-a|+|x-h(x)|
<R-\epsilon+\epsilon
=R.
$$
Thus $T_y$ maps the closed ball continuously into itself. Then  Brouwer's fixed-point theorem gives $x\in\overline{B(a,R)}$ such that $T_y(x)=x$.  This equality is equivalent to $h(x)=y$, proving \eqref{eq:brouwer-image-inclusion}.
\end{proof}

\begin{proposition} \label{prop:bridge-synchronization}
Assume \eqref{eq:small-dist}. There are constants
$$
\delta_{\mathrm{br}}=\delta_{\mathrm{br}}(n,p)>0,
\qquad
C_{\mathrm{br}}=C_{\mathrm{br}}(n,p)<\infty
$$
with the following property.  Let $P,Q\in\mathscr W$ be adjacent, let $a=a_{P,Q}$ and $r=r_{P,Q}$ be as in (W5) of \Cref{lem:deep-whitney}, and assume
$$
0<\delta\le\delta_{\mathrm{br}}.
$$
Define the transition M\"obius transformation
\begin{equation}\label{eq:defn-ThetaPQ}
\Theta_{P,Q}:=\varphi_P^{-1}\circ\varphi_Q.
\end{equation}
Then $\Theta_{P,Q}$ has no pole in $B(a,3r)$ and
\begin{equation}\label{eq:transition-bridge-estimate}
\sup_{y\in B(a,3r)}
|\Theta_{P,Q}(y)-y|
\le C_{\mathrm{br}}\delta r.
\end{equation}
\end{proposition}

\begin{proof}
According to \eqref{eq:bridge-ball}, $\overline{B(a,4r)}\subset8P\cap8 Q.$
Then the adjacent-size comparability and \eqref{eq:bridge-radius} give
$$
\ell(P)+\ell(Q)\le C(n)r.
$$
The uniform local estimates \eqref{eq:whitney-local-uniform} therefore imply
\begin{equation}\label{eq:g-local-uniform}
|g_P(x)-x|+|g_Q(x)-x|
\le\eta r
\qquad \text{ for any } \ x\in\overline{B(a,4r)},
\end{equation}
where $\eta:=C(n,p)\delta.$
We may decrease $\delta_{\mathrm{br}}$ so that $\eta\le 1$.

Applying \Cref{lem:brouwer-image} to $h=g_Q$, $R=4r$, and $\epsilon=\eta r$,  since $\eta\le 1$, we get
$$ B(a,3r)\subset B(a,(4-\eta)r)
\subset g_Q\bigl(\overline{B(a,4r)}\bigr).
$$
Fix $y\in B(a,3r)$. Then by applying \Cref{lem:brouwer-image} to $g_Q$,  there is $x\in\overline{B(a,4r)}$ such that
$$
g_Q(x)=y.
$$
On the common  region $8P\cap8Q$, the definitions of the local mappings together with \eqref{eq:defn-ThetaPQ} give,
first as an identity of maps into $\widehat{\R}^{\,n}$,
$$
g_P =\varphi_P^{-1}\circ f
= (\varphi_P^{-1}\circ\varphi_Q)\circ(\varphi_Q^{-1}\circ f)
= \Theta_{P,Q}\circ g_Q.
$$
Notice that, the identity
$$
g_P=\Theta_{P,Q}\circ g_Q
$$
should be  first understood as an identity with values in $\widehat{\mathbb R}^{\,n}$.
Then since both $g_Q(x)=y$ and $g_P(x)$ are finite, the point $y$ is not a pole of $\Theta_{P,Q}$, and hence the equality holds in $\mathbb R^n$ at $y$. 

Additionally, according to \eqref{eq:g-local-uniform}
\begin{align*}
|\Theta_{P,Q}(y)-y|
&=|g_P(x)-g_Q(x)|\\
&\le |g_P(x)-x|+|g_Q(x)-x| \le\eta r.
\end{align*}
This proves \eqref{eq:transition-bridge-estimate}, up to absorbing the multiplicative constant in $C_{\mathrm{br}}$.
\end{proof}

\begin{corollary} \label{cor:normalized-bridge-data}
Let $P,Q$ be adjacent and define
\begin{equation}
E_{P,Q}:=S_Q\circ\Theta_{P,Q}\circ S_Q^{-1},
\end{equation}
where $S_Q$ is  the affine normalization of $Q$ defined in \Cref{sec:notation}. 
Then there is a Euclidean ball $B(b_{P,Q},\rho_{P,Q})$
such that
\begin{equation}\label{eq:defn-rho}
|b_{P,Q}|\le\sqrt n, \qquad \frac 3 8 \le\rho_{P,Q}\le\frac 3 2,
\end{equation}
and
\begin{equation}\label{eq:normalized-edge-closeness}
\sup_{z\in B(b_{P,Q},\rho_{P,Q})} |E_{P,Q}(z)-z|\le C(n,p)\delta.
\end{equation}
In particular, the centers and radii of the  normalized bridge balls range over a compact family depending only on $n$.
\end{corollary}

\begin{proof}
Take
$$
b_{P,Q}:=S_Q(a_{P,Q}), \qquad \rho_{P,Q}:=\frac{3r_{P,Q}}{\ell(Q)}.
$$
Since $a_{P,Q}\in2Q$, $|b_{P,Q}|_\infty<1,$
and hence $|b_{P,Q}|\le\sqrt n$.  Moreover,
$$
\frac 1 4 \le\frac{s_{P,Q}}{\ell(Q)}\le1
$$
by adjacent-size comparability.  Since $r_{P,Q}=s_{P,Q}/2$,
$$
\frac 3 8 \le \rho_{P,Q} \le\frac 3 2.
$$
Scaling \eqref{eq:transition-bridge-estimate} by $S_Q$ gives \eqref{eq:normalized-edge-closeness}.
\end{proof}

The next subsection converts \eqref{eq:normalized-edge-closeness} into quantitative $C^1$ control on arbitrary fixed enlargements and into a near-identity estimate for the Lorentz matrix of $E_{P,Q}$.  Those conclusions are finite-dimensional consequences of the M\"obius group structure; no affine approximation of a general M\"obius transformation is used.

\subsection{The Lorentz representation of the M\"obius group}\label{sec:lorentz}

We equip
$$
\R^{n+2}=\R^{n+1}\times\R
$$
with the Lorentz bilinear form
\begin{equation}\label{eq:lorentz-product}
\langle (z,t),(w,s)\rangle_{\mathrm L}
:=z\cdot w-ts.
\end{equation}
Let
$$
\mathsf J:=\operatorname{diag}(I_{n+1},-1).
$$
Then we have
$$
\langle X,Y\rangle_{\mathrm L}=X\cdot \mathsf JY
\qquad \text{ for any }\ X,Y\in\R^{n+2}.
$$
The future null cone (or future light cone) is
$$
\mathscr C_+ :=\bigl\{(z,t)\in\R^{n+1}\times\R: |z|=t>0\bigr\}.
$$
We write
$$
O^{\uparrow}(n+1,1)
:=\bigl\{M\in\GL(n+2,\R): M^T\mathsf JM=\mathsf J, \ M(\mathscr C_+)=\mathscr C_+ \bigr\}.
$$
The arrow indicates preservation of the future component of the null cone.  All matrix norms in this section are Euclidean operator norms.

For $\xi\in\Sn$, put
$$
v_\xi:=(\xi,1)\in\mathscr C_+.
$$
If $M\in O^{\uparrow}(n+1,1)$, then $Mv_\xi\in\mathscr C_+$.  Consequently, there are unique quantities
$$
M^\#(\xi)\in\Sn,
\qquad
\alpha_M(\xi)>0,
$$
such that
\begin{equation}\label{eq:projective-definition}
M v_\xi=\alpha_M(\xi)
\bigl(M^\#(\xi),1\bigr).
\end{equation}
The map $M^\#:\Sn\to\Sn$ is called the \emph{projective action} of $M$.

Indeed, one can write the projective action explicitly.  Decompose
\begin{equation}\label{eq:lorentz-block}
M=
\begin{pmatrix}
A&b\\
c^T&d
\end{pmatrix},
\end{equation}
where
$$
A\in\R^{(n+1)\times(n+1)}, \qquad b,c\in\R^{n+1}, \qquad d\in\R.
$$
Then \eqref{eq:projective-definition} indeed gives
\begin{equation}\label{eq:projective-formula}
M^\#(\xi)
=\frac{A\xi+b}{c\cdot\xi+d} \quad \text{ and } 
\quad \alpha_M(\xi)=c\cdot\xi+d  \ \text{ for } \, \xi\in\Sn. 
\end{equation}
The denominator is positive once we choose $M$ so that it preserves $\mathscr C_+$. Indeed, the next proposition gives this Lorentz representation of M\"obius transformations. 
\begin{proposition}\label{prop:lorentz-realization}
For every $T\in\Mob(n)$, there exists a unique matrix
$$
\mathcal R(T)\in O^{\uparrow}(n+1,1)
$$
such that
\begin{equation}\label{eq:hat-sharp}
\widehat T=(\mathcal R(T))^\#;
\end{equation}
recall \eqref{eq:hat-operation}. 
The map $T\mapsto\mathcal R(T)$ is an isomorphism of topological groups.  In particular,
\begin{equation}\label{eq:lorentz-homomorphism}
\mathcal R(T_1\circ T_2) = \mathcal R(T_1)\mathcal R(T_2)
\qquad \text{ for any } \ T_1,T_2\in\Mob(n).
\end{equation}
In particular,
$$
\mathcal R(\operatorname{id})=I_{n+2}, \qquad \mathcal R(T^{-1})=\mathcal R(T)^{-1}.
$$
\end{proposition}

\begin{proof}
We apply the projective Lorentz representation theorem in Gehring--Martin--Palka~\cite[Corollary~3.7.4]{GMP2017}, which is equivalent to  Reshetnyak~\cite[Chapter~2, Theorem~1.2]{R1994}.  That theorem associates with each M\"obius transformation $T$ a class of matrices, which is exactly the projective Lorentz class $[M]=\{M,-M\}$, inducing $\widehat T$.  Observe that, exactly one of $M$ and $-M$ maps the future cone $\mathscr C_+$ to itself, and we define $\mathcal R(T)$ to be that representative. This proves existence and uniqueness.  

The matrix product $\mathcal R(T_1)\mathcal R(T_2)$ preserves $\mathscr C_+$ and induces the sphere map
$$
\widehat T_1\circ\widehat T_2
= \widehat{T_1\circ T_2}.
$$
Uniqueness of the future-preserving representative gives \eqref{eq:lorentz-homomorphism}.  Continuity of $T\mapsto\mathcal R(T)$ and its inverse is also part of the cited projective representation theorem.
\end{proof}

\begin{lemma}\label{lem:matrix-norm-bound}
Let $M\in O^{\uparrow}(n+1,1)$.  Then
\begin{equation}\label{eq:inverse-norm-lorentz}
\|M^{-1}\|_{\mathrm{op}}=\|M\|_{\mathrm{op}},
\end{equation}
and, for every $\xi\in\Sn$,
\begin{equation}\label{eq:alpha-bounds}
\|M\|_{\mathrm{op}}^{-1} \le \alpha_M(\xi) \le \|M\|_{\mathrm{op}}.
\end{equation}
In particular, $\|M\|_{\mathrm{op}}\ge1$.
\end{lemma}

\begin{proof}
The Lorentz identity $M^T\mathsf JM=\mathsf J$ implies
$$
M^{-1}=\mathsf JM^T\mathsf J.
$$
Since $\mathsf J$ is Euclidean orthogonal,
$$
\|M^{-1}\|_{\mathrm{op}} = \|M^T\|_{\mathrm{op}}
= \|M\|_{\mathrm{op}},
$$
which proves \eqref{eq:inverse-norm-lorentz}.

Since $Mv_\xi$ lies in the future null cone, \eqref{eq:projective-definition} together with $\xi, M^\#(\xi)\in \mathbb S^n$ gives
$$
|Mv_\xi|_{\R^{n+2}} = \sqrt2\,\alpha_M(\xi),
\qquad |v_\xi|_{\R^{n+2}}=\sqrt2.
$$
Hence
$$
\sqrt2\,\alpha_M(\xi)
= |Mv_\xi|
\le \|M\|_{\mathrm{op}}|v_\xi|
= \sqrt2\,\|M\|_{\mathrm{op}},
$$
which gives the upper bound in \eqref{eq:alpha-bounds}.  Conversely,
$$
\sqrt2
=|v_\xi|
=|M^{-1}Mv_\xi|
\le
\|M^{-1}\|_{\mathrm{op}}\sqrt2\,\alpha_M(\xi),
$$
and then \eqref{eq:inverse-norm-lorentz} gives the lower bound.  Finally,
$$
1=\|I_{n+2}\|_{\mathrm{op}}
\le \|M\|_{\mathrm{op}}\|M^{-1}\|_{\mathrm{op}}
= \|M\|_{\mathrm{op}}^2,
$$
which yields the last inequality. 
\end{proof}

Set the open half-space
$$
\mathcal U_M:=\{\xi\in\R^{n+1}:c\cdot\xi+d>0\}.
$$
Then it follows from the lower bound in \eqref{eq:alpha-bounds} that $$\overline{\mathbb B}^{\,n+1}\subset\mathcal U_M.$$
We denote the canonical rational extension \eqref{eq:projective-formula} of $M$ by $\widetilde M^{\#}$, that is
\begin{equation}\label{eq:canonical-extension}
\widetilde M^{\#}(\xi) :=\frac{A\xi+b}{c\cdot\xi+d}
 \qquad \text{ for any} \ \xi\in\mathcal U_M,
\end{equation}
and $M^{\#}=\widetilde M^{\#}|_{\Sn}$.  Throughout this subsection, $DM^{\#}$ and $D^2M^{\#}$ mean the ambient derivatives of $\widetilde M^{\#}$ evaluated on $\Sn$.  Restricting $DM^{\#}$ to tangent vectors gives the differential of the sphere action.  The norm $\|M^{\#}\|_{C^1(\Sn)}$ below uses this ambient first derivative restricted to $\Sn$. Moreover, define 
$$\|DM^\#\|_{L^\infty(\mathbb S^n)}=\sup_{\xi\in \Sn}\sup_{\substack{h\in T_\xi\Sn\\ |h|=1}} |DM^\#(\xi) h|,$$
and analogously define  the tangential second derivative. 

\begin{lemma}\label{lem:projective-derivatives}
Let $M\in O^{\uparrow}(n+1,1)$.  Then
\begin{equation}\label{eq:projective-C1}
\|DM^\#\|_{L^\infty(\Sn)}
\le
2\|M\|_{\mathrm{op}}^2
\end{equation}
and
\begin{equation}\label{eq:projective-C2}
\|D^2M^\#\|_{L^\infty(\Sn)}
\le 4\|M\|_{\mathrm{op}}^4.
\end{equation}
Moreover, letting
$$
 \overline{\mathbb B}^{\,n+1}
 :=\{\xi\in\R^{n+1}:|\xi|\le1\},
$$
the canonical rational extension satisfies
\begin{equation}\label{eq:projective-ambient-C2}
 \|D^2\widetilde M^{\#}\|_{L^\infty(\overline{\mathbb B}^{\,n+1})}
 \le
 5\|M\|_{\mathrm{op}}^6.
\end{equation}
\end{lemma}

\begin{proof}
Write
\begin{equation}\label{eq:defn-P-alpha}
    P_M(\xi):=A\xi+b,\qquad \alpha_M(\xi):=c\cdot\xi+d.
\end{equation}
As $|M^\#(\xi)|=1$, \eqref{eq:projective-formula} gives 
\begin{equation}\label{eq:P-alpha-equality}
|P_M(\xi)|=\alpha_M(\xi).
\end{equation}
Let $h\in T_\xi\Sn$.  Differentiating \eqref{eq:projective-formula} in $h$-direction yields,
\begin{equation}\label{eq:first-projective-derivative}
DM^\#(\xi)h
=\frac{Ah}{\alpha_M(\xi)}
-\frac{P_M(\xi)(c\cdot h)}{\alpha_M(\xi)^2}.
\end{equation}
Since
\begin{equation}\label{eq:upper-operator-bound}
\|A\|_{\mathrm{op}}\le\|M\|_{\mathrm{op}},
\qquad |c|\le\|M\|_{\mathrm{op}},    
\end{equation}
\eqref{eq:P-alpha-equality} and \eqref{eq:alpha-bounds} imply
$$
|DM^\#(\xi)h|
\le
\frac{2\|M\|_{\mathrm{op}}}{\alpha_M(\xi)}|h|
\le
2\|M\|_{\mathrm{op}}^2|h|.
$$
This proves \eqref{eq:projective-C1}.

Moreover, for tangent vectors $h,k\in T_\xi\Sn$, an additional differentiation gives
$$
D^2M^\#(\xi)[h,k]
=-\frac{Ah(c\cdot k)+Ak(c\cdot h)}{\alpha_M(\xi)^2} +2\frac{P_M(\xi)(c\cdot h)(c\cdot k)}{\alpha_M(\xi)^3}.
$$
Then applying \eqref{eq:P-alpha-equality}, we conclude that
$$
|D^2M^\#(\xi)[h,k]|
\le
4\frac{\|M\|_{\mathrm{op}}^2}{\alpha_M(\xi)^2}|h| |k|,
$$
and the lower bound in \eqref{eq:alpha-bounds} further gives
$$
\|D^2M^\#\|_{L^\infty(\Sn)}
\le
4\|M\|_{\mathrm{op}}^4,
$$
which proves \eqref{eq:projective-C2}.

It remains to prove \eqref{eq:projective-ambient-C2}.  Recall $\|M\|_{\mathrm{op}}\ge 1$ by \Cref{lem:matrix-norm-bound}. 
Since $\alpha_M$ is affine, its minimum on the closed unit ball is attained on the unit sphere.  Hence \eqref{eq:alpha-bounds} implies
\begin{equation}\label{eq:alpha-unit-ball-lower}
 \alpha_M(\xi)\ge \|M\|_{\mathrm{op}}^{-1}
 \qquad \text{ for any } \ |\xi|\le 1.
\end{equation}
Also,
\begin{equation}\label{eq:P-unit-ball-bound}
 |P_M(\xi)|  \le |M(\xi,1)|  \le \sqrt2 \|M\|_{\mathrm{op}} \qquad \text{ for any } \ |\xi|\le 1.
\end{equation}
Note that, the ambient second-derivative formula is
$$
D^2\widetilde M^\#(\xi)[h,k]
= -\frac{Ah(c\cdot k)+Ak(c\cdot h)}{\alpha_M(\xi)^2}
+2\frac{P_M(\xi)(c\cdot h)(c\cdot k)}{\alpha_M(\xi)^3}.
$$
Employing \eqref{eq:alpha-unit-ball-lower}, \eqref{eq:upper-operator-bound} and \eqref{eq:P-unit-ball-bound}, we conclude that, for $|\xi|\le1$,
\begin{align*}
 |D^2\widetilde M^\#(\xi)[h,k]|
 &\le  2\|M\|_{\mathrm{op}}^4|h|\,|k|
 +2\sqrt2 \|M\|_{\mathrm{op}}^6|h|\,|k|\\
 &\le5\|M\|_{\mathrm{op}}^6|h|\,|k|.
\end{align*}
This gives \eqref{eq:projective-ambient-C2}.
\end{proof}

\begin{lemma} \label{lem:projective-difference}
Let $M,N\in O^{\uparrow}(n+1,1)$, and set
$$
L:=\max\{1,\|M\|_{\mathrm{op}},\|N\|_{\mathrm{op}}\}.
$$
Then
\begin{equation}\label{eq:projective-difference}
\|M^\#-N^\#\|_{C^1(\Sn)} \le C\|M-N\|_{\mathrm{op}}L^3,
\end{equation}
where $C>0$ is an absolute constant.
\end{lemma}

\begin{proof}
Write the  block matrix decompositions of $M$ and $N$ as in \eqref{eq:lorentz-block}, and set
$$
\Delta:=\|M-N\|_{\mathrm{op}}.
$$
For every $\xi\in\Sn$, following the notation \eqref{eq:defn-P-alpha}
\begin{equation}\label{eq:block-difference}
|P_M(\xi)-P_N(\xi)|
+
|\alpha_M(\xi)-\alpha_N(\xi)|
\le
C\Delta.
\end{equation}
By \eqref{eq:projective-formula},
$$
M^\#-N^\#
=\frac{P_M-P_N}{\alpha_M}
+P_N\left(\frac1{\alpha_M}-\frac1{\alpha_N}\right).
$$
Since $|P_N|=\alpha_N$, \eqref{eq:alpha-bounds} and
\eqref{eq:block-difference} yield
\begin{equation}\label{eq:C0-projective-difference}
\|M^\#-N^\#\|_{L^\infty(\Sn)} \le 4L\Delta,
\end{equation}
which proves the zeroth-order part of
\eqref{eq:projective-difference}.

Moreover, the derivative formula \eqref{eq:first-projective-derivative} can be written as
\begin{equation}\label{eq:derivative-tensor-form}
DM^\# = \frac{A_M}{\alpha_M} - \frac{M^\#\otimes c_M}{\alpha_M},  
\end{equation}
which should be understood as the following mapping: 
$$h\mapsto \frac{A_M h}{\alpha_M}  - \frac{M^\#  (c_M\cdot h)}{\alpha_M}.$$
For the first term in \eqref{eq:derivative-tensor-form},  it follows from \eqref{eq:alpha-bounds}, \eqref{eq:block-difference} and \eqref{eq:C0-projective-difference} that
$$
\left\|\frac{A_M}{\alpha_M} - \frac{A_N}{\alpha_N}
\right\| \le \frac{\|A_M-A_N\|}{\alpha_M}
+ \|A_N\| \frac{|\alpha_M-\alpha_N|}{\alpha_M\alpha_N} \le C L^3\Delta.
$$
For the second term, by inserting and subtracting twice, we   get
\begin{equation*}
 \left\|\frac{M^\#\otimes c_M}{\alpha_M} -\frac{N^\#\otimes c_N}{\alpha_N}\right\|
\le \frac{|M^\#-N^\#|\,|c_M|}{\alpha_M}
+ \frac{|c_M-c_N|}{\alpha_M}
+ |c_N| \left|\frac1{\alpha_M}-\frac1{\alpha_N}\right|.
\end{equation*}
Now applying \eqref{eq:alpha-bounds}, \eqref{eq:block-difference} and \eqref{eq:C0-projective-difference} again, the right-hand side is bounded by $C\Delta L^3$.  Together with \eqref{eq:derivative-tensor-form}, this proves the first derivative part of \eqref{eq:projective-difference}. 
\end{proof}

\subsection{A finite  frame and quantitative matrix control}

For $x\in\R^n$, define its Euclidean null lift by
\begin{equation}\label{eq:euclidean-null-lift}
\mathcal V(x) := \left( x, \frac{|x|^2-1}{2},  \frac{|x|^2+1}{2} \right)\in\R^{n+2}.
\end{equation}
Notice that, the first $(n+1)$-coordinates in \eqref{eq:euclidean-null-lift} form the space-like component in \eqref{eq:lorentz-product}, while the last coordinate is time-like.  Then \eqref{eq:stereographic} gives
\begin{equation}\label{eq:null-lift-stereographic}
\mathcal V(x) = \frac{1+|x|^2}{2} \bigl(\sigma(x),1\bigr).
\end{equation}
In addition, a direct calculation gives the key identity
\begin{equation}\label{eq:null-distance-identity}
\langle\mathcal V(x),\mathcal V(y)\rangle_{\mathrm L}
=-\frac12|x-y|^2 \qquad \text{ for any } \ x,y\in\R^n.
\end{equation}

Fix $t_0:=\frac 1 4$ and the $n+2$ points
\begin{equation}\label{eq:finite-frame-points}
x_0:=0, \qquad x_i:=t_0\mathbf{e}_i\quad \text{ for } \ 1\le i\le n, \qquad x_{n+1}:=2t_0\mathbf{e}_1.
\end{equation}
All these points lie in $B(0,1)$.

\begin{lemma}\label{lem:null-frame-basis}
The vectors
$$
\mathcal V(x_0),\ldots,\mathcal V(x_{n+1})
$$
form a basis of $\R^{n+2}$.
\end{lemma}

\begin{proof}
Suppose
$$
a_0\mathcal V(x_0)
+ \sum_{i=1}^na_i\mathcal V(x_i)
+ a_{n+1}\mathcal V(x_{n+1}) =0.
$$

Apply \eqref{eq:euclidean-null-lift}. The first $n$-coordinates give
$$
a_i=0\quad \text{ for } \ 2\le i\le n, \qquad  a_1+2a_{n+1}=0.
$$
Moreover, for every $x\in \mathbb R^n$, the last coordinate of $\mathcal V(x)$ minus its $(n+1)$-st coordinate equals $1$.  Hence
\begin{equation}\label{eq:frame-coefficient-sum}
a_0+a_1+a_{n+1}=0.
\end{equation}
The sum of the last two coordinates of $\mathcal V(x)$ equals $|x|^2$.  Therefore
$$
t_0^2a_1+4t_0^2a_{n+1}=0.
$$
Together with $a_1+2a_{n+1}=0$, this gives
$$
a_{n+1}=a_1=0.
$$
Thus,  \eqref{eq:frame-coefficient-sum}  yields $a_0=0$, which concludes the lemma.
\end{proof}

\begin{proposition} \label{prop:finite-frame-matrix-control}
There exist constants
$$
\eta_{\mathrm{fr}}=\eta_{\mathrm{fr}}(n)>0,
\qquad
C_{\mathrm{fr}}=C_{\mathrm{fr}}(n)>0,
$$
such that the following holds.  Let $T\in\Mob(n)$ be finite on $B(0,1)$, and suppose that
\begin{equation}\label{eq:unit-ball-Mobius-closeness}
\sup_{x\in B(0,1)}|T(x)-x| \le \eta \le\eta_{\mathrm{fr}}.
\end{equation}
Then
\begin{equation}\label{eq:matrix-near-identity}
\|\mathcal R(T)-I_{n+2}\|_{\mathrm{op}} \le C_{\mathrm{fr}}\eta,
\end{equation}
where $\mathcal R(T)$ is the representation of $T$ given by \Cref{prop:lorentz-realization}. 
\end{proposition}

\begin{proof}
Set $M:=\mathcal R(T).$
By \eqref{eq:null-lift-stereographic} and the definition of $\mathcal R(T)$, for each point $x_i$ in \eqref{eq:finite-frame-points} there exists $\lambda_i>0$ such that
\begin{equation}\label{eq:lift-scaling}
M\mathcal V(x_i)
=
\lambda_i\mathcal V(T(x_i)).
\end{equation}
Then Lorentz invariance and \eqref{eq:null-distance-identity} imply, for $i\ne j$,
\begin{equation}\label{eq:lambda-products}
\lambda_i\lambda_j
=
\frac{|x_i-x_j|^2}{|T(x_i)-T(x_j)|^2}.
\end{equation}

Let
$$
d_0:=\min_{i\ne j}|x_i-x_j|>0.
$$
By \eqref{eq:unit-ball-Mobius-closeness} and the triangle inequality,
$$
\bigl||T(x_i)-T(x_j)|-|x_i-x_j|\bigr|\le 2 \eta.
$$
After decreasing $\eta_{\mathrm{fr}}$, we may assume $2\eta\le d_0/2$.  Hence,
$$
\frac{d_0}{2} \le |T(x_i)-T(x_j)| \le 3 \qquad \text{ for } \ i \ne j,
$$
and then \eqref{eq:lambda-products} gives
\begin{equation}\label{eq:lambda-product-close}
|\lambda_i\lambda_j-1| \le C(n)\eta
\qquad \text{ for } \ i \ne j.
\end{equation}

Up to further decreasing $\eta_{\mathrm{fr}}$, \eqref{eq:lambda-product-close} gives
$$
\frac 1 2\le \lambda_i\lambda_j\le \frac32
\qquad \text{ for any} \ i\ne j.
$$
For fixed  $i$, choose distinct $j,k\ne i$. Then
$$
\lambda_i^2=\frac{(\lambda_i\lambda_j)(\lambda_i\lambda_k)}{\lambda_j\lambda_k}.
$$
Employing that $|ab-1|\le |a-1|\,|b|+|b-1|$ and the above-mentioned upper bounds, we obtain
$$
|\lambda_i^2-1|\le C(n)\eta.
$$
Since $\lambda_i>0$, decreasing $\eta_{\mathrm{fr}}$ once more gives
\begin{equation}\label{eq:lambda-close}
|\lambda_i-1|\le C(n)\eta
\qquad \text{ for } \ 0\le i\le n+1.
\end{equation}

The points $T(x_i)$ belong to $B(0,2)$ when $\eta\le1$.  Since $\mathcal V$ is Lipschitz on $\overline{B(0,2)}$, \eqref{eq:lift-scaling}, \eqref{eq:unit-ball-Mobius-closeness}, and \eqref{eq:lambda-close} give
$$
|(M-I_{n+2})\mathcal V(x_i)| \le |\lambda_i-1|\,|\mathcal V(T(x_i))| + |\mathcal V(T(x_i))-\mathcal V(x_i)| \le C(n)\eta.
$$
By \Cref{lem:null-frame-basis}, the vectors $\mathcal V(x_i)$ form a fixed basis, and the norm equivalence in this basis gives \eqref{eq:matrix-near-identity}.
\end{proof}

The preceding proposition is stable under similarities whose parameters range in a compact set.

\begin{corollary}\label{cor:lorentz}
Assume \eqref{eq:small-dist}. 
Then there are constants
$$
\delta_{\mathcal L}=\delta_{\mathcal L}(n,p)>0,
\qquad
C_{\mathcal L}=C_{\mathcal L}(n,p)>0,
$$
such that, whenever $P,Q\in\mathscr W$ are adjacent and $0<\delta\le\delta_{\mathcal L}$,
\begin{equation}\label{eq:edge-Lorentz-matrix}
\|\mathcal R(E_{P,Q})-I_{n+2}\|_{\mathrm{op}}
\le
C_{\mathcal L}\delta.
\end{equation}
\end{corollary}

\begin{proof}
Let $B(b,\rho)$ be the normalized bridge ball given by \Cref{cor:normalized-bridge-data}.  Define the similarity
$$
L_{b,\rho}(z):=\frac{z-b}{\rho}.
$$
Then
$$
\widetilde E := L_{b,\rho}\circ E_{P,Q}\circ L_{b,\rho}^{-1}
$$
is finite on $B(0,1)$, and \eqref{eq:normalized-edge-closeness} gives
$$
\sup_{B(0,1)}|\widetilde E(z)-z|
\le
\frac{C(n,p)}{\rho}\delta
\le
C(n,p)\delta,
$$
since $\rho\ge3/8$ by \eqref{eq:defn-rho}.  For sufficiently small $\delta$, \Cref{prop:finite-frame-matrix-control} together with \eqref{eq:small-dist} implies
\begin{equation}\label{eq:tilde-edge-matrix}
\|\mathcal R(\widetilde E)-I_{n+2}\|_{\mathrm{op}}
\le
C(n,p)\delta.
\end{equation}

By \eqref{eq:defn-rho}, the parameters satisfy
$$
|b|\le\sqrt n,
\qquad
\frac38\le\rho\le\frac32.
$$
Hence the similarities $L_{b,\rho}$ and their inverses form a compact subset of $\Mob(n)$.  Then continuity of the finite-dimensional representation gives
\begin{equation}\label{eq:normalization-norm}
\|\mathcal R(L_{b,\rho})\|_{\mathrm{op}}
+\|\mathcal R(L_{b,\rho})^{-1}\|_{\mathrm{op}} \le C(n).
\end{equation}

Since
$$
E_{P,Q} = L_{b,\rho}^{-1}\circ\widetilde E\circ L_{b,\rho},
$$
\eqref{eq:lorentz-homomorphism} yields
$$
\mathcal R(E_{P,Q})-I_{n+2} = \mathcal R(L_{b,\rho})^{-1}
\bigl(\mathcal R(\widetilde E)-I_{n+2}\bigr)
\mathcal R(L_{b,\rho}).
$$
Thus, the estimate \eqref{eq:edge-Lorentz-matrix} follows from \eqref{eq:tilde-edge-matrix} and \eqref{eq:normalization-norm}.
\end{proof}

\subsection{Quantitative M\"obius propagation}

The next proposition turns local uniform closeness of M\"obius transformations into Euclidean $C^1$-closeness on an arbitrary fixed larger ball.  In particular, it also proves that there is no pole on that ball. 

\begin{proposition}\label{prop:Mobius-propagation}
Let $H\ge1$.  There exist constants
$$
\eta_H=\eta_H(n,H)>0,
\qquad C_H=C_H(n,H)>0,
$$
such that the following holds.  Let $a\in\R^n$, $r>0$, and let $T\in\Mob(n)$ be finite on $B(a,r)$.  Suppose that,
\begin{equation}
\sup_{x\in B(a,r)}|T(x)-x| \le\eta r
\qquad \text{ for some } \ 0<\eta\le\eta_H,
\end{equation}
then $T$ is finite on $B(a,Hr)$ and
\begin{equation}\label{eq:propagation-conclusion}
\sup_{x\in B(a,Hr)}
\left( \frac{|T(x)-x|}{r} +|DT(x)-I_n|
\right)\le C_H\eta.
\end{equation}
\end{proposition}

\begin{proof}
Conjugating $T$ by the similarity
$$
L_{a,r}(x):=\frac{x-a}{r},
$$
we reduce to the case where $a=0$ and $r=1$.  The normalized map
$$
\widetilde T
:=
L_{a,r}\circ T\circ L_{a,r}^{-1}
$$
satisfies
$$
\sup_{B(0,1)}|\widetilde T(x)-x| \le\eta.
$$
By \Cref{prop:finite-frame-matrix-control},
\begin{equation}
\|\mathcal R(\widetilde T)-I_{n+2}\|_{\mathrm{op}}
\le C(n)\eta.
\end{equation}
Up to decreasing $\eta_H$, the matrix norm of $\mathcal R(\widetilde T)$ is at most $2$.  Applying \Cref{lem:projective-difference} with $M=\mathcal R(\widetilde T)$ and $N=I_{n+2}$, we obtain
\begin{equation}\label{eq:sphere-near-identity}
\|\widehat{\widetilde T}-\operatorname{id}_{\Sn}\|_{C^1(\Sn)} \le C(n)\eta.
\end{equation}

Let
$$
\mathscr K_H:=\sigma(\overline{B(0,H)})\subset\Sn.
$$
This compact set does not contain the north pole $\mathbf{e}_{n+1}=\sigma(\infty)$.  Put
$$
d_H:=\dist(\mathscr K_H,\mathbf{e}_{n+1})>0.
$$
If $\eta_H$ is chosen so that the right-hand side of \eqref{eq:sphere-near-identity} is at most $d_H/4$, then
$$
\dist(\widehat{\widetilde T}(\mathscr K_H),\mathbf{e}_{n+1}) \ge \frac{3d_H}{4}.
$$
Hence $\widetilde T$ has no pole in $\overline{B(0,H)}$.

Let
$$
\mathcal N_H := \bigl\{\xi\in\overline{\mathbb B}^{\,n+1}: \dist(\xi,\mathscr K_H)\le d_H/2\bigr\}.
$$
This is a compact subset of the open half-space $\{\xi_{n+1}<1\}$.  Indeed, for $\xi\in\mathcal N_H$, the triangle inequality gives
$$
|\xi-\mathbf{e}_{n+1}|\ge d_H/2.
$$
Note that, the formula
$$
\Sigma^{-1}(\xi) :=\frac{\xi'}{1-\xi_{n+1}}
$$
defines a smooth ambient extension of $\sigma^{-1}$ to that half-space.  Consequently,
\begin{equation}\label{eq:inverse-stereo-tube-C2}
\|D\Sigma^{-1}\|_{L^\infty(\mathcal N_H)}
+\|D^2\Sigma^{-1}\|_{L^\infty(\mathcal N_H)} \le C(n,H).
\end{equation}
For $x\in\overline{B(0,H)}$, put
$$
u(x):=\widehat{\widetilde T}(\sigma(x)),
\qquad
v(x):=\sigma(x).
$$
Then $v(x)\in\mathscr K_H$, and \eqref{eq:sphere-near-identity} gives
$$
|u(x)-v(x)|\le d_H/4
$$
when $\eta_H$ is sufficiently small.  Since the closed unit ball is convex, every point of the line segment joining $u(x)$ to $v(x)$ belongs to $\overline{\mathbb B}^{\,n+1}$, with its distance from $\mathscr K_H$  bounded from above by  $|u(x)-v(x)|$.  Hence the entire segment is contained in $\mathcal N_H$.

On $B(0,H)$,
$$
\widetilde T=\Sigma^{-1}\circ u, \qquad \operatorname{id}_{\R^n}=\Sigma^{-1}\circ v.
$$
Then the fundamental theorem of calculus, \eqref{eq:inverse-stereo-tube-C2}, and \eqref{eq:sphere-near-identity} give
$$
\sup_{B(0,H)}|\widetilde T(x)-x| \le C(n,H)\eta.
$$
For the derivative, write $\widetilde T=\Sigma^{-1}\circ u$ and $\operatorname{id}=\Sigma^{-1}\circ v$.  Then
\begin{align*}
 D\widetilde T-I_n = D\Sigma^{-1}(u)(Du-Dv) +\bigl[D\Sigma^{-1}(u)-D\Sigma^{-1}(v)\bigr]Dv.
\end{align*}
The first term is bounded by \eqref{eq:inverse-stereo-tube-C2} times the $C^1$-difference in \eqref{eq:sphere-near-identity}.  
Indeed, since $u=\widehat{\widetilde T}\circ\sigma$ and $v=\sigma$, one has
$$
Du-Dv=\bigl(D\widehat{\widetilde T}(\sigma(x))-I_{n+1}\bigr)D\sigma(x).
$$
Notice that, the map $D\sigma$ is bounded on $\overline{B(0,H)}$, and \eqref{eq:sphere-near-identity} controls $D\widehat{\widetilde T}-I$ on tangent vectors.
Hence
$$\|Du-Dv\|_{L^\infty(B(0,H))}\le C(n,H)\eta,
$$
and the desired bound follows. 
For the second term, the line segment from $u(x)$ to $v(x)$ lies in $\mathcal N_H$, so the fundamental theorem of calculus and the second-order derivative bound in \eqref{eq:inverse-stereo-tube-C2} give
$$
 |D\Sigma^{-1}(u)-D\Sigma^{-1}(v)| \le C(n,H)|u-v|.
$$
Since $Dv=D\sigma$ is bounded on $\overline{B(0,H)}$ and $|u-v|\le C(n)\eta$ by \eqref{eq:sphere-near-identity}, both terms are bounded by $C(n,H)\eta$.  Hence
$$
\sup_{B(0,H)}|D\widetilde T-I_n| \le C(n,H)\eta.
$$
This proves \eqref{eq:propagation-conclusion} in normalized coordinates.  Conjugation by $L_{a,r}$ gives the stated factors in the original coordinates.
\end{proof}

We  now complete the estimate on neighboring local normalizers.

\begin{corollary} \label{cor:edge-propagation}
Let $H\ge1$.  There exist constants
$$
\delta_{\mathrm{edge}}=\delta_{\mathrm{edge}}(n,p,H)>0,
\qquad
C_{\mathrm{edge}}=C_{\mathrm{edge}}(n,p,H)>0,
$$
such that the following holds.  Let $P,Q\in\mathscr W$ be adjacent and let
$$
\Theta_{P,Q}=\varphi_P^{-1}\circ\varphi_Q.
$$
If $0<\delta\le\delta_{\mathrm{edge}}$, then $\Theta_{P,Q}$ is finite on $B(x_Q,H\ell(Q))$
and
\begin{equation}\label{eq:edge-propagated-C1}
\sup_{y\in B(x_Q,H\ell(Q))}
\left(
\frac{|\Theta_{P,Q}(y)-y|}{\ell(Q)}
+
|D\Theta_{P,Q}(y)-I_n|
\right)
\le
C_{\mathrm{edge}}\delta.
\end{equation}
Equivalently, the normalized transition $E_{P,Q}$ is finite on $B(0,H)$ and satisfies
\begin{equation}\label{eq:normalized-edge-C1}
\sup_{z\in B(0,H)}
\bigl(
|E_{P,Q}(z)-z|
+
|DE_{P,Q}(z)-I_n|
\bigr)
\le
C_{\mathrm{edge}}\delta.
\end{equation}
\end{corollary}

\begin{proof}
Let $B(b,\rho)$ be the normalized bridge ball from \Cref{cor:normalized-bridge-data}, where its parameters satisfy
$$
|b|\le\sqrt n,
\qquad
\frac38\le\rho\le\frac32.
$$
Choose
$$
H_0:=\frac{8}{3}(H+\sqrt n).
$$
Then
$$
B(0,H)\subset B(b,H_0\rho).
$$
Indeed, if $|z|\le H$, then
$$
|z-b|\le H+\sqrt n\le H_0\rho,
$$
as $\rho\ge3/8$.

Now the normalized bridge estimate gives
$$
\sup_{B(b,\rho)}|E_{P,Q}-\operatorname{id}|
\le C(n,p)\delta.
$$
After dividing by $\rho\ge 3/8$, the smallness hypothesis in \Cref{prop:Mobius-propagation} is satisfied for $\delta\le\delta_{\mathrm{edge}}(n,p,H)$.  That proposition, applied with enlargement factor $H_0$, yields \eqref{eq:normalized-edge-C1}.  Rescaling by $S_Q^{-1}$ gives \eqref{eq:edge-propagated-C1}.
\end{proof}

\begin{remark}
We record separately the two conclusions used later:
\begin{enumerate}[label=\textup{(\roman*)},leftmargin=2.2em]
\item Formula \eqref{eq:edge-propagated-C1} is the Euclidean input for the Boman-chain argument on a John domain;
\item Formula  \eqref{eq:edge-Lorentz-matrix} is the matrix input for the Whitney-chain argument on a general domain.
\end{enumerate}
\end{remark}

\section{The John-domain globalization}\label{sec:john}

Throughout this section, $\Omega\subset\R^n$ is a bounded $J$-John domain with John center $x_\ast$, in the sense of \Cref{def:john}.  We use the Whitney decomposition $\mathscr W$ from \Cref{def:deep-whitney}.  Up to translating the dyadic grid, if necessary, we may assume that the John center $x_\ast$ lies on no dyadic hyperplane.  Since the closures of the maximal cubes cover $\Omega$ and their interiors are disjoint, there is then a unique root cube $Q_\ast\in\mathscr W$ satisfying
\begin{equation}
 x_\ast\in Q_\ast.
\end{equation}
For each $Q\in\mathscr W$, the local M\"obius normalizer $\varphi_Q$ and the local  normalized transformation
$$
 g_Q=\varphi_Q^{-1}\circ f
$$
are those chosen in \Cref{prop:whitney-local-gauges}.

\subsection{Estimates on Boman chains}

The first lemma extracts from each John curve a simple Whitney chain.  The order in which the cubes are met by the curve is retained; this will be needed in the subsequent estimate.  It is a quantitative,  ordered   Boman-chain construction for John domains; see the original version by Bojarski in \cite[Lemma~3.1]{B1988} and also \cite[Theorem 9.3]{HK2000}. We postpone its proof to  the Appendix. 

\begin{lemma}\label{lem:john-whitney-chains}
There exist constants
$$
 \Lambda_J=\Lambda_J(n,J)\ge1,
 \qquad
 C_{\mathrm{suf}}=C_{\mathrm{suf}}(n,J)\ge1
$$
with the following property.  For every  cube $Q\in\mathscr W$, there is a finite sequence of distinct cubes
\begin{equation}\label{eq:john-chain}
 Q_0=Q_\ast,Q_1,\ldots,Q_m=Q
\end{equation}
such that
$$
 Q_{j-1}\sim Q_j  \qquad \text{ for any } \ 1\le j\le m,
$$
\begin{equation}\label{eq:john-terminal-shadow}
 \overline{2Q}\subset \Lambda_JQ_j
 \qquad \text{ for any } \  0\le j\le m,
\end{equation}
and
\begin{equation}\label{eq:john-suffix-sum}
 \sum_{k=j}^{m}\ell(Q_k) \le C_{\mathrm{suf}}\ell(Q_j)
 \qquad\text{ for any } \ 0\le j\le m.
\end{equation}
The integer $m=m(Q)$ may depend on $Q$.
\end{lemma}

For each terminal cube $Q\in\mathscr W$, fix one chain supplied by \Cref{lem:john-whitney-chains}.  Define its depth by
$ N(Q):=m(Q),$
and define the measurable function
\begin{equation}\label{eq:defn-Nx}
 N(x):=N(Q)  \quad\text{for almost every }x\in Q,
 \qquad Q\in\mathscr W.
\end{equation}
The interiors of the Whitney cubes are disjoint, so this definition is unambiguous almost everywhere.

Now we introduce the maximal operator and Bojarski's dilation estimate. 
For a locally integrable function $h$ on $\R^n$, define the uncentered cube maximal function by
\begin{equation*}
 \mathcal Mh(x)
 :=\sup_{R\ni x}\fint_R|h(y)|\,dy,
\end{equation*}
where the supremum is taken over all axis-parallel cubes $R$ containing $x$. We record the following well-known Hardy--Littlewood estimate. 

\begin{lemma}[{\cite[Chapter 2.1.1]{G2004}}] \label{lem:HL-maximal}
For every $1<s<\infty$, there is a constant $C_M=C_M(n,s)$ such that
$$
 \|\mathcal Mh\|_{L^s(\R^n)} \le C_M(n,s)\|h\|_{L^s(\R^n)}.
$$
Moreover, if $s=q'=q/(q-1)$ with $q\ge2$, then
\begin{equation}\label{eq:HL-qprime-bound}
 C_M(n,q')\le C(n)q.
\end{equation}
\end{lemma}

The following estimate is due to Bojarski, and we record its proof for completeness. This is the only maximal-function estimate needed in the John-domain argument.

\begin{lemma}[{\cite[Lemma~4.2]{B1988}}]\label{lem:Bojarski-dilation}
Let $1<q<\infty$, let $\Lambda\ge1$, let $\{R_\alpha\}_{\alpha\in A}$ be a finite family of axis-parallel cubes, and let $a_\alpha\ge0$.  Then
\begin{equation}\label{eq:Bojarski-dilation}
 \left\|
  \sum_{\alpha\in A}a_\alpha\mathbf 1_{\Lambda R_\alpha}
 \right\|_{L^q(\R^n)}
 \le
 \Lambda^n C_M(n,q')
 \left\|
  \sum_{\alpha\in A}a_\alpha\mathbf 1_{R_\alpha}
 \right\|_{L^q(\R^n)},
\end{equation}
where $q'=q/(q-1)$.  The same conclusion holds for a countable family whenever the right-hand side is finite.
\end{lemma}

\begin{proof}
The function on the left of \eqref{eq:Bojarski-dilation} is nonnegative.  By $L^q$--$L^{q'}$ duality, it suffices to test against a nonnegative function $\phi\in L^{q'}(\R^n)$ with $\|\phi\|_{L^{q'}}=1$.  Since $R_\alpha\subset\Lambda R_\alpha$, for every $x\in R_\alpha$,
$$
 \mathcal M\phi(x)
 \ge\fint_{\Lambda R_\alpha}\phi.
$$
Therefore
\begin{align*}
 \int_{\R^n}
 \left(\sum_{\alpha\in A}a_\alpha\mathbf 1_{\Lambda R_\alpha}\right)\phi
 &=\sum_{\alpha\in A}a_\alpha|\Lambda R_\alpha|
   \fint_{\Lambda R_\alpha}\phi\\
 &\le\Lambda^n\sum_{\alpha\in A}a_\alpha
   \int_{R_\alpha}\mathcal M\phi\\
 &=\Lambda^n\int_{\R^n}
   \left(\sum_{\alpha\in A}a_\alpha\mathbf 1_{R_\alpha}\right)
   \mathcal M\phi\\
 &\le\Lambda^n
 \left\|\sum_{\alpha\in A}a_\alpha\mathbf 1_{R_\alpha}\right\|_{L^q}
 \|\mathcal M\phi\|_{L^{q'}}.
\end{align*}
Apply \Cref{lem:HL-maximal} and take the supremum over $\phi$.  The countable-family statement follows by applying the finite result to increasing finite subfamilies and using monotone convergence.
\end{proof}

Now we can prove the exponential integrability of the chain depth in John domain. See also the proof of \cite[Theorem 2.1]{KR1997}. 
\begin{corollary} \label{cor:john-depth-exponential} Recall $N(x)$ defined in \eqref{eq:defn-Nx}. There are constants
$$
 \tau_J=\tau_J(n,J)>0,
 \qquad
 C_{\mathrm{depth}}=C_{\mathrm{depth}}(n,J)<\infty
$$
such that
\begin{equation}\label{eq:john-depth-exponential}
 \int_\Omega e^{\tau_JN(x)}\,dx
 \le C_{\mathrm{depth}}|\Omega|.
\end{equation}
\end{corollary}

\begin{proof}
Fix $x$ in the interior of a terminal cube $Q$.  The chain cubes are pairwise disjoint, and \eqref{eq:john-terminal-shadow} gives
$$
 x\in Q\subset\Lambda_JQ_j  \qquad \text{ for any} \ 0\le j\le N(Q).
$$
Hence, it follows that
$$
 N(x)+1 \le\sum_{R\in\mathscr W}\mathbf 1_{\Lambda_JR}(x)
 \quad\text{for almost every }x\in\Omega.
$$
Let $q\ge2$.  Apply \Cref{lem:Bojarski-dilation} first to a finite subfamily of $\mathscr W$, with all coefficients equal to one.  Since the  interiors of the Whitney cubes are disjoint,
$$
 \left\|\sum_{R\in\mathscr F}\mathbf 1_R\right\|_{L^q(\R^n)}
 \le|\Omega|^{1/q}.
$$
Using \eqref{eq:HL-qprime-bound} and then monotone convergence as the finite subfamilies increase to $\mathscr W$, we obtain
\begin{equation}\label{eq:depth-Lq}
 \|N+1\|_{L^q(\Omega)}
 \le C(n,J)q|\Omega|^{1/q}
 \qquad \text{for any} \ q\ge2.
\end{equation}

Now for every integer $k\ge2$, \eqref{eq:depth-Lq} gives
$$
 \int_\Omega N(x)^k\,dx \le[C(n,J)k]^k|\Omega|.
$$
The case $k=1$  follows from the estimate with $q=2$. Indeed,
$$
\int_\Omega N(x)\,dx \le |\Omega|^{1/2}\|N\|_{L^2(\Omega)} \le C(n,J)|\Omega|.
$$
Applying
$$
 k!\ge\left(\frac{k}{e}\right)^k,
$$
we may choose $\tau_J>0$ so small that
$$
 C(n,J)\tau_J\le\frac1 {2e}.
$$
Then
\begin{align*}
 \frac1{|\Omega|}\int_\Omega e^{\tau_JN(x)}\,dx
 &=1+\sum_{k=1}^{\infty}
   \frac{\tau_J^k}{k!|\Omega|}
   \int_\Omega N(x)^k\,dx\\
 &\le C(n,J)+\sum_{k=2}^{\infty}2^{-k} \le C_{\mathrm{depth}}(n,J).
\end{align*}
This is exactly \eqref{eq:john-depth-exponential}.
\end{proof}

\subsection{Comparison on chains}

For a Boman chain \eqref{eq:john-chain}, define the edge transitions
\begin{equation}
 \Theta_j  :=\varphi_{Q_{j-1}}^{-1}\circ\varphi_{Q_j}
 \qquad \text{ for each } \ 1\le j\le m.
\end{equation}
Recall $C_{\mathrm{suf}}$ in \eqref{eq:john-suffix-sum}.

\begin{proposition} \label{prop:john-backward-orbit}
There exist constants
$$
 H_J=H_J(n,J)\ge 1, \qquad \delta_{\mathrm{orb}}=\delta_{\mathrm{orb}}(n,p,J)>0,
 \qquad C_{\mathrm{orb}}=C_{\mathrm{orb}}(n,p,J)<\infty
$$
with the following property.  Assume $0<\delta\le\delta_{\mathrm{orb}}$.  Let $Q_0,\ldots,Q_m$ be the fixed chain associated with a cube $Q=Q_m$, and let $x\in\overline Q$.  Define
\begin{equation}\label{eq:backward-orbit-definition}
 z_m:=g_{Q_m}(x),
 \qquad
 z_{j-1}:=\Theta_j(z_j)
 \quad(j=m,m-1,\ldots,1).
\end{equation}
Then every term in \eqref{eq:backward-orbit-definition} is finite and
\begin{equation}\label{eq:backward-orbit-sum-estimate}
 |z_j-x|
 \le C_{\mathrm{orb}}\delta
      \sum_{k=j}^{m}\ell(Q_k)
 \le C_{\mathrm{orb}}C_{\mathrm{suf}}\delta\ell(Q_j)
 \qquad \text{ for any } \ 0\le j\le m.
\end{equation}
Moreover,
\begin{equation}\label{eq:orbit-identification}
 z_j=\varphi_{Q_j}^{-1}(f(x))
 \qquad \text{ for any} \ 0\le j\le m,
\end{equation}
where the equality is interpreted in $\whR$ and the left-hand side proves that the value is finite.
\end{proposition}

\begin{proof}
Let $\Lambda_J$ and $C_{\mathrm{suf}}$ be the constants in \Cref{lem:john-whitney-chains}.  Choose
\begin{equation}\label{eq:HJ-choice}
 H_J:=2+\frac{\sqrt n}{2}\Lambda_J.
\end{equation}
Apply \Cref{cor:edge-propagation} with this fixed value of $H_J$,   and denote the resulting  constants by
$$
\delta_{\rm edge}=\delta_{\rm edge}(n,p,H_J)>0,\qquad 
 C_{\mathrm{edge}}=C_{\mathrm{edge}}(n,p,J).
$$
Recall $\delta_{\rm loc}>0$ in  \Cref{thm:local-main} and  $C_{\mathrm{loc}}$ in \Cref{prop:whitney-local-gauges},  and set
$$
 C_{\mathrm{base}}:=C_{\mathrm{loc}}'(n,p),
 \qquad  C_{\mathrm{step}}:=C_{\mathrm{edge}},
 \qquad C_{\mathrm{orb}}:= C_{\mathrm{base}}+C_{\mathrm{step}}.
$$

We choose $\delta_{\mathrm{orb}}>0$ so that 
\begin{equation}\label{eq:orbit-smallness-choice}
 \delta_{\mathrm{orb}}\le \min\{\delta_{\rm{loc}},\,\delta_{\rm edge}\}
\qquad C_{\mathrm{orb}}C_{\mathrm{suf}}\delta_{\mathrm{orb}}\le1.
\end{equation}

Since $x\in\overline Q\subset8Q$, the local estimate \eqref{eq:whitney-local-uniform}, together with continuity of $g_Q$, gives
$$
 |z_m-x| \le C_{\mathrm{base}}\delta\ell(Q_m).
$$
In particular, $z_m$ is finite.  We now proceed by descending induction.  Suppose $z_j,\ldots,z_m$ have been defined and that
\begin{equation}\label{eq:orbit-inductive-estimate}
 |z_j-x| \le  C_{\mathrm{base}}\delta\ell(Q_m)
 +C_{\mathrm{step}}\delta
  \sum_{k=j+1}^{m}\ell(Q_k),
\end{equation}
where the sum is empty when $j=m$.  In particular,
\begin{equation}\label{eq:orbit-inductive-coarse}
 |z_j-x|
 \le C_{\mathrm{orb}}\delta
      \sum_{k=j}^{m}\ell(Q_k).
\end{equation}
By \eqref{eq:john-terminal-shadow},
$x\in\overline Q\subset\overline{2Q} \subset\Lambda_JQ_j,$
and hence
$$
 |x-x_{Q_j}|
 \le\frac{\sqrt n}{2}\Lambda_J\ell(Q_j).
$$
Combining this with  \eqref{eq:john-suffix-sum}, \eqref{eq:HJ-choice}, \eqref{eq:orbit-smallness-choice} and \eqref{eq:orbit-inductive-coarse}, we arrive at
\begin{align*}
 |z_j-x_{Q_j}|
 &\le|z_j-x|+|x-x_{Q_j}|\\
 &\le C_{\mathrm{orb}}C_{\mathrm{suf}}\delta\ell(Q_j)
      +\frac{\sqrt n}{2}\Lambda_J\ell(Q_j)\\
 &\le(H_J-1)\ell(Q_j) <H_J\ell(Q_j).
\end{align*}
Thus $z_j$ lies in the pole-free ball from \Cref{cor:edge-propagation}, applied to the adjacent pair $Q_{j-1},Q_j$.  The value
$$
 z_{j-1}=\Theta_j(z_j)
$$
is finite and satisfies
$$
 |z_{j-1}-z_j|
 \le C_{\mathrm{step}}\delta\ell(Q_j).
$$
Therefore,
\begin{align*}
 |z_{j-1}-x| \le |z_j-x|+|z_{j-1}-z_j| \le C_{\mathrm{base}}\delta\ell(Q_m)
 +C_{\mathrm{step}}\delta
  \sum_{k=j}^{m}\ell(Q_k).
\end{align*}
This is exactly \eqref{eq:orbit-inductive-estimate} with $j-1$ in place of $j$.  The induction closes, and \eqref{eq:orbit-inductive-coarse}, together with \eqref{eq:john-suffix-sum}, proves \eqref{eq:backward-orbit-sum-estimate}.

The identity in \eqref{eq:orbit-identification} holds for $j=m$ by the definition $g_{Q_m}=\varphi^{-1}_{Q_m}\circ f$.  Now assuming it holds at level $j$, we obtain, first as an identity of M\"obius transformations on $\whR$,
\begin{equation*}
 z_{j-1} = \Theta_j(z_j)  = \varphi_{Q_{j-1}}^{-1} \circ \varphi_{Q_j}\bigl(\varphi_{Q_j}^{-1}(f(x))\bigr)
 = \varphi_{Q_{j-1}}^{-1}(f(x)).
\end{equation*}
Thus finiteness of $z_{j-1}$ established in the previous induction  implies that no pole arises in this step. This concludes the proposition. 
\end{proof}

As a consequence, we obtain the following uniform control of the root normalization. 
\begin{corollary}\label{cor:john-root-uniform}
Let
$\Phi:=\varphi_{Q_\ast}.$
Under the assumptions of \Cref{prop:john-backward-orbit}, the map
$$
 G:=\Phi^{-1}\circ f
$$
is finite throughout $\Omega$, and
\begin{equation}\label{eq:john-global-uniform}
 \|G-\operatorname{id}\|_{L^\infty(\Omega)}
 \le C(n,p,J)\delta\,\diam(\Omega).
\end{equation}
\end{corollary}

\begin{proof}
Fix $x\in\Omega$.  By the covering property of the Whitney decomposition, choose $Q\in\mathscr W$ with $x\in\overline Q$.  Applying \Cref{prop:john-backward-orbit} at $j=0$ and using $Q_0=Q_\ast$ gives $ G(x)=z_0$
and
$$
 |G(x)-x|
 \le C\delta\sum_{k=0}^{m}\ell(Q_k)
 \le C\delta\ell(Q_\ast).
$$
Since $Q_\ast\subset\Omega$,
$$
 \ell(Q_\ast)\le\diam(\Omega).
$$
This proves \eqref{eq:john-global-uniform} and the finiteness of $G$ at every point.
\end{proof}

\subsection{Proof of \Cref{thm:john-main}}
We first introduce the following auxiliary lemma. 
\begin{lemma} \label{lem:near-identity-products}
Let $A_1,\ldots,A_m\in\R^{n\times n}$ and suppose
$$
 |A_j-I_n|\le a
 \qquad \text{ for any } \ 1\le j\le m
$$
for some $a\ge0$.  Then
\begin{equation}\label{eq:matrix-product-size}
 \|A_1\cdots A_m\|_{\mathrm{op}} \le e^{am}
\end{equation}
and
\begin{equation}\label{eq:matrix-multi}
 |A_1\cdots A_m-I_n|  \le am e^{am}.
\end{equation}
\end{lemma}

\begin{proof}
Since $\|A_j-I_n\|_{\mathrm{op}}\le|A_j-I_n|\le a$, one has
$$
 \|A_j\|_{\mathrm{op}}\le 1+a.
$$
Therefore, 
$$
 \|A_1\cdots A_m\|_{\mathrm{op}}  \le(1+a)^m  \le e^{am},
$$
which proves \eqref{eq:matrix-product-size}. 

Moreover, the telescoping identity
$$
 A_1\cdots A_m-I_n  =\sum_{j=1}^{m} A_1\cdots A_{j-1}(A_j-I_n)
$$
gives
\begin{align*}
 |A_1\cdots A_m-I_n|
 &\le\sum_{j=1}^{m} \|A_1\cdots A_{j-1}\|_{\mathrm{op}}|A_j-I_n|\\
 &\le a\sum_{j=1}^{m}(1+a)^{j-1} \le am e^{am}.
\end{align*}
This concludes \eqref{eq:matrix-multi}. 
\end{proof}

\begin{proof}[Proof of \Cref{thm:john-main}]
Put
$$
 \delta=K-1,
 \qquad
 \Phi=\varphi_{Q_\ast},
 \qquad
 G=\Phi^{-1}\circ f.
$$
Let $H_J$ be given by \eqref{eq:HJ-choice}, and  
$$
C_{\mathrm{edge}}=C_{\mathrm{edge}}(n,p,J)
$$
be the constant in \Cref{cor:edge-propagation} for this value of $H_J$.
Initially choose
$$
\delta_J(n,p,J)\le\delta_{\mathrm{orb}}(n,p,J).
$$
We impose one further restriction below in terms of the exponential-depth constant $\tau_J$.

By \Cref{cor:john-root-uniform}, $G$ is finite throughout $\Omega$ and satisfies \eqref{eq:john-global-uniform}.  Thus it remains to prove the derivative estimate.

Fix a   cube $Q=Q_m$ and its chain $Q_0,\ldots,Q_m$.  For $x\in Q$ outside the null set where one of the Sobolev derivatives fails to exist, let $z_j=z_j(x)$ be the backward orbit from \Cref{prop:john-backward-orbit} and set
\begin{equation}
 A_j(x):=D\Theta_j(z_j(x)) \qquad \text{ for any } \ 1\le j\le m.
\end{equation}
Recall the choice $H=H_J$, and then the orbit remains in the balls on which \eqref{eq:edge-propagated-C1} holds.  Consequently,
\begin{equation}
 |A_j(x)-I_n| \le C_{\mathrm{edge}}\delta.
\end{equation}
Moreover, for fixed $Q$, the chain has finite length.  By
\Cref{prop:john-backward-orbit}, every point of the orbit stays in a ball on
which the corresponding transition $\Theta_j$ is smooth and pole-free.
Starting from $g_Q\in W^{1,p}(Q;\mathbb R^n)$, repeated application of the Sobolev chain rule to 
$$
G=\Theta_1\circ\Theta_2\circ\cdots\circ\Theta_m\circ g_Q
\quad\text{in }W^{1,p}_{\mathrm{loc}}(Q;\mathbb R^n),
$$
yields that,  for almost every $x\in Q$,
\begin{align}\label{eq:john-chain-rule}
DG(x) &= D\Theta_1(z_1(x))\cdots D\Theta_m(z_m(x))Dg_Q(x)\notag \\
 &=A_1(x)\cdots A_m(x)Dg_Q(x).
\end{align}

Set
$$
 a=C_{\mathrm{edge}}\delta, \qquad  P_m(x)=A_1(x)\cdots A_m(x).
$$
Then \Cref{lem:near-identity-products} yields
$$
 \|P_m(x)\|_{\mathrm{op}} \le e^{aN(Q)},
 \qquad  |P_m(x)-I_n| \le aN(Q)e^{aN(Q)}.
$$
Employing
$$
 DG-I_n=P_m(Dg_Q-I_n)+(P_m-I_n),
$$
we obtain
\begin{align}
 |DG(x)-I_n|^p
 \le{}&2^{p-1}e^{paN(Q)}|Dg_Q(x)-I_n|^p\notag\\
 &+2^{p-1}a^pN(Q)^p e^{paN(Q)}.
 \label{eq:john-pointwise-derivative-bound}
\end{align}

Observe that, since $Q\subset8Q$ and $|8Q|=8^n|Q|$, the local estimate \eqref{eq:whitney-local-derivative} gives
\begin{equation}\label{eq:local-Q-derivative-integral}
 \int_Q|Dg_Q-I_n|^p\,dx  \le C(n,p)\delta^p|Q|.
\end{equation}
Now summing \eqref{eq:john-pointwise-derivative-bound} over the Whitney interiors and using \eqref{eq:local-Q-derivative-integral},
\begin{equation}\label{eq:john-global-derivative-before-depth}
\int_\Omega|DG-I_n|^p\,dx
 \le C(n,p,J)\left(\delta^p \int_\Omega e^{paN(x)}\,dx + \delta^p  \int_\Omega N(x)^p e^{paN(x)}\,dx\right).
\end{equation}

Let $\tau_J$ be the constant given by \Cref{cor:john-depth-exponential}.  Define
\begin{equation}
 \delta_J
 :=\min\left\{
 \delta_{\mathrm{orb}}(n,p,J),
 \frac{\tau_J}{2pC_{\mathrm{edge}}}
 \right\}.
\end{equation}
Thus $pC_{\mathrm{edge}}\delta\le\tau_J/2$ whenever $0<\delta\le\delta_J$.
Then
$$
 e^{paN(x)}\le e^{\tau_JN(x)/2}  \le e^{\tau_JN(x)}.
$$
Moreover,
$$
 N(x)^p e^{paN(x)}  \le C(p,\tau_J)e^{\tau_JN(x)},
$$
since
$ \sup_{t\ge0}t^p e^{-\tau_Jt/2}<\infty.$
Substituting these bounds into \eqref{eq:john-global-derivative-before-depth} and applying \eqref{eq:john-depth-exponential}, we obtain
\begin{equation}\label{eq:john-global-derivative}
 \int_\Omega|DG-I_n|^p\,dx
 \le C(n,p,J)\delta^p|\Omega|.
\end{equation}

For completeness, we show that $G$ belongs to $W^{1,p}(\Omega;\R^n)$.  Indeed, $G$ is finite throughout $\Omega$.  For every $x_0\in\Omega$, continuity of $f$ gives a neighborhood $U\subset\subset\Omega$ of $x_0$ such that $f(U)$ stays in a compact subset of the finite locus of $\Phi^{-1}$.  The   Sobolev chain rule therefore gives 
$$
G=\Phi^{-1}\circ f\in W^{1,n}_{\mathrm{loc}}(\Omega;\mathbb R^n).
$$
Formula \eqref{eq:john-chain-rule} identifies its weak derivative almost everywhere on each Whitney cube, hence almost everywhere in $\Omega$ since the union of Whitney cube boundaries has measure zero.
The estimate \eqref{eq:john-global-derivative}  gives $DG\in L^p(\Omega)$. Then the uniform bound \eqref{eq:john-global-uniform}, together with the boundedness of $\Omega$, upgrades $G$ to $W^{1,p}(\Omega)$. Now taking the $p$-th root of \eqref{eq:john-global-derivative} and combining it with \eqref{eq:john-global-uniform} proves \Cref{thm:john-main}.
\end{proof}

\section{The weighted theorem on bounded connected domains}\label{sec:weighted}

Throughout this section, $\Omega\subset\R^n$ is bounded, open, and connected.  We use   Whitney family $\mathscr W$ given by \Cref{lem:deep-whitney}, choose a cube $Q_\ast\in\mathscr W$ of maximal sidelength given by \Cref{cor:max-cube}, and write
$ \ell_\ast:=\ell(Q_\ast).$
Then
\begin{equation}\label{eq:bounded-side-bound}
\ell(Q)\le \ell_\ast \qquad \text{ for any } \ Q\in\mathscr W.
\end{equation}
The local normalizer $\varphi_Q$ and the local normalized function
$$
 g_Q=\varphi_Q^{-1}\circ f
$$
for $Q\in\mathscr W$ are those fixed in \Cref{prop:whitney-local-gauges}.  In this section we use the normalized Lorentz estimate \eqref{eq:edge-Lorentz-matrix}. 

We define the breadth-first Whitney tree as follows. 
\begin{definition} \label{def:breadth-first-tree}
Let $\mathscr G$ be the adjacency graph of $\mathscr W$, with root $Q_\ast$.  For $Q\in\mathscr W$, let
$$
 d_T(Q):=\dist_{\mathscr G}(Q,Q_\ast)
$$
be the graph distance from $Q_\ast$ to $Q$.  For every $Q\ne Q_\ast$, choose one adjacent cube $\operatorname{par}(Q)$ satisfying
$$
 d_T(\operatorname{par}(Q))=d_T(Q)-1.
$$
\end{definition}
We show that the edges
$$
 \{\operatorname{par}(Q),Q\},  \qquad Q\in\mathscr W\setminus\{Q_\ast\},
$$
form a rooted spanning tree $T$, and we call $d_T(Q)$
the \emph{depth} of $Q$.
\begin{lemma} 
The construction in \Cref{def:breadth-first-tree} gives a rooted spanning tree.  If $N_W=N_W(n)$ is the degree bound in \Cref{lem:deep-whitney}, then
\begin{equation}\label{eq:number-at-depth}
 \#\{Q\in\mathscr W:d_T(Q)=m\}
 \le N_W^m
 \qquad  \text{ for any }\ m\in\mathbb N_0.
\end{equation}
\end{lemma}

\begin{proof}
According to the construction, every nonroot vertex $Q$ has exactly one selected parent $\operatorname{par}(Q)$, and the depth decreases by one whenever one follows a parent edge $\{\operatorname{par}(Q),Q\}$.  Hence the selected graph contains no cycle.  Repeatedly following parents from a cube of depth $m$ reaches the root after exactly $m$ steps, so the selected graph is connected and contains every vertex.  It is therefore a spanning tree.

A vertex at depth $m$ is reached from the root by a walk of length $m$ in the original adjacency graph.  At each step there are at most $N_W$ choices.  The number of such walks is at most $N_W^m$, which proves \eqref{eq:number-at-depth}.
\end{proof}

For a nonroot cube $Q$, put
$$
 P:=\operatorname{par}(Q).
$$
Define the normalized similarity
\begin{equation}\label{eq:tree-CQ-definition}
 \mathfrak S_Q:=S_P\circ S_Q^{-1}
\end{equation}
and the normalized edge transition
\begin{equation}\label{eq:tree-EQ-definition}
 E_Q:=E_{P,Q}
 =S_Q\circ\varphi_P^{-1}\circ\varphi_Q\circ S_Q^{-1}.
\end{equation}

We next show the compactness of the inter-cube similarities.
\begin{lemma}
There is a constant $C_{\rm par}=C_{\rm par}(n)\ge1$ such that
\begin{equation}\label{eq:CQ-Lorentz-bound}
 \|\mathcal R(\mathfrak S_Q)\|_{\mathrm{op}}
 +\|\mathcal R(\mathfrak S_Q)^{-1}\|_{\mathrm{op}} \le C_{\rm par}
\end{equation}
for every nonroot cube $Q$.
\end{lemma}

\begin{proof}
Writing $P=\operatorname{par}(Q)$, by the definition of $S_Q$ in \eqref{eq:defn-SQ}  one has
\begin{equation}
 \mathfrak S_Q(y)
 =a_Q+\lambda_Qy,
 \qquad
 a_Q:=\frac{x_Q-x_P}{\ell(P)},
 \qquad
 \lambda_Q:=\frac{\ell(Q)}{\ell(P)}.
\end{equation}
Since $P\sim Q$, \eqref{eq:adjacent-size-comparability} gives
$$
 \frac14\le\lambda_Q\le4.
$$
Moreover, the intersection of $2P$ and $2Q$ implies
$$
 |x_Q-x_P|_\infty<\ell(P)+\ell(Q)\le5\ell(P),
$$
and hence
$$
 |a_Q|\le5\sqrt n.
$$
Thus the similarities $\mathfrak S_Q$ belong to a compact subset of the similarity group.  Their inverses belong to another compact subset.  Then the continuity of the finite-dimensional representation $\mathcal R$, proved in \Cref{prop:lorentz-realization}, gives \eqref{eq:CQ-Lorentz-bound}.
\end{proof}

\subsection{The root-relative chain}
Set
\begin{equation}\label{eq:weighted-root-normalizer}
 \Phi:=\varphi_{Q_\ast}.
\end{equation}
For $Q\in\mathscr W$, define the M\"obius transformation
\begin{equation}\label{eq:HQ-definition}
 H_Q
 :=S_Q\circ\Phi^{-1}\circ\varphi_Q\circ S_Q^{-1}.
\end{equation}
Observe that, at the root of the tree, 
\begin{equation}
 H_{Q_\ast}=\operatorname{id}_{\whR}.
\end{equation}

In addition, for any $Q\ne Q_\ast$ and $P=\operatorname{par}(Q)$, we note that
\begin{equation}\label{eq:exact-tree-recurrence}
 H_Q
 =\mathfrak S_Q^{-1}\circ H_P\circ \mathfrak S_Q\circ E_Q.
\end{equation}
Indeed, according to \eqref{eq:tree-CQ-definition}, \eqref{eq:tree-EQ-definition}, and \eqref{eq:HQ-definition}, we compute
\begin{align*}
 \mathfrak S_Q^{-1}\circ H_P\circ \mathfrak S_Q\circ E_Q
 ={}&
 S_Q\circ S_P^{-1}
 \circ S_P\circ\Phi^{-1}\circ\varphi_P\circ S_P^{-1}
 \circ S_P\circ S_Q^{-1}
 \\
 &\circ S_Q\circ\varphi_P^{-1}\circ\varphi_Q\circ S_Q^{-1}
 \\
 ={}&
 S_Q\circ\Phi^{-1}\circ\varphi_Q\circ S_Q^{-1}
 =H_Q.
\end{align*}
Note that each equality comes from an identity of M\"obius transformations on $\whR$, so no Euclidean pole restriction is needed in this calculation. Thus we conclude \eqref{eq:exact-tree-recurrence}. 

\begin{proposition} \label{prop:root-relative-Lorentz} 
Assume \eqref{eq:small-dist}. 
There exist constants
$$
 \delta_{\mathrm{tree}}=\delta_{\mathrm{tree}}(n,p)>0,
 \qquad  A=A(n,p)>1,
 \qquad  C_{\rm{tree}}=C_{\rm{tree}}(n,p)>0,
$$
such that, whenever $0<\delta\le\delta_{\mathrm{tree}}$, one has
\begin{equation}\label{eq:HQ-matrix-size}
 \|\mathcal R(H_Q)\|_{\mathrm{op}}
 \le A^{d_T(Q)}
\end{equation}
and
\begin{equation}\label{eq:HQ-matrix-difference}
 \|\mathcal R(H_Q)-I_{n+2}\|_{\mathrm{op}}
 \le C_{\rm tree}\delta A^{d_T(Q)}
\end{equation}
for every $Q\in\mathscr W$.
\end{proposition}

\begin{proof}
Let
$$
 \mathbf H_Q:=\mathcal R(H_Q),
 \qquad \mathbf S_Q:=\mathcal R(\mathfrak S_Q),
 \qquad \mathbf E_Q:=\mathcal R(E_Q).
$$
Up to decreasing $\delta_{\mathrm{tree}}>0$ if necessary, we may assume that
$ C_{\mathcal L}(n,p)\delta_{\mathrm{tree}}\le 1,$
and
$$
 \delta_{\mathrm{tree}} \le\min\{\delta_{\mathrm{loc}}(n,p),\delta_{\mathcal L}(n,p)\}, 
$$
where $\delta_{\rm{loc}}>0$ is determined in Theorem~\ref{thm:local-main} and $\delta_{\mathcal L}>0$ is given by \Cref{cor:lorentz}. 

By \eqref{eq:edge-Lorentz-matrix},
\begin{equation}\label{eq:EQ-size}
 \|\mathbf E_Q-I_{n+2}\|_{\mathrm{op}} \le C_{\mathcal L}\delta \quad \text{ and }
 \quad \|\mathbf E_Q\|_{\mathrm{op}} \le 1+C_{\mathcal L}\delta.
\end{equation}
By \eqref{eq:exact-tree-recurrence} and the homomorphism property of $\mathcal R$ \eqref{eq:lorentz-homomorphism},
\begin{equation}\label{eq:matrix-tree-recurrence}
 \mathbf H_Q =\mathbf S_Q^{-1}\mathbf H_P\mathbf S_Q\mathbf E_Q.
\end{equation}
Let $C_{\rm par}$ be the constant in \eqref{eq:CQ-Lorentz-bound}, and set
\begin{equation}\label{eq:defn-A}
 A:=2 C_{\rm par}^2 (1+C_{\mathcal L}\delta_{\mathrm{tree}})>1.
\end{equation}
Since $\delta<\delta_{\rm tree}$,  then \eqref{eq:CQ-Lorentz-bound}, \eqref{eq:EQ-size}, and \eqref{eq:matrix-tree-recurrence} give
$$
 \|\mathbf H_Q\|_{\mathrm{op}}  \le\frac A2\|\mathbf H_P\|_{\mathrm{op}}.
$$
Starting from $\mathbf H_{Q_\ast}=I_{n+2}$, induction on the depth gives, as $d_T(Q)\ge 1$,
$$
 \|\mathbf H_Q\|_{\mathrm{op}} \le\left(\frac A2\right)^{d_T(Q)} \le A^{d_T(Q)},
$$
which is \eqref{eq:HQ-matrix-size}.

For the difference from the identity, rewrite \eqref{eq:matrix-tree-recurrence} as
\begin{equation*}
 \mathbf H_Q-I_{n+2} =\mathbf  S_Q^{-1}(\mathbf H_P-I_{n+2})\mathbf   S_Q\mathbf E_Q +(\mathbf E_Q-I_{n+2}).
\end{equation*}
Therefore, \eqref{eq:EQ-size} and \eqref{eq:defn-A} yield
\begin{equation}\label{eq:difference-recursive-bound}
 \|\mathbf H_Q-I_{n+2}\|_{\mathrm{op}}
 \le\frac A2\|\mathbf H_P-I_{n+2}\|_{\mathrm{op}}
 +C_{\mathcal L}\delta.
\end{equation}

We claim that \eqref{eq:HQ-matrix-difference} holds with
$$
 C_{\rm tree}:=2C_{\mathcal L},
$$
and prove it by induction. 
At the root, $H_{Q_\ast}=\operatorname{id}$, hence
$\mathcal R(H_{Q_\ast})-I_{n+2}=0$, so the estimate is immediate. Now suppose that $d_T(Q)=m\ge1$ and that the estimate holds for the parent. Then \eqref{eq:difference-recursive-bound}, together with the induction hypothesis, gives
\begin{align*}
 \|\mathbf H_Q-I_{n+2}\|_{\mathrm{op}}
 &\le \frac A2 C_{\rm tree}\delta A^{m-1}+C_{\mathcal L}\delta
 =\frac{C_{\rm tree}}{2}\delta A^m+C_{\mathcal L}\delta  \\
 &\le\frac{C_{\rm tree}}{2}\delta A^m  +\frac{C_{\rm tree}}{2}\delta A^m
 =C_{\rm tree}\delta A^m,
\end{align*}
where we used $A^m\ge1$ and $C_{\rm tree}=2C_{\mathcal L}$.  This implies the proposition.
\end{proof}

We now change from the earlier neighboring cube coordinates to the fixed root coordinates.  Define
\begin{equation}\label{eq:RQ-ZQ-definitions}
 R_Q:=S_{Q_\ast}\circ S_Q^{-1},
 \qquad
 Z_Q:=R_Q\circ H_Q.
\end{equation}
Observe that, at the root one has $R_{Q_\ast}=Z_{Q_\ast}=\operatorname{id}_{\whR}$.

\begin{corollary}
Under the assumptions of \Cref{prop:root-relative-Lorentz}, there exist
$$
 G_0=G_0(n,p)>1,
 \qquad
 C_Z=C_Z(n,p)>0
$$
such that, for $m=d_T(Q)$,
\begin{equation}\label{eq:RQ-ZQ-size}
 \|\mathcal R(R_Q)\|_{\mathrm{op}}
 +\|\mathcal R(Z_Q)\|_{\mathrm{op}}
 \le2G_0^m
\end{equation}
and
\begin{equation}\label{eq:ZQ-RQ-difference}
 \|\mathcal R(Z_Q)-\mathcal R(R_Q)\|_{\mathrm{op}}
 \le C_Z\delta G_0^m.
\end{equation}
\end{corollary}

\begin{proof}
If $Q\ne Q_\ast$ has a parent $P$, then
$$
 R_Q
 =S_{Q_\ast}\circ S_Q^{-1}
 =S_{Q_\ast}\circ S_P^{-1}\circ S_P\circ S_Q^{-1}
 =R_P\circ \mathfrak S_Q.
$$
Hence \eqref{eq:CQ-Lorentz-bound} and the induction give
\begin{equation}\label{eq:RQ-size-alone}
 \|\mathcal R(R_Q)\|_{\mathrm{op}}
 \le C_{\rm par}^{d_T(Q)}.
\end{equation}
Since $Z_Q=R_Q\circ H_Q$, \eqref{eq:HQ-matrix-size} and \eqref{eq:RQ-size-alone} imply
\begin{equation}\label{eq:ZQ-size-alone}
 \|\mathcal R(Z_Q)\|_{\mathrm{op}}
 \le (AC_{\rm par})^{d_T(Q)}.
\end{equation}
Set
$$
 G_0:=AC_{\rm par}>1.
$$
Since $A>1$, combining both \eqref{eq:RQ-size-alone} and \eqref{eq:ZQ-size-alone} implies \eqref{eq:RQ-ZQ-size}.

Finally, as $Z_Q=R_Q\circ H_Q$,
\begin{align*}
 \mathcal R(Z_Q)-\mathcal R(R_Q)
 &=\mathcal R(R_Q)\bigl(\mathcal R(H_Q)-I_{n+2}\bigr).
\end{align*}
Applying \eqref{eq:RQ-size-alone} and \eqref{eq:HQ-matrix-difference}, we conclude 
$$
 \|\mathcal R(Z_Q)-\mathcal R(R_Q)\|_{\mathrm{op}}  \le C_{\rm tree}\delta (AC_{\rm par})^{d_T(Q)}.
$$
This is \eqref{eq:ZQ-RQ-difference} with $C_Z=C_{\rm tree}$.
\end{proof}

\subsection{The fixed-root estimate on one cube}

Recall that
$$
 \Qzero=\left(-\frac 1 2,\frac 1 2\right)^n.
$$
Thus $S_Q(Q)=\Qzero$ and $|\Qzero|=1$.
We next study the derivatives of $\sigma$, the inverse stereographic projection. 
\begin{lemma} \label{lem:sigma-derivatives}
There exists $C_\sigma=C_\sigma(n)>0$ such that
\begin{equation}\label{eq:sigma-C2}
 \|D\sigma\|_{L^\infty(\R^n)}+\|D^2\sigma\|_{L^\infty(\R^n)}
 \le C_\sigma.
\end{equation}
\end{lemma}

\begin{proof}
Write
$$
 d(x):=1+|x|^2.
$$
For $1\le i,j,k\le n$, the first $n$-components of $\sigma$ satisfy
$$
 \partial_j\sigma_i(x) =\frac{2\delta_{ij}}{d(x)}  -\frac{4x_ix_j}{d(x)^2},
$$
while the last component satisfies
$$
 \partial_j\sigma_{n+1}(x) =\frac{4x_j}{d(x)^2}.
$$
Differentiating once more, we have
$$
 \partial_k\partial_j\sigma_i(x)
 =  -\frac{4\delta_{ij}x_k}{d(x)^2} -\frac{4(\delta_{ik}x_j+\delta_{jk}x_i)}{d(x)^2} +\frac{16x_ix_jx_k}{d(x)^3}, 
 $$
and
 $$
 \partial_k\partial_j\sigma_{n+1}(x)
 = \frac{4\delta_{jk}}{d(x)^2} -\frac{16x_jx_k}{d(x)^3}.
$$

Observe that, for every integer $m \ge 0$ and every $s \ge m/2$, one has
$$
 \sup_{x\in\R^n}\frac{|x|^m}{(1+|x|^2)^s}<\infty.
$$
Consequently, each term appearing in the derivative formulas above is uniformly bounded in $x$. Consequently, each term appearing in the derivative formulas above is uniformly bounded in $x$. Passing from componentwise bounds to Euclidean operator norms introduces only a factor depending on the dimension. This proves \eqref{eq:sigma-C2}.
\end{proof}

Now for $Q\in\mathscr W$, we define the normalized local map
\begin{equation}\label{eq:hQ-definition}
 h_Q :=S_Q\circ\varphi_Q^{-1}\circ f\circ S_Q^{-1}.
\end{equation}
It is finite on $8\Qzero$.  Set
\begin{equation}\label{eq:uQ-v-definitions}
 \mathsf U_Q:=\sigma\circ h_Q,
 \qquad \mathsf V:=\sigma|_{\Qzero}.
\end{equation}

\begin{lemma} 
Assume \eqref{eq:small-dist}. Then
\begin{equation}\label{eq:uQ-v-uniform}
 \|\mathsf U_Q-\mathsf V\|_{L^\infty(\Qzero)} \le C(n,p)\delta,
\end{equation}
\begin{equation}\label{eq:uQ-v-derivative}
 \|D(\mathsf U_Q-\mathsf V)\|_{L^p(\Qzero)} \le C(n,p)\delta,
\end{equation}
and
\begin{equation}\label{eq:uQ-derivative-bound}
 \|D\mathsf U_Q\|_{L^p(\Qzero)} +\|D\mathsf V\|_{L^\infty(\Qzero)} \le C(n,p).
\end{equation}
\end{lemma}

\begin{proof}
Let $y\in\Qzero$ and $x=S_Q^{-1}(y)\in Q$.  By \eqref{eq:whitney-local-uniform},
\begin{equation}\label{eq:hQ-uniform-normalized}
 |h_Q(y)-y|
 =\frac{|g_Q(x)-x|}{\ell(Q)}
 \le C(n,p)\delta.
\end{equation}
Thus \eqref{eq:sigma-C2} and \eqref{eq:hQ-uniform-normalized} give \eqref{eq:uQ-v-uniform}.

Also, by the change of variables $x=S_Q^{-1}(y)$ and \eqref{eq:whitney-local-derivative}, we get
\begin{align}
 \int_{\Qzero}|Dh_Q-I_n|^p\,dy
 &=\ell(Q)^{-n}\int_Q|Dg_Q-I_n|^p\,dx\notag\\
 &\le\ell(Q)^{-n}\int_{8Q}|Dg_Q-I_n|^p\,dx \le C(n,p)\delta^p.
 \label{eq:hQ-derivative-normalized}
\end{align}
The  chain rule yields
\begin{equation}
\label{eq:sigma-composition-difference}
D\mathsf U_Q-D\mathsf V = \bigl[D\sigma(h_Q)-D\sigma(\iota_{\mathrm E})\bigr]Dh_Q +D\sigma(\iota_{\mathrm E})(Dh_Q-I_n),
\end{equation}
where $\iota_{\mathrm E}(y):=y$ is the Euclidean identity map on $\Qzero$ and the notation $\mathsf V=\sigma\circ \iota_{\mathrm E}$ in \eqref{eq:uQ-v-definitions} is used for the sphere-valued map.  By \eqref{eq:hQ-derivative-normalized} together with the triangle inequality
$$
 \|Dh_Q\|_{L^p(\Qzero)} \le |I_n||\Qzero|^{1/p}+C\delta  \le C(n,p).
$$
Applying the Lipschitz bound for $D\sigma$ from \eqref{eq:sigma-C2}, together with \eqref{eq:hQ-uniform-normalized} and \eqref{eq:hQ-derivative-normalized}, to \eqref{eq:sigma-composition-difference}, we obtain \eqref{eq:uQ-v-derivative}.  

Finally,
$$
 \|D\mathsf U_Q\|_{L^p} \le\|D(\mathsf U_Q-\mathsf V)\|_{L^p}+\|D\mathsf V\|_{L^p},
$$
and $D\mathsf V=D\sigma|_{\Qzero}$ is uniformly bounded by \eqref{eq:sigma-C2}.  This proves \eqref{eq:uQ-derivative-bound}.
\end{proof}

Recall the definition of $\Phi$ in \eqref{eq:weighted-root-normalizer}. 
We introduce the following globally defined sphere-valued mappings:
\begin{equation}\label{eq:F-F0-global-definitions}
 F := \widehat{S_{Q_\ast}\circ\Phi^{-1}}\circ\sigma\circ f,
 \qquad F_0 := \sigma\circ S_{Q_\ast},
\end{equation}
where the first expression is to be understood on $\whR$. 
Equivalently, at any point where the composition $S_{Q_\ast}\circ\Phi^{-1}\circ f$ is finite, we have
$$
F = \sigma\circ S_{Q_\ast}\circ\Phi^{-1}\circ f.
$$
Consequently, the map $F$ is well defined even at points for which $\Phi^{-1}(f(x))=\infty$. Furthermore, the map $\sigma$ has a uniformly bounded first derivative by \Cref{lem:sigma-derivatives}, and the fixed-sphere action $\widehat{S_{Q_\ast}\circ\Phi^{-1}}$ is smooth on the compact sphere. Thus, by the Sobolev chain rule we obtain
$$
 F\in W^{1,n}_{\mathrm{loc}}(\Omega;\R^{n+1}).
$$
In particular, $DF$ below denotes a globally defined weak derivative; the cubewise formulas describe its restriction almost everywhere on each Whitney cube.

\begin{lemma}\label{lem:local-global-representation}
Let $Q\in\mathscr W$, $x\in Q$, and $y=S_Q(x)\in\Qzero$.  Put
$$
 \mathbf Z_Q:=\mathcal R(Z_Q), \qquad \mathbf R_Q:=\mathcal R(R_Q),
$$
and define
$$
 A_Q:=\mathbf Z_Q^{\#},  \qquad B_Q:=\mathbf R_Q^{\#}.
$$
Then
\begin{equation}\label{eq:local-F-representation}
 F(x)=A_Q(\mathsf U_Q(y)), \qquad F_0(x)=B_Q(\mathsf V(y)).
\end{equation}
\end{lemma}

\begin{proof}
By \eqref{eq:RQ-ZQ-definitions}, \eqref{eq:HQ-definition}, and \eqref{eq:hQ-definition}, since $y=S_Q(x)$, 
\begin{align*}
 Z_Q\bigl(h_Q(y)\bigr)
 &=S_{Q_\ast}\circ S_Q^{-1} \circ S_Q\circ\Phi^{-1}\circ\varphi_Q\circ S_Q^{-1}
   \circ S_Q\circ\varphi_Q^{-1}\circ f\circ S_Q^{-1}(y) \\
 &=S_{Q_\ast}\circ\Phi^{-1}(f(x)).
\end{align*}
Then by recalling \eqref{eq:hat-sharp} in \Cref{prop:lorentz-realization}
$$
 \widehat{Z_Q}=(\mathcal R(Z_Q))^{\#} = (\mathbf Z_Q)^{\#}=A_Q,
$$
we get from \eqref{eq:F-F0-global-definitions} that
$$
F(x)=\sigma\circ Z_Q(h_Q(y))
=A_Q\bigl(\sigma(h_Q(y))\bigr)
=A_Q\bigl(\mathsf U_Q(y)\bigr),
$$
which gives the first identity in \eqref{eq:local-F-representation}.  Indeed, all identities before applying $\sigma$ are identities in  $\widehat{\R}^{\,n}$. Therefore they remain valid even if
$\Phi^{-1}(f(x))=\infty$.

For the second identity,
$$
 R_Q(y)=S_{Q_\ast}\circ S_Q^{-1}(y)=S_{Q_\ast}(x),
$$
and therefore
$$
 B_Q(\mathsf V(y)) =\widehat{R_Q}(\sigma(y)) =\sigma(S_{Q_\ast}(x)) =F_0(x).
$$
\end{proof}

\begin{proposition}\label{prop:fixed-root-cube-estimate}
There exist constants
$$
 G=G(n,p)>1, \qquad C_{\mathrm{cube}}=C_{\mathrm{cube}}(n,p)>0
$$
such that, whenever $0<\delta\le\delta_{\mathrm{tree}}$,
\begin{equation}\label{eq:fixed-root-cube-estimate}
 \fint_Q|F-F_0|^p\,dx +\ell(Q)^p\fint_Q|DF-DF_0|^p\,dx
 \le C_{\mathrm{cube}}(\delta G^{ d_T(Q)})^p
\end{equation}
for every $Q\in\mathscr W$.
\end{proposition}

\begin{proof}
Fix $Q$ and write $m=d_T(Q)$ for simplicity.  Let $A_Q,B_Q,\mathsf U_Q,\mathsf V$ be as in \Cref{lem:local-global-representation}.  Let $\widetilde A_Q$ denote the canonical ambient rational extension of the projective sphere map $A_Q=(\mathcal R(Z_Q))^\#$; see \eqref{eq:canonical-extension}.  

Since $A_Q=(\mathcal R(Z_Q))^\#$ and $B_Q=(\mathcal R(R_Q))^\#$, \eqref{eq:RQ-ZQ-size} and \Cref{lem:projective-derivatives} give
\begin{equation}\label{eq:AQ-BQ-C1}
 \|DA_Q\|_{L^\infty(\Sn)} +\|DB_Q\|_{L^\infty(\Sn)} \le C(n,p)G_0^{2m},
\end{equation}
and
\begin{equation}\label{eq:AQ-ambient-C2}
 \|D^2\widetilde A_Q\|_{L^\infty(\overline{\mathbb B}^{\,n+1})}  \le C(n,p)G_0^{6m}.
\end{equation}
Moreover, by \eqref{eq:RQ-ZQ-size}, \eqref{eq:ZQ-RQ-difference}, and \Cref{lem:projective-difference},
\begin{equation}\label{eq:AQ-BQ-C1-difference}
 \|A_Q-B_Q\|_{C^1(\Sn)}  \le C(n,p)G_0^{3m} \|\mathcal R(Z_Q)-\mathcal R(R_Q)\|_{\rm op}\le C(n,p)\delta G_0^{4m}.
\end{equation}

Set
$ G:=G_0^6>1.$
For $\xi,\eta\in\Sn$, let $d_{\Sn}(\xi,\eta)$ be their intrinsic spherical distance.  Since
$$
 d_{\Sn}(\xi,\eta) =2\arcsin\!\left(\frac{|\xi-\eta|}{2}\right) \le\frac{\pi}{2}|\xi-\eta|,
$$
integrating $DA_Q$ along a minimizing geodesic gives
\begin{equation}\label{eq:sphere-map-lipschitz-chord}
 |A_Q(\xi)-A_Q(\eta)|
 \le  \frac{\pi}{2}\|DA_Q\|_{L^\infty(\Sn)}|\xi-\eta|.
\end{equation}

For the zeroth-order term in \eqref{eq:fixed-root-cube-estimate}, \eqref{eq:local-F-representation} and \eqref{eq:sphere-map-lipschitz-chord} give
\begin{align*}
 |F(x)-F_0(x)|
 &\le |A_Q(\mathsf U_Q(y))-A_Q(\mathsf V(y))|
 +|A_Q(\mathsf V(y))-B_Q(\mathsf V(y))|\\
 &\le
 \frac{\pi}{2}\|DA_Q\|_{L^\infty(\Sn)}
 |\mathsf U_Q(y)-\mathsf V(y)|
 +\|A_Q-B_Q\|_{L^\infty(\Sn)}.
\end{align*}
Thus, after changing variables and applying \eqref{eq:uQ-v-uniform}, \eqref{eq:AQ-BQ-C1}, and \eqref{eq:AQ-BQ-C1-difference}, we conclude
\begin{equation}\label{eq:normalized-map-difference-final}
 (\fint_Q|F-F_0|^p\,dx)^{\frac 1 p}=\|A_Q\circ \mathsf U_Q-B_Q\circ\mathsf V\|_{L^p(\Qzero)}
 \le C(n,p)\delta[ G_0^{2m}+ G_0^{4m}]\le C(n,p)\delta G^m.
\end{equation}

For the derivative term in \eqref{eq:fixed-root-cube-estimate}, the identity
\begin{align}
 D(A_Q\circ \mathsf U_Q)-D(B_Q\circ\mathsf V)
 ={}&  [DA_Q(\mathsf U_Q)-DA_Q(\mathsf V)]D\mathsf U_Q +DA_Q(\mathsf V)[D\mathsf U_Q-D\mathsf V]\notag\\
 &\quad +[DA_Q(\mathsf V)-DB_Q(\mathsf V)]D\mathsf V
 \label{eq:three-term-composition-identity}
\end{align}
holds almost everywhere on $\Qzero$.  Since $\mathsf U_Q(y),\mathsf V(y)\in\Sn$, the line segment joining them is contained in $\overline{\mathbb B}^{\,n+1}$.  Thus, as $\widetilde A_Q$ is an extension of $A_Q$, by the fundamental theorem of calculus we conclude
$$
 |DA_Q(\mathsf U_Q(y))-DA_Q(\mathsf V(y))|
 \le \|D^2\widetilde A_Q\|_{L^\infty(\overline{\mathbb B}^{\,n+1})} |\mathsf U_Q(y)-\mathsf V(y)|.
$$
Now applying \eqref{eq:AQ-ambient-C2}, \eqref{eq:uQ-v-uniform}, and \eqref{eq:uQ-derivative-bound}, the first term on the right of \eqref{eq:three-term-composition-identity} satisfies
$$
 \|[DA_Q(\mathsf U_Q)-DA_Q(\mathsf V)]D\mathsf U_Q\|_{L^p(\Qzero)}\le C(n,p) \delta G_0^{6m} \le C(n,p)\delta G^m.
$$
Moreover, by \eqref{eq:AQ-BQ-C1} and \eqref{eq:uQ-v-derivative}, the second term has the same bound:
$$\|DA_Q(\mathsf V)[D\mathsf U_Q-D\mathsf V]\|_{L^p(\Qzero)}\le C(n,p) \delta G_0^{2m} \le C(n,p)\delta G^m.$$
Additionally, by \eqref{eq:AQ-BQ-C1-difference} and the $L^\infty$-bound for $D\mathsf V$ in \eqref{eq:uQ-derivative-bound}
$$\|[DA_Q(\mathsf V)-DB_Q(\mathsf V)]D\mathsf V\|_{L^p(\Qzero)} \le C(n,p) \delta G_0^{4m}\le C(n,p)\delta G^m.$$  
Hence, we conclude that
\begin{equation}\label{eq:normalized-derivative-difference-final}
 \|D(A_Q\circ \mathsf U_Q-B_Q\circ\mathsf V)\|_{L^p(\Qzero)}
 \le C(n,p)\delta G^m.
\end{equation}
Since the change of variable $x=S_Q^{-1}(y)$ gives
$$
 \ell(Q)^p\fint_Q|DF-DF_0|^p\,dx
 =\int_{\Qzero}|D(A_Q\circ \mathsf U_Q-B_Q\circ\mathsf V)|^p\,dy,
$$
equations \eqref{eq:normalized-map-difference-final} and \eqref{eq:normalized-derivative-difference-final} prove \eqref{eq:fixed-root-cube-estimate}.
\end{proof}

\subsection{The depth-damped diameter measure}

We use the following notation for the  weighted Sobolev class.  For a finite Borel measure  $\nu$   on $\Omega$ that is absolutely continuous with respect to Lebesgue measure, define
\begin{equation}
 \mathcal W^{1,p}(\Omega,\nu;\R^N)
 := \left\{ U\in W^{1,1}_{\mathrm{loc}}(\Omega;\R^N)\colon \int_\Omega\bigl(|U|^p+|DU|^p\bigr)\,d\nu<\infty \right\}.
\end{equation}
We adopt this explicit definition since the weight constructed below is not required to satisfy any Muckenhoupt-type condition, and, in addition, an appropriate density theorem for smooth functions may not be available. For a length scale $L>0$, set
\begin{equation}
 \|U\|_{\mathcal W^{1,p}_L(\Omega,\nu)}^p := \int_\Omega\bigl(|U|^p+L^p|DU|^p\bigr)\,d\nu.
\end{equation}

Recall $N_W=N_W(n)$ defined in \eqref{eq:defn-Dw} and $G=G(n,p)$ in \Cref{prop:fixed-root-cube-estimate}. 
Choose
\begin{equation}\label{eq:theta-choice}
 0<\theta=\theta(n,p)
 \le\min\left\{\frac12,\frac{1}{2N_W(n)G(n,p)^p}\right\}.
\end{equation}
For almost every $x\in Q$ with $Q\in\mathscr W$, we define
\begin{equation}\label{eq:weight-definition-proof}
 w(x) :=\theta^{d_T(Q)} \left(\frac{\ell(Q)}{\ell_\ast}\right)^p,  \qquad  d\mu(x):=w(x)\,dx.
\end{equation}
As the cube boundaries form a null set, so the definition is unambiguous almost everywhere. 

\begin{proposition}\label{prop:weight-properties}
The measure defined in \eqref{eq:weight-definition-proof} satisfies
\begin{equation}\label{eq:weight-between-zero-one}
 0<w(x)\le1
 \quad\text{for almost every }x\in\Omega,
\end{equation}
\begin{equation}\label{eq:measure-Lebesgue-dominated}
 \mu(E)\le|E|
 \quad\text{for every measurable }E\subset\Omega,
\end{equation}
and
\begin{equation}\label{eq:measure-total-bounds}
 \ell_\ast^n
 \le\mu(\Omega)
 \le2\ell_\ast^n.
\end{equation}
Moreover, $\mu$ has full support in $\Omega$, and for every $Q\in\mathscr W$,
\begin{equation}\label{eq:measure-cube-mass}
 \mu(Q)
 =\theta^{d_T(Q)}
  \frac{\ell(Q)^{n+p}}{\ell_\ast^p}.
\end{equation}
\end{proposition}

\begin{proof}
By \eqref{eq:bounded-side-bound}, $\ell(Q)\le\ell_\ast$.  Since $0<\theta<1$, then \eqref{eq:weight-definition-proof} gives \eqref{eq:weight-between-zero-one}, and \eqref{eq:measure-Lebesgue-dominated} follows directly by integration.

At the root, $d_T(Q_\ast)=0$ and $\ell(Q_\ast)=\ell_\ast$, so $w=1$ almost everywhere on $Q_\ast$.  Hence
$$
 \mu(\Omega)\ge\mu(Q_\ast)=|Q_\ast|=\ell_\ast^n.
$$
For the upper bound, we recall that $G>1$ and apply \eqref{eq:number-at-depth}, \eqref{eq:bounded-side-bound}, and \eqref{eq:theta-choice}  to conclude that, 
\begin{align*}
 \mu(\Omega)
 &=\sum_{Q\in\mathscr W}
   \theta^{d_T(Q)}
   \left(\frac{\ell(Q)}{\ell_\ast}\right)^p|Q|\\
 &\le\sum_{m=0}^{\infty}
   \sum_{d_T(Q)=m}\theta^m|Q|\\
 &\le\ell_\ast^n
   \sum_{m=0}^{\infty}(N_W\,\theta)^m\\
 &\le \ell_\ast^n
   \sum_{m=0}^{\infty} 2^{-m}\le 2\ell_\ast^n.
\end{align*}
This proves \eqref{eq:measure-total-bounds}.

The union of the Whitney cube boundaries has Lebesgue measure zero.  Thus, one gets $w(x)>0$ for almost every $x\in\Omega$. Hence $w>0$ almost everywhere in $\Omega$. If $U\subset\Omega$ is nonempty and relatively open, then
$|U|>0$. Since $w>0$ almost everywhere, 
$$\mu(U)=\int_U w\,dx>0.$$
Thus, $\mu$ has full support in $\Omega$.

Finally, the density is constant on the interior of each cube, and $|Q|=\ell(Q)^n$.  This gives \eqref{eq:measure-cube-mass}.
\end{proof}

\begin{proof}[Proof of \Cref{thm:weighted-main}]
Set
$$
\delta_\ast(n,p):=\delta_{\mathrm{tree}}(n,p),
$$
where $\delta_{\mathrm{tree}}$ is defined in
\Cref{prop:root-relative-Lorentz}, and choose $\theta$ as in
\eqref{eq:theta-choice}. Fix
$0<\delta=K-1\le\delta_\ast$, and let $\Phi$ be the root normalizer in \eqref{eq:weighted-root-normalizer}. Then both maps $F$ and $F_0$ are defined globally by \eqref{eq:F-F0-global-definitions}.

We show that the map $F$ belongs to $W^{1,p}_{\mathrm{loc}}(\Omega;\R^{n+1})$ for $p>n$.  Indeed, the global composition in \eqref{eq:F-F0-global-definitions} first gives $F\in W^{1,n}_{\mathrm{loc}}(\Omega;\R^{n+1})$. Moreover, observe that the smooth map $F_0$ belongs to $W^{1,p}_{\mathrm{loc}}(\Omega;\R^{n+1})$. On the other hand, on every Whitney cube,   \Cref{lem:local-global-representation} and \Cref{prop:fixed-root-cube-estimate} identify the weak derivative of $F$ and show that it is $p$-integrable there. As every compact subset of $\Omega$ meets only finitely many cubes of $\mathscr W$, we conclude that $DF\in L^p_{\mathrm{loc}}(\Omega)$, and the desired local Sobolev regularity follows.

Fix a cube $Q$ and set $m = d_T(Q)$. Since the density 
$$w=\theta^m\left(\frac{\ell(Q)}{\ell_\ast}\right)^p$$
is constant on $Q$, the zeroth-order term in \eqref{eq:fixed-root-cube-estimate} (with respect to the modified measure $\mu$) yields
\begin{align}
 \int_Q|F-F_0|^p\,d\mu
 &\le C_{\mathrm{cube}}\delta^p G^{pm}|Q| \theta^m  \left(\frac{\ell(Q)}{\ell_\ast}\right)^p\notag\\
 &\le C_{\mathrm{cube}}\delta^p (\theta G^p)^m|Q|.
 \label{eq:weighted-map-cube-sum}
\end{align}
For the derivative term, the scale in the weight cancels the local derivative scale as follows:
\begin{align}
 \ell_\ast^p\int_Q|DF-DF_0|^p\,d\mu
 &\le C_{\mathrm{cube}}\delta^p G^{pm}\ell_\ast^p\theta^m  \left(\frac{\ell(Q)}{\ell_\ast}\right)^p |Q|\ell(Q)^{-p}\notag\\
 &= C_{\mathrm{cube}}\delta^p  (\theta G^p)^m|Q|.
 \label{eq:weighted-derivative-cube-sum}
\end{align}
Summing \eqref{eq:weighted-map-cube-sum} and \eqref{eq:weighted-derivative-cube-sum} over $Q\in \mathscr W$, and using \eqref{eq:number-at-depth}, \eqref{eq:bounded-side-bound}, and \eqref{eq:theta-choice}, we obtain
\begin{align*}
 \int_\Omega  \left( |F-F_0|^p +\ell_\ast^p|DF-DF_0|^p\right)d\mu
 & \le C(n,p)\delta^p \sum_{m=0}^{\infty}
 \sum_{d_T(Q)=m} (\theta G^p)^m|Q|\\
 &\le C(n,p)\delta^p\ell_\ast^n \sum_{m=0}^{\infty}
 (N_W\, \theta G^p)^m\\
 &\le 2C(n,p)\delta^p\ell_\ast^n.
\end{align*}
By the lower bound in \eqref{eq:measure-total-bounds}, $\ell_\ast^n\le\mu(\Omega).$ Absorbing the factor $2$ into the constant proves the desired weighted stability estimate. Therefore, we conclude
$ F-F_0\in\mathcal W^{1,p}(\Omega,\mu;\R^{n+1})$
and
$$
\int_\Omega\left(|F-F_0|^p+\ell_\ast^p|DF-DF_0|^p
\right)d\mu\le C(n,p)\delta^p\mu(\Omega).
$$
All further statements of \Cref{thm:weighted-main} are already contained in \Cref{prop:weight-properties}.
\end{proof}

\appendix
\section{Proof of some lemmas}
\subsection{Proof of the classification of conformal Killing fields}
\begin{proof}[Proof of \Cref{prop:conformal-killing-fields}]
Assume first that \eqref{eq:ck-field} holds. The constant field $\mathbf a$ has zero derivative. Since $A^T=-A$, the field $x\mapsto Ax$ has zero symmetric gradient. The field $x\mapsto\lambda x$ satisfies
$$
\sym D(\lambda x)=\lambda\Id,
\qquad
\operatorname{div}(\lambda x)=n\lambda.
$$
For
$$
X_\mathbf{c}(x):=2(\mathbf{c}\cdot x)x-|x|^2\mathbf{c},
$$
we have
$$
\partial_j(X_\mathbf{c})_i=2\mathbf{c}_jx_i+2(\mathbf{c}\cdot x)\delta_{ij}-2x_j\mathbf{c}_i.
$$
Therefore
$$
\frac{\partial_j(X_\mathbf{c})_i+\partial_i(X_\mathbf{c})_j}{2} =2(\mathbf{c}\cdot x)\delta_{ij},
$$
while
$$
\operatorname{div}X_\mathbf{c} =2n(\mathbf{c}\cdot x).
$$
Hence $\Edev X_\mathbf{c}=0$. By the linearity of the operator, $\Edev u=0$.

Conversely, assume that $\Edev u=0$ in distributions. Write
$d:=\operatorname{div}u.$ Then the equation $\Edev u=0$ is equivalent to
\begin{equation}\label{eq:ck-system}
\partial_i u_j+\partial_j u_i
=\frac{2}{n}d\,\delta_{ij}
\quad\text{in }\mathscr D'(U)
\end{equation}
for $1\le i,j\le n$. 

Differentiate \eqref{eq:ck-system} with respect to $x_k$, add the resulting equation to the equation obtained by differentiating the pair $(i,k)$ with respect to $x_j$, and subtract the equation obtained by differentiating the pair $(j,k)$ with respect to $x_i$. Then the commutation of distributional derivatives gives
\begin{equation}\label{eq:second-u}
\partial_j\partial_k u_i = \frac1 n\left(\delta_{ij}\partial_kd+\delta_{ik}\partial_jd -\delta_{jk}\partial_id \right).
\end{equation}
Differentiate \eqref{eq:second-u} with respect to $x_i$ and sum over $i$. Since $d=\sum_i\partial_iu_i$, we obtain
\begin{align*}
\partial_j\partial_kd
&= \frac 1 n\sum_{i=1}^n \left( \delta_{ij}\partial_i\partial_kd +\delta_{ik}\partial_i\partial_jd -\delta_{jk}\partial_i\partial_id\right)\\
&= \frac 1 n \left(2\partial_j\partial_kd-\delta_{jk}\Delta d \right).
\end{align*}
Thus we arrive at
\begin{equation}\label{eq:hessian-d}
(n-2)\partial_j\partial_kd
=-\delta_{jk}\Delta d.
\end{equation}
If $j\ne k$, \eqref{eq:hessian-d} gives $\partial_j\partial_kd=0$. If $j=k$ and we sum over $j$, then
$$
(n-2)\Delta d=-n\Delta d.
$$
Since $n\ge3$, it follows that $\Delta d=0$. Returning to \eqref{eq:hessian-d}, we conclude that
$$
D^2d=0
\quad\text{in }\mathscr D'(U).
$$
By \Cref{lem:vanishing-hessian}\textup{(ii)}, there are $\alpha\in\R$ and $\beta\in\R^n$ such that
$$
d(x)=\alpha+\beta\cdot x.
$$

Now set
$$
\lambda:=\frac{\alpha}{n}, \qquad \mathbf{c}:=\frac{\beta}{2n},
\qquad
v(x):=u(x)-\lambda x-X_\mathbf{c}(x).
$$
Recall that $d=\operatorname{div} u$, and 
$$
\operatorname{div}\left(\lambda x + X_\mathbf{c}(x)\right)= n\lambda+2n(\mathbf{c}\cdot x)=\alpha+\beta\cdot x=d(x).
$$
Thus $\operatorname{div}v=0$.  Since $\partial_k d=\beta_k=2n \mathbf{c}_k$,  \eqref{eq:second-u} becomes
$$
 \partial_j\partial_k u_i
 =2\bigl(\delta_{ij}\mathbf{c}_k+\delta_{ik}\mathbf{c}_j-\delta_{jk}\mathbf{c}_i\bigr).
$$
On the other hand, a direct differentiation of
$$
 (X_\mathbf{c})_i=2(\mathbf{c}\cdot x)x_i-|x|^2\mathbf{c}_i
$$
gives
$$
 \partial_j\partial_k(X_\mathbf{c})_i
 =2\bigl(\delta_{ij}\mathbf{c}_k+\delta_{ik}\mathbf{c}_j-\delta_{jk}\mathbf{c}_i\bigr).
$$

Note that the field $x\mapsto\lambda x$ has zero Hessian.
Hence
$$
\partial_j\partial_k v_i=0 \qquad \text{ for any } \ 1\le i,j,k\le n,
$$
and therefore $D^2v=0$ in $\mathscr D'(U)$.
By \Cref{lem:vanishing-hessian}\textup{(iii)}, $v(x)=a+Bx$ for some $a\in\R^n$ and $B\in\R^{n\times n}$. Since $\Edev v=0$ and $\tr B=\operatorname{div}v=0$, we have
$\sym B=0.$
Thus $B^T=-B$, and \eqref{eq:ck-field} follows.
\end{proof}

\subsection{Proof of Whitney decomposition}
\begin{proof}[Proof of Lemma~\ref{lem:deep-whitney}]
Fix $x\in\Omega$. Since $\Omega$ is open, there exists $r_x>0$ such that $B(x,r_x)\subset\Omega$. Hence every sufficiently small dyadic cube $Q$ whose closure contains $x$ satisfies $64\overline Q\subset\Omega$.
Starting from one such cube and passing to its dyadic parents, the process must eventually stop. Indeed, choose $y\in\mathbb R^n\setminus\Omega$. If the sidelengths of the parents tend to infinity, then some sufficiently large parent has
$y\in 64\overline Q$, and so cannot be admissible. Consequently, there is a last cube in this chain for which the admissibility condition holds, and this cube is therefore maximal. From the arbitrariness of $x$, it follows that the closures of all maximal cubes form a covering of $\Omega$.

Moreover, two distinct maximal dyadic cubes cannot have overlapping interiors, since any two dyadic cubes with intersecting interiors must be nested. Then via the elementary inclusion $\overline{Q}\subset 2Q,$
we obtain formula \eqref{eq:doubled-whitney-cover}.

Let $Q\in\mathscr W$.  If $z\in Q$ and $y\in\R^n\setminus\Omega$, then \eqref{eq:deep-admissible} implies $y\notin64\overline Q$.  Applying the maximum norm $|x|_\infty:=\max_i|x_i|$, we obtain
$$
|y-z|\ge |y-z|_\infty \ge|y-x_Q|_\infty-|z-x_Q|_\infty
\ge 32\ell(Q)-\frac1 2 \ell(Q).
$$
Thus the lower bound in \eqref{eq:ordinary-whitney-distance} holds with $c_W=63/2$.

Let $Q^+$ be the dyadic parent of $Q$.  By the maximality,
$$
64\overline{Q^+}\not\subset\Omega.
$$
Choose $y\in64\overline{Q^+}\setminus\Omega$.  Since
$$
\ell(Q^+)=2\ell(Q),
\qquad
|x_{Q^+}-x_Q|_\infty=\frac12\ell(Q),
$$
we have
$$
|y-x_Q|
\le
\sqrt n\left(64+\frac12\right)\ell(Q).
$$
The segment joining $x_Q\in\Omega$ to $y\notin\Omega$ meets $\partial\Omega$.  Hence
$$
\dist(Q,\partial\Omega)
\le
\dist(x_Q,\partial\Omega)
\le
\sqrt n\left(64+\frac12\right)\ell(Q),
$$
which proves the upper bound in \eqref{eq:ordinary-whitney-distance}.

We also record local finiteness.  Let $E\subset\subset\Omega$, put
$$
 d_E:=\dist(E,\partial\Omega)>0,
$$
and fix $y_0\in\R^n\setminus\Omega$.  If $Q\in\mathscr W$ meets $E$, choose $x\in Q\cap E$.  Since $64\overline Q\subset\Omega$, one has $y_0\notin64\overline Q$, and therefore
$$
 |y_0-x|_\infty
 \ge |y_0-x_Q|_\infty-|x-x_Q|_\infty
 \ge\left(32-\frac12\right)\ell(Q).
$$
Thus $\ell(Q)$ is bounded above in terms of $E$ and $y_0$.  For the lower bound, choose $z\in\overline Q$ and $y\in\partial\Omega$ such that
$$
 |z-y|=\dist(\overline Q,\partial\Omega)=\dist(Q,\partial\Omega).
$$
Then
$$
 d_E
 \le\dist(x,\partial\Omega)
 \le |x-z|+|z-y|
 \le\bigl(\sqrt n+C_W\bigr)\ell(Q).
$$
Only finitely many dyadic generations can therefore occur among cubes meeting $E$.  At each fixed generation, only finitely many dyadic cubes meet the bounded set $E$.  Hence $\mathscr W$ is locally finite.

We prove \eqref{eq:adjacent-size-comparability}.  Suppose that $P\sim Q$ and, after interchanging the two cubes if necessary,
$$
L:=\ell(P)\ge4\ell(Q)=:4l.
$$
Let $Q^+$ be the parent of $Q$.  Since $2P$ and $2Q$ intersect,
\begin{equation}\label{eq:adjacent-center-distance}
|x_P-x_Q|_\infty<L+l.
\end{equation}
For every $y\in64\overline{Q^+}$,
\begin{align*}
|y-x_P|_\infty &\le|y-x_{Q^+}|_\infty +|x_{Q^+}-x_Q|_\infty
+|x_Q-x_P|_\infty\\
&\le 64l+\frac12l+L+l\\
&=L+\frac{131}{2}l \le \frac{139}{8}L <32L.
\end{align*}
Therefore
$$
64\overline{Q^+}\subset64\overline P\subset\Omega.
$$
This contradicts the maximality of $Q$.  Thus $\ell(P)<4\ell(Q)$.  Interchanging $P$ and $Q$ gives \eqref{eq:adjacent-size-comparability}.

We next prove the degree bound.  Fix $P\in\mathscr W$ and put $L=\ell(P)$.  If $Q\sim P$, then \eqref{eq:adjacent-size-comparability} gives
$$
\frac L4\le\ell(Q)\le4L,
$$
while \eqref{eq:adjacent-center-distance} gives
$$
|x_Q-x_P|_\infty<L+\ell(Q)\le5L.
$$
Consequently every such $Q$ is contained in the cube
$$
x_P+(-7L,7L)^n.
$$
The interiors of the neighboring cubes are disjoint and each has volume at least $(L/4)^n$.  Hence the number of neighbors is at most
$$
N_W(n):=56^n.
$$

For each connected component $\mathcal C$ of the adjacency graph, define
$$
U_{\mathcal C}:= \bigcup_{Q\in\mathcal C}\interior(2Q).
$$
The sets $U_{\mathcal C}$ are open.  Distinct graph components give disjoint sets, since an intersection would produce an edge.  Formula \eqref{eq:doubled-whitney-cover} shows that their union is $\Omega$.  Since $\Omega$ is connected, there is only one graph component.

Finally, let $P\sim Q$.  Choose
$$
a_{P,Q}\in\interior(2P)\cap\interior(2Q)
$$
and set $s=s_{P,Q}$.  If $x\in B(a_{P,Q},2s)$, then
$$
|x-x_P|_\infty \le |x-a_{P,Q}|+|a_{P,Q}-x_P|_\infty <2s+\ell(P) \le 3\ell(P),
$$
and
$$
|x-x_Q|_\infty <2s+\ell(Q)  \le 3\ell(Q).
$$
Since $a_{P,Q}\in\interior(2P)\cap\interior(2Q)$, the same estimates hold with non-strict inequalities for every $x\in\overline{B(a_{P,Q},2s)}$.  Since $8P$ and $8Q$ have respective half-sidelengths $4\ell(P)$ and $4\ell(Q)$, this proves
$$
\overline{B(a_{P,Q},2s)}\subset8P\cap8Q.
$$
Together with \eqref{eq:bridge-radius}, this is \eqref{eq:bridge-ball}.
\end{proof}

\subsection{Proof of the Boman chain construction in John domains}

\begin{proof}[Proof of \Cref{lem:john-whitney-chains}]
Fix $Q\in\mathscr W$.  If $Q=Q_\ast$, simply take the chain of length zero. Then it satisfies the three conclusions once $\Lambda_J\ge 2$ and $C_{\mathrm{suf}}\ge 1$.  Thus we may assume henceforth that $Q\ne Q_\ast$, and let $x_Q$ be its center.  

Choose an arclength-parametrized John curve
$$
 \gamma:[0,L]\to\Omega, \qquad  \gamma(0)=x_Q,
 \qquad  \gamma(L)=x_\ast,
$$
satisfying
\begin{equation}\label{eq:john-curve-inequality}
 t\le J\,\dist(\gamma(t),\partial\Omega) \qquad \text{ for any } \ 0\le t\le L.
\end{equation}

We first construct a finite sequence of Whitney cubes that follows the curve $\gamma$ in neighboring order. The sets
$$
 I_R:=\gamma^{-1}(2R),
 \qquad R\in\mathscr W,
$$
are relatively open in $[0,L]$ and cover that interval by \eqref{eq:doubled-whitney-cover}.  Since $x_Q\in Q\subset2Q$ and $x_\ast\in2Q_\ast$, there are numbers $a_0,a_\ast>0$ such that
$$
 [0,a_0]\subset I_Q, \qquad  [L-a_\ast,L]\subset I_{Q_\ast}.
$$
Choose a finite subcover of $[0,L]$ by sets of the form $I_R$, making sure that $I_Q$ and $I_{Q_\ast}$ are included.  Let $\lambda>0$ be a Lebesgue number for this finite cover.  Choose a partition
$$
 0=s_0<s_1<\cdots<s_M=L
$$
whose mesh is smaller than $\lambda$.  For each $0\le i<M$, choose a cube $T_i\in\mathscr W$ such that
$$
 \gamma([s_i,s_{i+1}])\subset2T_i.
$$
The finite sequence
$$
 Q,T_0,T_1,\ldots,T_{M-1},Q_\ast
$$
constitutes the required collection of Whitney cubes, with the property that the first two dilated cubes both contain $\gamma(0)$, the last two dilated cubes both contain $\gamma(L)$, and, for each $i$, the doubled cubes $2T_i$ and $2T_{i+1}$ both contain the point $\gamma(s_{i+1})$.  Up to deleting consecutive repetitions, we relabel this sequence as
\begin{equation}
 R_0=Q,R_1,\ldots,R_N=Q_\ast.
\end{equation}
We choose encounter times in neighboring order as follows.  Assign time $0$ to the initial cube $Q$, assign the left endpoint $s_i$ to $T_i$, and assign time $L$ to the final cube $Q_\ast$.  When consecutive repetitions are deleted, keep the earliest assigned time.  This gives
$$
 0\le t_0\le t_1\le\cdots\le t_N\le L.
$$
They satisfy
\begin{equation}\label{eq:raw-walk-encounter}
 \gamma(t_i)\in2R_i
 \qquad \text{ for any } \ 0\le i\le N
\end{equation}
and
\begin{equation}\label{eq:raw-walk-adjacency}
 R_i\sim R_{i+1}
 \qquad \text{ for any } \  0\le i<N.
\end{equation}

The sequence may contain repetitions, and we need to remove loops without destroying the neighboring order. Towards this, set
$$
 i_0:=\max\{i:R_i=R_0\}, \qquad  S_0:=R_0.
$$
If $i_k<N$, define
$$
 S_{k+1}:=R_{i_k+1}, \qquad  i_{k+1}:=\max\{i\ge i_k+1:R_i=S_{k+1}\}.
$$
The set defining $i_{k+1}$ is nonempty, since it contains the
index $i_k+1$. Moreover,
$$
i_{k+1}\ge i_k+1>i_k.
$$
Thus the indices $i_k$ strictly increase. Since they all belong to
the finite set $\{0,\ldots,N\}$, the construction terminates after
finitely many steps, say at $i_m$. Necessarily $i_m=N$; otherwise, if $i_m<N$, our construction would define extra $S_{m+1}$ and $i_{m+1}$. Consequently,
$$
S_m=R_{i_m}=R_N=Q_\ast.
$$

The selected cubes are pairwise distinct. Indeed, suppose that
$S_\ell=S_k$ for some $\ell>k$. Then $i_\ell>i_k$ and
$$
R_{i_\ell}=S_\ell=S_k.
$$
If $k=0$, this contradicts the definition of $i_0$ as the last
occurrence of $S_0=R_0$ in the original sequence. If $k\ge1$, then
$i_\ell\ge i_k+1\ge i_{k-1}+1$, contradicting the definition of
$i_k$ as the largest index $i\ge i_{k-1}+1$ for which
$R_i=S_k$.

Finally,
$$
S_k=R_{i_k},
\qquad
S_{k+1}=R_{i_k+1},
$$
and hence \eqref{eq:raw-walk-adjacency} gives
$S_k\sim S_{k+1}.$
Since the original encounter times $t_i$ form a nondecreasing sequence and the indices $i_k$ are strictly increasing, the corresponding new encounter times $\tau_k := t_{i_k}$ also form a nondecreasing sequence.  Consider the sequence of cubes $S_0,\ldots,S_m$ connecting the terminal cube $Q$ to the root cube $Q_\ast$. Define
$$
Q_j := S_{m-j},\qquad 0\le j\le m.
$$
By construction, $Q_0=Q_\ast$, $Q_m=Q$, and consecutive cubes in
the sequence $Q_0,\ldots,Q_m$ are adjacent. Furthermore, for every $0\le j\le m$, letting $k:=m-j$, one has the exact identity
\begin{equation}\label{eq:exact-cubes}
 \{Q_j,\ldots,Q_m\}
= \{S_k,S_{k-1},\ldots,S_0\}
= \{S_0,\ldots,S_k\}.     
\end{equation}
Since the encounter times are nondecreasing,
\begin{equation}\label{eq:increasing-tau}
    \tau_h\le\tau_k \qquad\text{whenever }0\le h\le k. 
\end{equation}

We next record the boundary-distance estimates used repeatedly below.  If $R\in\mathscr W$ and $z\in2R$, then
\begin{equation}\label{eq:doubled-cube-boundary-comparison}
 31\ell(R) \le\dist(z,\partial\Omega)
 \le C_2(n)\ell(R),
\end{equation}
where
$$
 C_2(n):=C_W+\frac{3\sqrt n}{2}.
$$
Indeed, for the lower bound, let $y\in\partial\Omega$.  Since $64\overline R\subset\Omega$, one has $y\notin64\overline R$ and therefore
$$
 |y-x_R|_\infty\ge32\ell(R).
$$
On the other hand, $z\in2R$ implies $|z-x_R|_\infty<\ell(R)$.  Hence
$$
 |y-z|\ge |y-z|_\infty\ge31\ell(R).
$$
Taking the infimum over $y\in\partial\Omega$ proves the lower bound.  For the upper bound, the Euclidean distance from $z\in2R$ to $R$ is at most $\sqrt n\,\ell(R)/2$. Thus \eqref{eq:ordinary-whitney-distance}, together with the triangle inequality, gives the stated estimate. 

Let $R=S_k$ be any retained cube and let $\tau_R=\tau_k$ be its encounter time.  By \eqref{eq:john-curve-inequality}, \eqref{eq:raw-walk-encounter}, and the upper estimate in \eqref{eq:doubled-cube-boundary-comparison},
\begin{equation}\label{eq:encounter-time-bound}
 \tau_R
 \le J\,\dist(\gamma(\tau_R),\partial\Omega)
 \le JC_2(n)\ell(R).
\end{equation}
Since $\gamma(0)=x_Q$ and $\gamma$ is parametrized by arclength,
\begin{equation} \label{eq:terminal-center-to-chain-center}
 |x_Q-x_R|\le |x_Q-\gamma(\tau_R)|+|\gamma(\tau_R)-x_R| \le \tau_R+\sqrt n\ell(R) \le C(n,J)\ell(R).
\end{equation}
Also,
$$
 c_W\ell(Q)\le\dist(x_Q,\partial\Omega) \le\dist(\gamma(\tau_R),\partial\Omega)+\tau_R \le (1+J)C_2(n)\ell(R).
$$
Hence
\begin{equation}\label{eq:terminal-scale-controlled}
 \ell(Q)\le C(n,J)\ell(R).
\end{equation}

Let $C_{\rm cen}=C_{\rm cen}(n,J)$ and
$C_{\rm sc}=C_{\rm sc}(n,J)$ be constants such that \eqref{eq:terminal-center-to-chain-center} and
\eqref{eq:terminal-scale-controlled} read
$$
|x_Q-x_R|\le C_{\rm cen}\ell(R),
\qquad \text{ and } \qquad
\ell(Q)\le C_{\rm sc}\ell(R).
$$
If $y\in2Q$, then $|y-x_Q|_\infty<\ell(Q). $
Consequently,
\begin{align*}
|y-x_R|_\infty
&\le |y-x_Q|_\infty+|x_Q-x_R|_\infty \\
&\le |y-x_Q|_\infty+|x_Q-x_R| \\
&< \ell(Q)+C_{\rm cen}\ell(R) \\
&\le (C_{\rm sc}+C_{\rm cen})\ell(R).
\end{align*}
Choose, for example,
$$
\Lambda_J:= 2\bigl(C_{\rm sc}+C_{\rm cen}+1\bigr).
$$
Since the half-sidelength of $\Lambda_J R$ is
$\Lambda_J\ell(R)/2$, the estimate above implies
$2Q\subset\Lambda_JR.$
This proves \eqref{eq:john-terminal-shadow}.

It remains to prove the  estimate \eqref{eq:john-suffix-sum}.  Fix $k\in\{0,\ldots,m\}$, set
$$
R:=S_k, \qquad \tau_R:=\tau_k,
$$
and define the index suffix
$$
\mathcal S_R:=\{S_0,\ldots,S_k\}. 
$$
By \eqref{eq:exact-cubes}, this family consists exactly of the cubes in the corresponding terminal segment of the reversed chain. Moreover, by \eqref{eq:increasing-tau}, every $S=S_h\in\mathcal S_R$ satisfies
$$
    \tau_S:=\tau_h\le\tau_k=\tau_R. 
$$
First, using \eqref{eq:encounter-time-bound} for $R$ and the $1$-Lipschitz property of the boundary distance,
\begin{align*}
 \dist(\gamma(\tau_S),\partial\Omega) \le\dist(\gamma(\tau_R),\partial\Omega) +|\tau_R-\tau_S| \le C_2(n)\ell(R)+\tau_R \le(1+J)C_2(n)\ell(R).
\end{align*}
The lower estimate in \eqref{eq:doubled-cube-boundary-comparison} therefore gives
\begin{equation}\label{eq:suffix-largest-scale}
 \ell(S)\le C(n,J)\ell(R).
\end{equation}

Since $\mathcal S_R$ is finite and nonempty, define
$$
L_R:=\max_{S\in\mathcal S_R}\ell(S).
$$
By \eqref{eq:suffix-largest-scale}, 
\begin{equation}\label{eq:C3}
    L_R\le C_3\ell(R)
\end{equation}
for some constant $C_3=C_3(n,J)$.
Recall that all cubes belong to the same dyadic grid. Hence every sidelength
occurring in $\mathcal S_R$ is of the form $2^{-h}L_R$ for some $h\in\mathbb N_0$.
For $h\in\mathbb N_0$, set
$$
\mathcal S_R^{(h)}
:= \{S\in\mathcal S_R:\ell(S)=2^{-h}L_R\}.  
$$
These families are pairwise disjoint and their union is
$\mathcal S_R$.

Fix $h\ge0$, let $r_h:=2^{-h}L_R,$
and let $S\in\mathcal S_R^{(h)}$. By  \eqref{eq:john-curve-inequality}, \eqref{eq:raw-walk-encounter}, and the upper bound in \eqref{eq:doubled-cube-boundary-comparison},
$$
\tau_S \le J\operatorname{dist}(\gamma(\tau_S),\partial\Omega)
\le JC_2(n)\ell(S) =JC_2(n)r_h.
$$
Since $\gamma(0)=x_Q$ and $\gamma$ is parametrized by arclength,
$$
|\gamma(\tau_S)-x_Q|\le\tau_S.
$$
Furthermore, since $\gamma(\tau_S)\in 2S$, it follows that
$$
|x_S-\gamma(\tau_S)| \le \sqrt{n}\,\ell(S)
= \sqrt{n}\,r_h.
$$
Combining these estimates, we obtain
$$
|x_S - x_Q|\le |x_S - \gamma(\tau_S)|+ |\gamma(\tau_S) - x_Q| \le C(n,J)\,r_h.
$$

If $y\in S$, then
$$
|y-x_Q|\le |y-x_S|+|x_S-x_Q|
\le \frac{\sqrt n}{2}r_h+C(n,J)r_h.
$$
Thus every cube in $\mathcal S_R^{(h)}$ is contained in a ball
$B(x_Q,C(n,J)r_h).$
The interiors of these cubes are pairwise disjoint, and every such cube has volume $r_h^n$. Hence
$$
\#\mathcal S_R^{(h)}\,r_h^n
\le |B(x_Q,C(n,J)r_h)| \le C(n,J)r_h^n.
$$
Consequently,
$$
\#\mathcal S_R^{(h)} \le N_J(n,J)  $$
for every $h\ge0$.
We now sum explicitly over the dyadic generations and apply \eqref{eq:C3} to get
\begin{align}
\sum_{S\in\mathcal S_R}\ell(S)
&=\sum_{h=0}^{\infty}\sum_{S\in\mathcal S_R^{(h)}}\ell(S)=\sum_{h=0}^{\infty}\#\mathcal S_R^{(h)}\,2^{-h}L_R \notag \\
&\le N_J(n,J)L_R\sum_{h=0}^{\infty}2^{-h}\notag \\
&=2N_J(n,J)L_R \le 2N_J(n,J)C_3(n,J)\ell(R).\label{eq:dyadic-generations}
\end{align}  
Finally, let $0\le j\le m$ and put $k=m-j$. Then
$Q_j=S_k$, and by \eqref{eq:exact-cubes},
$$
\{Q_j,\ldots,Q_m\} =\mathcal S_{S_k}.
$$
Applying \eqref{eq:dyadic-generations} with $R=S_k=Q_j$, we obtain
$$
\sum_{a=j}^{m}\ell(Q_a) =\sum_{S\in\mathcal S_{S_k}}\ell(S)
\le C_{\rm suf}(n,J)\ell(S_k) = C_{\rm suf}(n,J)\ell(Q_j),
$$
where
$$
C_{\rm suf}(n,J) := 2N_J(n,J)C_3(n,J).
$$
This is  precisely \eqref{eq:john-suffix-sum}.
\end{proof}


\begin{thebibliography}{99}

\bibitem{A1974}
L. V. Ahlfors, \emph{Conditions for Quasiconformal Deformations in Several Variables}. Contributions to Analysis: A Collection of Papers Dedicated to Lipman Bers,  Academic Press, 1974, pp. 19–25.

\bibitem{A1976}
L. V. Ahlfors, \emph{Quasiconformal deformations and mappings in $\mathbb R^n$}. J. Anal. Math. \textbf{30} (1976) 74--97.


\bibitem{B1988}
B.~Bojarski,
\emph{Remarks on Sobolev imbedding inequalities},
in: Complex Analysis, Joensuu 1987,
Lecture Notes in Math. \textbf{1351}, Springer, Berlin, 1988, pp.~52--68.

\bibitem{D2006}
S.~Dain,
\emph{Generalized Korn's inequality and conformal Killing vectors},
Calc. Var. Partial Differential Equations \textbf{25} (2006), 535--540.

\bibitem{DL1976}
G. Duvaut, J.-L. Lions, \emph{Inequalities in mechanics and physics}. Translated from the French by C. W. John. Grundlehren der Mathematischen Wissenschaften, Vol. 219. Springer-Verlag, Berlin-New York, 1976.

\bibitem{EG1992}
L. C. Evans, R. F. Gariepy,  \emph{Measure theory and fine properties of functions.} Studies in Advanced Mathematics. CRC Press, Boca Raton, FL, 1992. 

\bibitem{FZ2005}
D. Faraco and X. Zhong, \emph{Geometric rigidity of conformal matrices}, Ann. Sc. Norm. Super. Pisa Cl. Sci. (5)  \textbf{4} (2005), no. 4, 557--585.

\bibitem{FZ2022}
A.~Figalli and Y.~R.-Y. Zhang,
\emph{Sharp gradient stability for the Sobolev inequality},
Duke Math. J. \textbf{171} (2022), 2407--2459.

\bibitem{GMP2017}
F.~W. Gehring, G.~J. Martin, and B.~P. Palka,
\emph{An Introduction to the Theory of Higher-Dimensional Quasiconformal Mappings},
Math. Surveys Monogr. \textbf{216}, Amer. Math. Soc., Providence, RI, 2017.

\bibitem{G2004}
L. Grafakos,  \emph{Classical and modern Fourier analysis}. Pearson Education, Inc., Upper Saddle River, NJ, 2004. 

\bibitem{GLZ2025}
A. Guerra, X. Lamy,  K. Zemas, 
\emph{Sharp quantitative stability of the M\"obius group among sphere-valued maps in arbitrary dimension}.  
Trans. Amer. Math. Soc. \textbf{378} (2025), no. 2, 1235--1259.

\bibitem{GLZ2024}
A. Guerra, X. Lamy,  K. Zemas, 
\emph{Optimal quantitative stability of the M\"obius group of the sphere in all dimensions}, 	arXiv:2401.06593.

\bibitem{H2008}
N. J. Higham,  \emph{Functions of matrices.
Theory and computation}. Society for Industrial and Applied Mathematics (SIAM), Philadelphia, PA, 2008.

\bibitem{HK2000}
P. Haj{\l}asz, P. Koskela, 
\emph{Sobolev met Poincar\'e.}
Mem. Amer. Math. Soc. 145 (2000), no. 688, x+101 pp.

\bibitem{KR1997}
P. Koskela, S. Rohde, \emph{Hausdorff dimension and mean porosity},  Math. Ann. \textbf{309} (1997), 593--609.

\bibitem{LN2021}
P.~Lewintan and P.~Neff,
\emph{$L^p$-trace-free version of the generalized Korn inequality for incompatible tensor fields in arbitrary dimensions},
Z. Angew. Math. Phys. \textbf{72} (2021), Article 127.

\bibitem{L2013}
Z. Liu, \emph{Another proof of the Liouville theorem}. Ann. Acad. Sci. Fenn. Math. \textbf{38} (2013), no. 1, 327--340.

\bibitem{LZ2022}
S. Luckhaus, K. Zemas, \emph{Rigidity estimates for isometric and conformal maps from $\mathbb S^{n-1}$ to $\mathbb R^n$}. Invent. math. 230 (2022), no. 1, 375--461.

\bibitem{N2012}
J.~Ne\v{c}as,
\emph{Direct methods in the theory of elliptic equations},
Springer Monogr. Math., Springer, Heidelberg, 2012.

\bibitem{R2005}
K. Rajala, 
\emph{The local homeomorphism property of spatial quasiregular mappings with distortion close to one. }
Geom. Funct. Anal. 15 (2005), no. 5, 1100--1127.


\bibitem{R1970}
Yu.~G. Reshetnyak,
\emph{Estimates for certain differential operators with finite-dimensional kernel},
Siberian Math. J. \textbf{11} (1970), 315--326.


\bibitem{R1976}
Yu.~G. Reshetnyak,
\emph{Stability estimates in Liouville's theorem, and the $L_p$-integrability of the derivatives of quasiconformal mappings.}
Siberian Math. J. \textbf{17} (1976), no. 4, 653--674.

\bibitem{R19762}
Yu.~G. Reshetnyak,
\emph{Stability estimates in the class $W^1_p$ in Liouville's conformal mapping theorem for a closed domain.}
Siberian Math. J.  \textbf{17} (1976), no. 6, 1009--1018.

\bibitem{R1989}
Yu.~G. Reshetnyak,
\emph{Space Mappings with Bounded Distortion},
Transl. Math. Monogr. \textbf{73}, Amer. Math. Soc., Providence, RI, 1989.


\bibitem{R1994}
Yu.~G. Reshetnyak,
\emph{Stability Theorems in Geometry and Analysis},
Mathematics and Its Applications \textbf{304}, Kluwer Academic Publishers, Dordrecht, 1994.

\bibitem{R1993}
S.~Rickman,
\emph{Quasiregular Mappings},
Ergeb. Math. Grenzgeb. (3), vol.~26, Springer-Verlag, Berlin, 1993.

\bibitem{S1976}
V. I. Semenov, 
\emph{One-parameter groups of quasiconformal homeomorphisms in a Euclidean space}. 
Sibirsk. Mat. Z. \textbf{17} (1976), no. 1, 177--193, 240.

\bibitem{S1987}
V. I. Semenov,
\emph{Stability estimates for spatial quasiconformal mappings of a star domain.}
Sibirsk. Mat. Z. \textbf{28} (1987), no. 6, 102--118, 219.


\bibitem{S1988}
T.V. Sokolova, \emph{Structure of maps with bounded distortion which are nearly conformal}. Sibirsk. Mat. Z. \textbf{29} (1988), 508--509.

\end{thebibliography}
\end{document}